\documentclass[twoside,a4paper,reqno,11pt]{amsart} 
\usepackage[top=28mm,right=30mm,bottom=28mm,left=30mm]{geometry}

\usepackage{amsfonts, amsmath, amssymb, mathrsfs, bm, latexsym, stmaryrd, array, hyperref, mathtools, bold-extra, xcolor}

\renewcommand{\a}{\alpha}

\newcommand{\e}{\varepsilon}

\newcommand{\normeq}{\trianglelefteqslant}

\newcommand{\la}{\langle}
\newcommand{\ra}{\rangle}

\renewcommand{\to}{\rightarrow}

\newcommand{\leqs}{\leqslant}
\newcommand{\geqs}{\geqslant}

\newcommand{\vs}{\vspace{2mm}}
\newcommand{\fpr}{\mbox{{\rm fpr}}}

\newcommand{\Aut}{\mathrm{Aut}}
\newcommand{\Inn}{\mathrm{Inn}}
\newcommand{\Out}{\mathrm{Out}}

\newcommand{\GL}{\mathrm{GL}}
\newcommand{\PGL}{\mathrm{PGL}}
\newcommand{\SL}{\mathrm{SL}}

\newcommand{\PGammaL}{\mathrm{P\Gamma L}}

\newcommand{\PSp}{\mathrm{PSp}}
\newcommand{\Sp}{\mathrm{Sp}}

\newcommand{\POmega}{\mathrm{P\Omega}}
\newcommand{\LL}{\mathrm{L}}
\newcommand{\UU}{\mathrm{U}}

\newcommand{\PGU}{\mathrm{PGU}}

\newcommand{\GU}{\mathrm{GU}}

\newcommand{\fix}{\mathrm{fix}}
\newcommand{\Hol}{\mathrm{Hol}}

\makeatletter
\newcommand{\imod}[1]{\allowbreak\mkern4mu({\operator@font mod}\,\,#1)}
\makeatother

\newtheorem{theorem}{Theorem} 
\newtheorem*{conj*}{Conjecture}

\newtheorem{corol}[theorem]{Corollary}

\newtheorem{thm}{Theorem}[section] 
\newtheorem{prop}[thm]{Proposition} 
\newtheorem{lem}[thm]{Lemma}
\newtheorem{cor}[thm]{Corollary}

\newtheorem*{prob*}{Problem}

\theoremstyle{definition}
\newtheorem{rem}[thm]{Remark}

\begin{document}

\title[Base sizes of primitive groups of diagonal type]{Base sizes of primitive groups of diagonal type} 
\author{Hong Yi Huang}
\address{H.Y. Huang, School of Mathematics, University of Bristol, Bristol BS8 1UG, UK}
\email{hy.huang@bristol.ac.uk}
\date{\today}

\maketitle

\begin{abstract}
	Let $G$ be a permutation group on a finite set $\Omega$. The base size of $G$ is the minimal size of a subset of $\Omega$ with trivial pointwise stabiliser in $G$. In this paper, we extend earlier work of Fawcett by determining the precise base size of every finite primitive permutation group of diagonal type. In particular, this is the first family of primitive groups arising in the O'Nan-Scott theorem for which the exact base size has been computed in all cases. Our methods also allow us to determine all the primitive groups of diagonal type with a unique regular suborbit.
\end{abstract}

\section{Introduction}

\label{s:intro}

Let $G\leqs\mathrm{Sym}(\Omega)$ be a permutation group on a finite set $\Omega$ of size $n$. A subset of $\Omega$ is called a \textit{base} for $G$ if its pointwise stabiliser in $G$ is trivial. The minimal size of a base, denoted $b(G)$, is called the \textit{base size} of $G$. Equivalently, if $G$ is transitive with point stabiliser $H$, then $b(G)$ is the smallest number $b$ such that the intersection of some $b$ conjugates of $H$ in $G$ is trivial. This classical concept has been studied since the early years of permutation group theory in the nineteenth century, finding natural connections to other areas of algebra and combinatorics. For example, see \cite{BC_base} for details of the relationship between the metric dimension of a finite graph and the base size of its automorphism group, and \cite[Section 4]{S_algorithms} for an account of the key role played by bases in the computational study of finite groups. We refer the reader to survey articles \cite[Section 5]{Bur181} and \cite{LSh3} for further connections.

In general, determining $b(G)$ is a difficult problem and there are no efficient algorithms for computing $b(G)$, or constructing a base of minimal size. Blaha \cite{B_NP} proves that determining whether $G$ has a base of size a given constant is an NP-complete problem. Historically, there has been an intense focus on studying the base sizes of finite primitive groups (recall that a transitive permutation group is \textit{primitive} if its point stabiliser is a maximal subgroup). A trivial lower bound is $b(G)\geqs \log_n|G|$ and it turns out that all primitive groups admit small bases in the sense that there is an absolute constant $c$ such that $b(G)\leqs c\log_n|G|$ for every primitive group $G$. This was originally conjectured by Pyber \cite{P_Pyber} in the 1990s and the proof was completed by Duyan et al. in \cite{DHM_Pyber}. It was subsequently extended by Halasi et al. \cite{HLM_Pyber}, who show that
\begin{equation*}
b(G)\leqs 2\log_n|G|+24
\end{equation*}
and the multiplicative constant $2$ is best possible. In fact, one can prove stronger bounds in special cases. For example, Seress \cite{S_sol} proves that $b(G)\leqslant 4$ if $G$ is soluble, and this result was recently extended by Burness \cite{B_sol} who shows that $b(G)\leqslant 5$ if $G$ has a soluble point stabiliser (both bounds in \cite{B_sol} and \cite{S_sol} are best possible).

The O'Nan-Scott theorem divides the finite primitive groups into several families that are defined in terms of the structure and action of the socle of the group (recall that the \textit{socle} of a group is the product of its minimal normal subgroups). Following \cite{LPS}, these families are: affine, almost simple, diagonal type, product type, and twisted wreath products. There are partial results on base sizes when $G$ is affine, product type or a twisted wreath product. For example, if $G = VH\leqs \mathrm{AGL}(V)$ is affine, then Halasi and Podoski \cite{HP_coprime} show that $b(G)\leqs 3$ if $(|V|,|H|) = 1$, and we refer the reader to \cite{BH_prod,F_tw} for some results on base sizes of product type groups and twisted wreath products. In recent years, base sizes of almost simple primitive groups have been intensively studied (recall that $G$ is called \textit{almost simple} if there exists a non-abelian simple group $G_0$ such that $G_0\normeq G\leqs\Aut(G_0)$). Roughly speaking, such a group is said to be \textit{standard} if $G_0 = A_m$ and $\Omega$ is a set of subsets or partitions of $\{1,\dots,m\}$, or $G_0$ is a classical group and $\Omega$ is a set of subspaces of the natural module for $G_0$, otherwise $G$ is \textit{non-standard} (see \cite[Definition 1]{B_base_classical} for the formal definition). A conjecture of Cameron \cite[p. 122]{C_perm} asserts that $b(G)\leqs 7$ if $G$ is non-standard, with equality if and only if $G = \mathrm{M}_{24}$ in its natural action of degree $24$. This conjecture was proved in a sequence of papers by Burness et al. \cite{B_base_classical,BGS_base_sym,BLS_Cameron,BOW_base_sporadic}. In addition, the precise base sizes of all non-standard groups with alternating or sporadic socle are computed in \cite{BGS_base_sym} and \cite{BOW_base_sporadic,NNOW_B}, respectively.

In this paper, we focus on bases for primitive diagonal type groups. Here $G \leqs \mathrm{Sym}(\Omega)$ has socle $T^k$, where $T$ is a non-abelian simple group and $k \geqs 2$ is an integer. More precisely, we have $|\Omega| = |T|^{k-1}$ and
\begin{equation*}
T^k\normeq G\leqslant T^k.(\Out(T)\times S_k).
\end{equation*}
The primitivity of $G$ implies that the subgroup $P\leqs S_k$ induced by the conjugation action of $G$ on the set of factors of $T^k$ is either primitive, or $k = 2$ and $P = A_2 = 1$. The group $P$ is called the \textit{top group} of $G$ and we note that
\begin{equation}\label{e:diag}
T^k\normeq G \leqs T^k.(\Out(T)\times P).
\end{equation}

The first systematic study of bases for diagonal type groups was initiated by Fawcett in \cite{F_diag}. Here she shows that $b(G) = 2$ if $P\notin\{A_k,S_k\}$, and in the general setting she determines the exact base size of $G$ up to one of two possibilities (see Theorem \ref{t:F}). One of the key ingredients in \cite{F_diag} is a theorem of Seress \cite{S_dist}, which asserts that if $k > 32$ and $P\notin\{A_k,S_k\}$, then there exists a subset of $\{1,\dots,k\}$ with trivial setwise stabiliser in $P$. However, this does not hold if $P \in\{A_k,S_k\}$, and hence a different approach is required. In this paper, we extend Fawcett's work by determining the exact base size in all cases (see Theorem \ref{thm:main} below).

In recent years, there has been significant interest in studying the base-two primitive groups (we say $G$ is \textit{base-two} if $b(G) = 2$). Indeed, a project with the ambitious aim of classifying these groups was initiated by Jan Saxl in the 1990s and it continues to be actively pursued, with many interesting applications and open problems. For example, Burness and Giudici \cite{BG_Saxl} define the \emph{Saxl graph} of a base-two group $G \leqs \mathrm{Sym}(\Omega)$ to be the graph with vertex set $\Omega$, with two vertices adjacent if they form a base for $G$. It is easy to see that the Saxl graph of a base-two primitive group is connected and an intriguing conjecture asserts that its diameter is at most $2$ (see \cite[Conjecture 4.5]{BG_Saxl}). This has been verified in several special cases (for example, see \cite{BH_prod,BH_Saxl,CD_Saxl,LP_Saxl}), but it remains an open problem.

Returning to a diagonal type group $G$ as in \eqref{e:diag}, recall that Fawcett \cite{F_diag} has proved that $b(G) = 2$ if $P \notin\{A_k,S_k\}$. Our first result resolves the base-two problem for diagonal type groups in full generality.

\begin{theorem}
	\label{thm:b(G)=2}
	Let $G$ be a diagonal type primitive group with socle $T^k$ and top group $P\leqs S_k$. Then $b(G) = 2$ if and only if one of the following holds:
	\begin{enumerate}\addtolength{\itemsep}{0.2\baselineskip}
		\item[{\rm (i)}] $P\notin \{A_k,S_k\}$.
		\item[{\rm (ii)}] $3\leqs k\leqs |T|-3$.
		\item[{\rm (iii)}] $k\in\{|T|-2,|T|-1\}$ and $G$ does not contain $S_k$.
	\end{enumerate}
\end{theorem}

Note that $b(G)\leqs 2$ if and only if $G$ has a regular suborbit, and there is a natural interest in studying the finite primitive groups with a unique regular suborbit. For example, notice that $G$ has a unique regular suborbit if and only if the Saxl graph of $G$ is $G$-arc-transitive. In this direction, we refer the reader to \cite[Theorem 1.6]{BH_Saxl} for a classification of the relevant almost simple primitive groups with soluble point stabilisers, and \cite[Corollary 5]{BH_prod} for partial results on product type groups. Here we resolve this problem for diagonal type groups.

\begin{theorem}
	\label{thm:r=1}
	Let $G$ be a diagonal type primitive group with socle $T^k$. Then $G$ has a unique regular suborbit if and only if $T = A_5$, $k\in\{3,57\}$  and $G = T^k.(\Out(T)\times S_k)$.
\end{theorem}

We now present our main result, which determines the precise base size of every primitive group of diagonal type. This is the first family of primitive groups arising in the O'Nan-Scott theorem for which the exact base sizes are known.

\begin{theorem}
	\label{thm:main}
	Let $G$ be a diagonal type primitive group with socle $T^k$ and top group $P\leqs S_k$.
	\begin{enumerate}\addtolength{\itemsep}{0.2\baselineskip}
		\item[{\rm (i)}] If $P\notin\{A_k,S_k\}$, then $b(G) = 2$.
		\item[{\rm (ii)}] If $k = 2$, then $b(G)\in\{3,4\}$, with $b(G) = 4$ if and only if $T\in\{A_5,A_6\}$ and $G = T^2.(\Out(T)\times S_2)$.
		\item[{\rm (iii)}] If $k\geqs 3$, $P\in\{A_k,S_k\}$ and $|T|^{\ell-1} < k\leqs |T|^{\ell}$ with $\ell\geqs 1$, then $b(G)\in\{\ell+1,\ell+2\}$. Moreover, $b(G) = \ell+2$ if and only if one of the following holds:
		\vspace{1mm}
		\begin{enumerate}\addtolength{\itemsep}{0.2\baselineskip}
			\item[{\rm (a)}] $k = |T|$.
			\item[{\rm (b)}] $k\in\{|T|-2,|T|^\ell-1,|T|^\ell\}$ and $S_k\leqs G$.
			\item[{\rm (c)}] $k = |T|^2-2$, $T\in\{A_5,A_6\}$ and $G = T^k.(\Out(T)\times S_k)$.
		\end{enumerate}
	\end{enumerate}
\end{theorem}

Let us briefly discuss the methods we will use to establish our main theorems. Focussing first on Theorem \ref{thm:b(G)=2}, recall that the \emph{holomorph} of a non-abelian finite simple group $T$ is the group
\begin{equation*}
\Hol(T) = T{:}\Aut(T) = T^2.\Out(T),
\end{equation*}
which can be viewed as a primitive diagonal type group (with $k = 2$ and top group $P = 1$) in terms of its natural action on $T$. We write $\Hol(T,S)$ for the setwise stabiliser of $S\subseteq T$ in $\Hol(T)$. A key observation is Lemma \ref{l:Sk_eq_Hol}, which implies that
\begin{equation*}
\mbox{$b(G) = 2$ if there exists $S\subseteq T$ such that $|S| = k$ and $\Hol(T,S) = 1$}.
\end{equation*}
This essentially reduces the proof of Theorem \ref{thm:b(G)=2} to the cases where $3\leqs k\leqs |T|/2$. However, it is rather difficult to directly construct an appropriate subset $S$ of $T$ such that $\Hol(T,S) = 1$.

To overcome this difficulty, we adopt a probabilistic approach for $k\geqs 5$ in the proof of Theorem \ref{thm:b(G)=2} (see Section \ref{s:prob} for more details). More specifically, we estimate the probability that a random $k$-subset $S$ of $T$ satisfies $\Hol(T,S) = 1$, and we also use fixed point ratios to study the probability that a random pair in $\Omega$ is a base for $G$. The former is a new idea, which involves computing
\begin{equation*}
\max\{|C_T(x)|:1\ne x\in\Aut(T)\}
\end{equation*}
in Theorem \ref{t:fixity_Hol}, while the latter is a widely used technique in the study of base sizes introduced by Liebeck and Shalev \cite{LS_prob}. The cases where $k = 3$ or $4$ will be treated separately in Section \ref{ss:34}. Here we use the fact that $T$ is invariably generated by two elements (which is proved in \cite{GM_inv_gen} and \cite{KLS_inv_gen}, independently), and a theorem of Gow \cite{G_reg_semi} on the products of regular semisimple classes in groups of Lie type. We will use a very similar approach to establish Theorem \ref{thm:r=1}.

The proof of Theorem \ref{thm:main} will be completed in Section \ref{s:proof_main}, and the main step involves constructing a base of size $\ell + 1$ when $|T|^{\ell-1} < k \leqs  |T|^\ell-3$ for some $\ell\geqs 2$. Once again, our construction requires the existence of a suitable subset $S$ of $T$ such that $\Hol(T,S) = 1$. We will treat the case where $k = 2$ separately, working with a theorem of Leemans and Liebeck \cite{LL_gen} on the existence of a generating pair of $T$ with a certain property (see Theorem \ref{t:LL_t:1.1}).

As described above, a key ingredient in our study of bases for diagonal type groups is the following result, which may be of independent interest.

\begin{theorem}
	\label{thm:Hol}
	Let $T$ be a non-abelian finite simple group and suppose $3\leqs m\leqs |T|-3$. Then there exists $S\subseteq T$ such that $|S| = m$ and $\Hol(T,S) = 1$.
\end{theorem}

Similarly, let $\Aut(T,S)$ be the setwise stabiliser of $S\subseteq T^\#$ in $\Aut(T)$, where $T^\# = T\setminus \{1\}$. Note that $\Aut(T,S) = \Aut(T,T^\#\setminus S)$. By Theorem \ref{thm:Hol} and the transitivity of $\Hol(T)$, if $3\leqs m\leqs |T|-3$, then there exists $S\subseteq T$ containing $1$ such that $|S| = m$ and $\Hol(T,S) = 1$. This implies that $\Aut(T,S\setminus\{1\}) = 1$ and we have the following corollary.

\begin{corol}
	\label{cor:Aut}
	Let $T$ be a non-abelian finite simple group and suppose $2\leqs m\leqs |T|-3$. Then there exists $S\subseteq T^\#$ such that $|S| = m$ and $\Aut(T,S) = 1$.
\end{corol}

To conclude this section, we highlight a connection to some interesting problems in algebraic combinatorics. A digraph $\Gamma$ is said to be a \textit{digraphical regular representation (DRR)} of a group $X$ if $\Aut(\Gamma)\cong X$ acts regularly on the vertex set of $\Gamma$. In particular, if $\Gamma$ is a DRR of $X$, then $\Gamma$ is isomorphic to a Cayley digraph $\mathrm{Cay}(X,S)$ for some $S\subseteq X^\#$ with $\Aut(X,S) = 1$. A classical result of Babai \cite{B_DRR} shows that a finite group $X$ admits a DRR if and only if $X$ is not a quaternion group nor one of four elementary abelian groups. Moreover, it was conjectured by Babai and Godsil \cite{BG_DRR,G_DRR} that if $X$ is a group of order $n$, then the proportion of subsets $S\subseteq X^\#$ such that $\mathrm{Cay}(X,S)$ is a DRR tends to $1$ as $n\to\infty$. This conjecture has been proved recently by Morris and Spiga \cite{MS_DRR}.

Given a finite group $X$, it is natural to consider the existence of a DRR with a prescribed valency, noting that the valency of $\mathrm{Cay}(X,S)$ is $|S|$. Recently, there are some results concerning this problem in relation to finite simple groups (for example, see \cite{VX_DRR,XZZ_DRR} for the existence of some families of DRRs with a fixed valency $k\leqs 3$, and \cite{X_DRR} for $k\geqs 5$). However, there appear to be no asymptotic results in the literature concerning the proportion of DRRs of a fixed valency of a given finite group. With this problem in mind, let $\mathbb{Q}_k(X)$ be the probability that a random $k$-subset of $X^\#$ has a non-trivial setwise stabiliser in $\Aut(X)$. That is,
\begin{equation*}
\mathbb{Q}_k(X) = \frac{|\{R\in\mathscr{S}_k:\Aut(X,R)\ne 1\}|}{|\mathscr{S}_k|},
\end{equation*}
where $\mathscr{S}_k$ is the set of $k$-subsets of $X^\#$. In Section \ref{s:GRR}, we will prove the following results.

\begin{theorem}
	\label{thm:Aut(T,S)_to_1}
	Let $k\geqs 4$ be an integer and let $(T_n)$ be a sequence of non-abelian finite simple groups such that $|T_n|\to\infty $ as $n\to\infty$. Then $\mathbb{Q}_k(T_n)\to 0$ as $n\to\infty$.
\end{theorem}

\begin{theorem}\label{thm:prob_Aut}
	Let $T$ be a non-abelian finite simple group and let $k$ be an integer such that $5\log_2|T| < k < |T|-5\log_2|T|$. Then $\mathbb{Q}_k(T)<1/|T|$.
\end{theorem}

We anticipate that these two results will be useful in studying the abundance of fixed-valent DRRs of non-abelian finite simple groups.

\subsection*{Notation}

Let $G\leqs\mathrm{Sym}(\Omega)$ be a permutation group and $\Delta\subseteq\Omega$. Then the pointwise and setwise stabilisers of $\Delta$ in $G$ are sometimes denoted $G_{(\Delta)}$ and $G_{\{\Delta\}}$, respectively. We adopt the standard notation for simple groups of Lie type from \cite{KL_classical}. All logarithms, if not specified, are in base $2$. Finally, if $k$ is a positive integer, then we write $[k]$ for the set $\{1,\dots,k\}$.

\section{Preliminaries}

\label{s:pre}

\subsection{Diagonal type groups}

Here we adopt the notation in \cite{F_diag}. Let $k\geqs 2$ be an integer and let $T$ be a non-abelian finite simple group. Define
\begin{equation*}
\begin{aligned}
&W(k,T):=\{(\alpha_1,\dots,\alpha_k)\pi\in \Aut(T)\wr_kS_k:\alpha_1 \Inn(T) = \alpha_i\Inn(T)\mbox{ for all }i\},\\
&D(k,T):=\{(\alpha,\dots,\alpha)\pi\in \Aut(T)\wr_kS_k\},\\
&\Omega(k,T):=[W(k,T):D(k,T)].
\end{aligned}
\end{equation*}
Then $|\Omega(k,T)| = |T|^{k-1}$ and $W(k,T) = T^k.(\Out(T)\times S_k)$ acts faithfully on $\Omega(k,T)$. We say that a group $G\leqs\mathrm{Sym}(\Omega)$ with $\Omega = \Omega(k,T)$ is of \textit{diagonal type} if
\begin{equation*}
T^k\normeq G\leqs T^k.(\Out(T)\times S_k).
\end{equation*}
Let $P_G$ denote the subgroup of $S_k$ induced by the conjugation action of $G$ on the set of factors of $T^k$. That is,
\begin{equation*}
P_G = \{\pi\in S_k:(\alpha_1,\dots,\alpha_k)\pi\in G\mbox{ for some }\alpha_1,\dots,\a_k\in\Aut(T)\}.
\end{equation*}
Then naturally we have $G\leqs T^k.(\Out(T)\times P_G)$ as in \eqref{e:diag}. Moreover, $G$ is primitive if and only if either $P_G$ is primitive on $[k] = \{1,\dots,k\}$, or $k = 2$ and $P_G = 1$. From now on, if $G$ is clear from the context, we denote $P = P_G$ and
\begin{equation*}
\begin{aligned}
&W:=T^k.(\Out(T)\times P),\\
&D:=\{(\alpha,\dots,\alpha)\pi:\alpha\in\Aut(T),\pi\in P\},\\
&\Omega:=\Omega(k,T) = [W:D].
\end{aligned}
\end{equation*}
We write $\varphi_t\in\Inn(T)$ for the inner automorphism such that $x^{\varphi_t} = t^{-1}xt$ for any $x\in T$. Thus,
\begin{equation*}
\Omega = \{D(\varphi_{t_1},\dots,\varphi_{t_k}):t_1,\dots,t_k\in T\}.
\end{equation*}
The action of $G$ on $\Omega$ is given by
\begin{equation*}
D(\varphi_{t_1},\dots,\varphi_{t_k})^{(\alpha_1,\dots,\a_k)\pi} = D(\varphi_{t_{1^{\pi^{-1}}}}\a_{1^{\pi^{-1}}},\dots,\varphi_{t_{k^{\pi^{-1}}}}\a_{k^{\pi^{-1}}}),
\end{equation*}
and the stabiliser of $D\in \Omega$ in $W$ is $D$ itself. In particular, for any element $(\a,\dots,\a)\pi\in D$, we have
\begin{equation*}
D(\varphi_{t_1},\dots,\varphi_{t_k})^{(\alpha,\dots,\a)\pi} = D(\varphi_{t_{1^{\pi^{-1}}}^{\a}},\dots,\varphi_{t_{k^{\pi^{-1}}}^{\a}}),
\end{equation*}
noting that $ \a^{-1}\varphi_t\a = \varphi_{t^\a}$ for all $t\in T$.

We begin by recording some preliminary results on bases for diagonal type groups from \cite{F_diag}. We start with \cite[Lemma 3.4]{F_diag}.

\begin{lem}
	\label{l:l:3.4_diag}
	Let $t_1,\dots,t_k$ be elements of $T$ such that the following two properties are satisfied:
	\begin{enumerate}\addtolength{\itemsep}{0.2\baselineskip}
		\item[{\rm (i)}] At least two of the $t_i$ are trivial and at least one is non-trivial.
		\item[{\rm (ii)}] If $t_i$ and $t_j$ are non-trivial and $i\ne j$, then $t_i\ne t_j$.
	\end{enumerate}
	Then $(\alpha,\dots,\alpha)\pi\in G$ fixes $D(\varphi_{t_1},\dots,\varphi_{t_k})$ only if $t_i^\alpha = t_{i^\pi}$ for all $i$.
\end{lem}

For any $\mathbf{x} = (\varphi_{t_1},\dots,\varphi_{t_k})\in \Inn(T)^k$, we define an associated partition $\mathcal{P}_{\mathbf{x}} = \{\mathcal{P}_t:t\in T\}$ of $[k]$ such that $i\in\mathcal{P}_t$ if $t_i = t$. Note that some parts $\mathcal{P}_t$ in $\mathcal{P}_{\mathbf{x}}$ might be empty. The following lemma is an extension of Lemma \ref{l:l:3.4_diag}, which will be useful later in Section \ref{s:proof_main}. Recall that $P_{\{\mathcal{P}_{\mathbf{x}}\}}$ is the setwise stabiliser of the partition $\mathcal{P}_{\mathbf{x}}$ in $P$. In particular, if $t_{i^\pi} = t_{j^\pi}$ whenever $t_i = t_j$, then we have $\pi\in P_{\{\mathcal{P}_{\mathbf{x}}\}}$.

\begin{lem}
	\label{l:l:3.4_diag_ext}
	Let $\mathbf{x} = (\varphi_{t_1},\dots,\varphi_{t_k})\in \Inn(T)^k$, $\omega = D\mathbf{x}\in\Omega$ and let $\mathcal{P}_{\mathbf{x}} = \{\mathcal{P}_t:t\in T\}$ be the associated partition of $[k]$ as above. Suppose $(\a,\dots,\a)\pi\in G_\omega$. Then
	\begin{enumerate}\addtolength{\itemsep}{0.2\baselineskip}
		\item[{\rm (i)}] $\pi\in P_{\{\mathcal{P}_{\mathbf{x}}\}}$; and
		\item[{\rm (ii)}] If $0<|\mathcal{P}_1|\ne |\mathcal{P}_t|$ for all $t\ne 1$, then $t_i^\alpha = t_{i^\pi}$ for all $i$.
	\end{enumerate}
\end{lem}

\begin{proof}
	As $(\a,\dots,\a)\pi$ fixes $\omega = D(\varphi_{t_1},\dots,\varphi_{t_k})$, there exists a unique $g\in T$ such that $t_i^\alpha = gt_{i^\pi}$ for all $i\in\{1,\dots,k\}$. Suppose $t_i = t_j$ for some $i\ne j$ (so $i$ and $j$ are in the same part of $\mathcal{P}_{\mathbf{x}}$). Then $t_{i^\pi} = g^{-1}t_i^\a = g^{-1}t_j^\a = t_{j^\pi}$. This gives part (i).
	
	For part (ii), it suffices to show that $g = 1$. If $t_i = 1$, then $t_{i^\pi} = g^{-1}$, and we get $t_{j^\pi} = g^{-1}t_j^\alpha \ne g^{-1}$ if $t_j\ne 1$. This implies that $|\mathcal{P}_{g^{-1}}| = |\mathcal{P}_1|$, so $g = 1$ by our assumption.
\end{proof}

Now we combine Fawcett's main results on base sizes of diagonal type groups from \cite{F_diag}.

\begin{thm}
	\label{t:F}
	Let $G$ be a diagonal type primitive group with socle $T^k$ and top group $P\leqs S_k$.
	\begin{enumerate}\addtolength{\itemsep}{0.2\baselineskip}
		\item[{\rm (i)}] If $P\notin\{A_k,S_k\}$, then $b(G) = 2$.
		\item[{\rm (ii)}] If $k = 2$, then $b(G) = 3$ if $P = 1$, and $b(G)\in\{3,4\}$ if $P = S_2$.
		\item[{\rm (iii)}] If $k\geqs 3$, $P\in\{A_k,S_k\}$ and $|T|^{\ell-1}<k\leqs |T|^\ell$ with $\ell\geqs 1$, then $b(G)\in\{\ell+1,\ell+2\}$. Moreover, if either $k = |T|$, or $k\in\{|T|^{\ell}-1,|T|^\ell\}$ and $S_k\leqs G$, then $b(G) = \ell+2$.
	\end{enumerate}
\end{thm}

\begin{cor}
	\label{c:b(G)=2}
	If $P\in\{A_k,S_k\}$ and $b(G) = 2$, then $2<k<|T|$.
\end{cor}

The following is \cite[Lemma 3.11]{F_diag}.

\begin{lem}
	\label{l:l:3.11_F}
	Suppose $P\in\{A_k,S_k\}$ and there exists an odd integer $3\leqs s\leqs k$ that is relatively prime to the order of every element of $\Out(T)$. Then $G$ contains $A_k$.
\end{lem}

\begin{cor}
	\label{c:Ak<G}
	If $P\in\{A_k,S_k\}$ and $k\geqs |T|-3$, then $G$ contains $A_k$.
\end{cor}

\begin{proof}
	We have $|\Out(T)|<|T|^{1/3}$ by Lemma \ref{l:|Out(T)|} below. In particular, $|\Out(T)|<|T|/3$, so there exists a prime $s$ such that $|\Out(T)|<s<k$ (Bertrand's postulate). Now apply Lemma \ref{l:l:3.11_F}.
\end{proof}

The following extends \cite[Proposition 3.3]{F_diag}, which asserts that $b(G) = 2$ if $k  > 32$ and $P\notin\{A_k,S_k\}$. Here $r(G)$ is the number of regular suborbits of $G$, noting that $r(G)\geqs 1$ if and only if $b(G)\leqs 2$.

\begin{prop}
	\label{p:r(G)>1_k>32}
	If $k > 32$ and $P\notin\{A_k,S_k\}$, then $r(G)\geqs 2$.
\end{prop}

\begin{proof}
	We use the same construction as in the proof of \cite[Proposition 3.3]{F_diag}. By \cite[Theorem 1]{S_dist}, there exists a partition $\mathcal{P}=\{\Pi_1,\Pi_2,\Pi_3\}$ of $[k]$ such that each $\Pi_i$ is non-empty, $|\Pi_1|$, $|\Pi_2|$ and $|\Pi_3|$ are distinct, and
	\begin{equation}\label{e:partition}
	\bigcap_{m = 1}^3 P_{\{\Pi_m\}} = 1.
	\end{equation}
	Let $x_1,x_2\in T$ be non-trivial elements of distinct orders. By the main theorem of \cite{GK_3/2}, there exist $y_1,y_2\in T$ such that $\la x_i,y_i\ra = T$. Let $\Delta_i = \{D,D(\varphi_{t_{i,1}},\dots,\varphi_{t_{i,k}})\}$ for $i\in\{1,2\}$, where $t_{i,j} = 1$ if $j\in\Pi_1$, $t_{i,j} = x_i$ if $j\in\Pi_2$, and $t_{i,j} = y_i$ if $j\in\Pi_3$. As explained in the proof of \cite[Proposition 3.3]{F_diag}, both $\Delta_1$ and $\Delta_2$ are bases for $G$.
	
	Suppose $\Delta_1^{(\a,\dots,\a)\pi} = \Delta_2$. Then there exists $g\in T$ such that $t_{1,j}^\a = gt_{2,j^\pi}$ for all $j\in [k]$. If $t_{1,j} = t_{1,j'}$ for some $j'\in [k]$, then $t_{2,j} = t_{2,j'}$ and
	\begin{equation*}
	t_{2,j^\pi} = g^{-1}t_{1,j}^\a = g^{-1}t_{1,j'}^\a = t_{2,(j')^\pi}.
	\end{equation*}
	Hence, $\pi\in P_{\{\mathcal{P}\}}$, and so $\pi\in P_{\{\Pi_m\}}$ for each $m\in\{1,2,3\}$ as $|\Pi_1|$, $|\Pi_2|$ and $|\Pi_3|$ are distinct. This implies that $\pi = 1$ by \eqref{e:partition}, and so $g = 1$. However, it follows that $x_1^\a = x_2$, which is incompatible with $|x_1|\ne |x_2|$. We conclude that $\Delta_1$ and $\Delta_2$ are in distinct $G_D$-orbits, and thus $r(G)\geqs 2$.
\end{proof}

\begin{rem}
	\label{r:32}
	In fact, as we will show in Section \ref{s:proof_b(G)=2}, we have $r(G)\geqs 1$ whenever $3\leqs k\leqs |T|-3$, with equality if and only if $T = A_5$, $k \in\{3,57\}$ and $G = T^k.(\Out(T)\times S_k)$. In particular, it follows that $r(G)\geqs 2$ if $k\leqs 32$ and $P\notin\{A_k,S_k\}$.
\end{rem}

\subsection{Simple groups}

In this section, we record some properties of finite simple groups that will be used to prove our main results. In the whole paper, $T$ is a non-abelian finite simple group. We start with \cite[Lemma 4.8]{F_diag}.

\begin{lem}
	\label{l:|Out(T)|}
	We have $|\Out(T)| < |T|^{1/3}$.
\end{lem}

Let $T$ be a finite simple group of Lie type defined over $\mathbb{F}_q$, where $q = p^f$ and $p$ is a prime. Then we may write $T = O^{p'}(Y_\sigma)$, where $Y$ is the ambient simple algebraic group over the algebraic closure $K$ of $\mathbb{F}_q$ and $\sigma$ is an appropriate Steinberg endomorphism. Note that $Y_\sigma = \mathrm{Inndiag}(T)$ is the group of inner-diagonal automorphisms of $T$.

\begin{lem}
	\label{l:Inndiag}
	Let $d = \frac{1}{2}\cdot \dim Y$ if $T\in\{{^2}B_2(q),{^2}G_2(q)',{^2}F_4(q)'\}$ and $d = \dim Y$ otherwise. Then $\frac{1}{2}q^d<|\mathrm{Inndiag}(T)|<q^d$.
\end{lem}

\begin{proof}
	This is \cite[Proposition 3.9(i)]{B_fpr2} when $T$ is a classical group, and the bounds for exceptional groups are clear.
\end{proof}

Recall that a semisimple element $x\in T$ is \textit{regular} if the connected component of $C_Y(x)$ is a maximal torus of $Y$. Equivalently, $x\in T$ is regular semisimple if and only if $|C_T(x)|$ is indivisible by $p$. In particular, if $T$ is a classical group with natural module $V$, then a semisimple element $x\in T$ is regular if a pre-image $\widehat{x}\in \GL(\overline{V})$ has distinct eigenvalues on $\overline{V} = V\otimes K$. And if $T$ is an orthogonal group, then $x$ is also regular if $\widehat{x}$ has a $2$-dimensional $(\pm 1)$-eigenspace and all the other eigenvalues are distinct.

We say that a subset $\{t_1,\dots,t_m\}$ of $T$ is an \textit{invariable} generating set if $\la t_1^{g_1},\dots,t_m^{g_m}\ra = T$ for any $g_1,\dots,g_m\in T$. It has been proved in \cite{GM_inv_gen} and \cite{KLS_inv_gen}, independently, that every non-abelian finite simple group is invariably generated by $2$ elements.

\begin{thm}
	\label{t:inv_gen_reg}
	Suppose $T\notin\{\LL_2(5),\LL_2(7),\Omega_8^+(2),\POmega_8^+(3)\}$ is a finite simple group of Lie type. Then there exist regular semisimple elements $x$ and $y$ of distinct orders such that $T$ is invariably generated by $\{x,y\}$.
\end{thm}

\begin{proof}
	If $T$ is an exceptional group, then we take $x$ and $y$ to be $t_1$ and $t_2$ in \cite[Table 2]{KLS_inv_gen}, respectively, noting that $t_1$ is a generator of the maximal torus $T_1$ in that table. It is evident that $|t_1|\ne |t_2|$ in each case, and $\{t_1,t_2\}$ invariably generates $T$ by \cite{KLS_inv_gen} (see \cite[p. 312]{KLS_inv_gen}). Moreover, we observe that $\la t_1\ra$ and $\la t_2\ra$ are both maximal tori, which implies that each $t_i$ is regular semisimple.
	
	To complete the proof, we may assume $T$ is a classical group. Here we will work with the corresponding quasisimple group $Q\in\{\SL_n^\e(q),\Sp_n(q),\Omega_n^\e(q)\}$, noting that if $Q$ is invariably generated by $\{t_1,t_2\}$, with $t_1$ and $t_2$ regular semisimple , then $T$ is invariably generated by $\{x,y\}$, where $x$ and $y$ are the images of $t_1$ and $t_2$ in $T$, respectively (so $x$ and $y$ are also regular semisimple). Moreover, $|x| = |t_1|/a$ and $|y| = |t_2|/b$ for some integers $a,b$ dividing $|Q|/|T|$, so $|x|\ne |y|$ if
	\begin{equation}
	\label{e:order}
	\mbox{$|t_1|$ is indivisible by $|t_2||Q|/|T|$ and $|t_2|$ is indivisible by $|t_1||Q|/|T|$.}
	\end{equation}
	
	First assume $Q\notin\{\SL_2(q),\Omega_8^+(q)\}$. Here we use the same $t_1$ and $t_2$ as presented in \cite[Table 1]{KLS_inv_gen}. In each case, it is clear that $t_1$ and $t_2$ are semisimple elements satisfying \eqref{e:order}, and $\{t_1,t_2\}$ invariably generates $Q$ by \cite[Lemma 5.3]{KLS_inv_gen}. Thus, it suffices to show that $t_1$ and $t_2$ are regular in every case, which is a straightforward exercise (for instance, we can work with the criterion for regularity in terms of the eigenvalues on $\overline{V}$ discussed as above). For example, consider the element $t_2\in Q = \Omega_{4m}^+(q)$. Here, $t_2$ lifts to an element $\widehat{t_2}\in \GL(V)$ of the form
	\begin{equation*}
	\widehat{t_2} =
	\begin{pmatrix}
	A&\\
	&B
	\end{pmatrix}^\delta
	\end{equation*}
	with respect to a standard basis (see \cite[Proposition 2.5.3]{KL_classical}),
	where $\delta\in\{1,2\}$, $A\in\mathrm{SO}_{4m-4}^-(q)$ has order $q^{2m-2}+1$ and $B\in\mathrm{SO}_4^-(q)$ has order $q^{2}+1$. We only deal with the case where $\delta = 1$ since a similar argument holds for $\delta = 2$. Then the eigenvalues of $A$ over the algebraic closure $K$ of $\mathbb{F}_q$ are
	\begin{equation*}
	\lambda,\lambda^q,\dots,\lambda^{q^{4m-3}}
	\end{equation*}
	for some $\lambda\in K$ of order $q^{2m-2}+1$. Similarly, the set of eigenvalues of $B$ over $K$ is $\{\mu,\mu^q,\mu^{q^2},\mu^{q^3}\}$ for some $\mu\in K$ of order $q^2+1$. If $\mu = \lambda^{q^i}$ for some $i\in\{0,\dots,4m-3\}$, then $\lambda^{q^i(q^2+1)} = 1$ and so $q^{2m-2}+1$ divides $q^i(q^2+1)$, which implies that $q^{2m-2}+1$ divides $q^2+1$ since $(q^{2m-2}+1,q^i) = 1$. However, since $m\geqs 3$, this is impossible. It follows that the eigenvalues of $\widehat{t_2}$ over $K$ are distinct, and so $t_2$ is a regular semisimple element.
	
	Finally, let us handle the two excluded cases above. If $Q = \SL_2(q)$ with $q\notin\{4,5,7,9\}$, then we take the same $t_1$ and $t_2$ as indicated in the proof of \cite[Lemma 5.3]{KLS_inv_gen}. The group $\LL_2(4)$ is invariably generated by an element of order $3$ and an element of order $5$, and if $q = 9$ then we take $x$ and $y$ to be of order $4$ and $5$, respectively.  If $Q = \Omega_8^+(q)$ with $q\notin\{2,3\}$, then we take $t_1$ as in \cite[Table 1]{KLS_inv_gen}, and $t_2$ an element of order $(q^3-1)/(2,q-1)$ as described in the proof of \cite[Lemma 5.4]{KLS_inv_gen}, where it is denoted $t_3$.
\end{proof}

It is worth noting that the excluded groups $\LL_2(5)$, $\LL_2(7)$, $\Omega_8^+(2)$ and $\POmega_8^+(3)$ in Theorem \ref{t:inv_gen_reg} are not invariably generated by any pair of regular semisimple elements of distinct orders. This is can be checked using {\sc Magma} V2.26-11 \cite{Magma}. More specifically, we find the set of maximal overgroups of an element $x\in T$ up to $T$-conjugacy using the method as in \cite[Section 1.2]{BHa_code}, noting that $x$ and $y$ do not invariably generate $T$ if they have a common maximal overgroup in $T$ up to $T$-conjugacy. 

From now on, we will assume $n\geqs 3$ if $T = \UU_n(q)$, $n\geqs 4$ is even if $T = \PSp_n(q)$, and $n\geqs 7$ if $T = \POmega_n^\e(q)$. We will also exclude the groups
\begin{equation}\label{e:iso_simple}
\LL_2(4),\LL_2(5),\LL_2(9),\LL_3(2),\LL_4(2),\UU_4(2),\Sp_4(2)',G_2(2)',{^2}G_2(3)'
\end{equation}
as each of them is isomorphic to one of the following groups:
\begin{equation*}
A_5,A_6,A_8,\LL_2(7),\LL_2(8),\UU_3(3),\PSp_4(3).
\end{equation*}

As mentioned in Section \ref{s:intro}, one of our probabilistic approaches in Section \ref{s:prob} relies on computing
\begin{equation*}
h(T):=\max\{|C_T(x)|:1\ne x\in\Aut(T)\}
\end{equation*}
for every non-abelian finite simple group $T$.

\begin{thm}
	\label{t:fixity_Hol}
	Let $T$ be a non-abelian finite simple group. Then $h(T)$ is listed in Table \ref{tab:fix(Hol(T))}.
\end{thm}

{\small
	\begin{table}
		\[
		\begin{array}{llll} \hline
		T&h(T)&x&\mbox{Conditions}\\\hline
		A_n&(n-2)!&(1,2)&\\
		\mathrm{M}_{11}&48&\texttt{2A}\\
		\mathrm{M}_{12}&240&\texttt{2A}\\
		\mathrm{M}_{22}&1344&\texttt{2B}\\
		\mathrm{M}_{23}&2688&\texttt{2A}\\
		\mathrm{M}_{24}&21504&\texttt{2A}\\
		\mathrm{J}_1&120&\texttt{2A}\\
		\mathrm{J}_2&1920&\texttt{2A}\\
		\mathrm{J}_3&2448&\texttt{2B}\\
		\mathrm{J}_4&21799895040&\texttt{2A}\\
		\mathrm{HS}&40320&\texttt{2C}\\
		\mathrm{McL}&40320&\texttt{2A}\\
		\mathrm{Suz}&9797760&\texttt{3A}\\
		\mathrm{He}&161280&\texttt{2A}\\
		\mathrm{HN}&177408000&\texttt{2A}\\
		\mathrm{Ru}&245760&\texttt{2A}\\
		\mathrm{Ly}&2694384000&\texttt{3A}\\
		\mathrm{Co}_1&1345036492800&\texttt{3A}\\
		\mathrm{Co}_2&743178240&\texttt{2A}\\
		\mathrm{Co}_3&2903040&\texttt{2A}\\
		\mathrm{Th}&92897280&\texttt{2A}\\
		\mathrm{O'N}&175560&\texttt{2B}\\
		\mathrm{Fi}_{22}&18393661440&\texttt{2A}\\
		\mathrm{Fi}_{23}&129123503308800&\texttt{2A}\\
		\mathrm{Fi}_{24}'&4089470473293004800&\texttt{2C}\\
		\mathbb{B}&306129918735099415756800&\texttt{2A}\\
		\mathbb{M}&8309562962452852382355161088000000&\texttt{2A}\\
		E_8(q)&q^{57}|E_7(q)|(2,q-1)&u_\a\\
		E_7(q)&q^{33}|\mathrm{SO}_{12}^+(q)|/(2,q)&u_\a\\
		E_6^\e(q)&q^{21}|\SL_6^\e(q)|/(3,q-\e)&u_\a\\
		F_4(q)&q^{15}|\Sp_6(q)|&u_\a\\
		G_2(q)&q^5|\SL_2(q)|&u_\a\\
		{^3}D_4(q)&q^{12}(q^6-1)&u_\a\\
		{^2}F_4(q)&q^{10}|{^2}B_2(q)|&u_\a&q>2\\
		{^2}F_4(2)'&10240&u_\a\\
		{^2}G_2(q)&q^3&u_\a\\
		{^2}B_2(q)&q^2&u_\a\\
		\LL_n^\e(q)&|\PGL_2(q^{1/2})|&\phi^{f/2}&\mbox{$n=2$, $f$ is even}\\
		&q+1&s&\mbox{$n=2$, $f$ is odd}\\
		&|\PGL_3(q^{1/2})|&\phi^{f/2}&\mbox{$n=3$, $\e=+$, $f$ is even, $3\mid q^{1/2}+1$}\\
		&|\PGU_3(q^{1/2})|&\phi^{f/2}\gamma&\mbox{$n=3$, $\e=+$, $f$ is even, $3\nmid q^{1/2}+1$}\\
		&(2,q-\e)|\mathrm{PGSp}_4(q)|/(4,q-\e)&\gamma_1&n=4\\
		&|\GU_{n-1}(q)|/(n,q+1)&[\omega I_1,I_{n-1}]&\mbox{$n\geqs 6$ is even, $\e = -$}\\
		&q^{2n-3}|\GL_{n-2}^\e(q)|/(n,q-\e)&u_\a&\mbox{otherwise}\\
		\PSp_n(q)&|\Sp_2(q^2)|&t_1&\mbox{$n=4$, $q$ is odd}\\
		&q^{n-1}|\Sp_{n-2}(q)|&u_\a&\mbox{otherwise}\\
		\POmega_n^\e(q)&|\mathrm{SO}_{n-1}^-(q)|&t_1'&\mbox{$n$ is odd}\\
		&|\Sp_{n-2}(q)|&b_1&\mbox{$q$ is even}\\
		&|\Omega_{n-1}(q)|&\gamma_1&\mbox{$n$ is even, $q$ is odd}\\
		\hline
		\end{array}
		\]
		\caption{$h(T)$ in Theorem \ref{t:fixity_Hol}}
		\label{tab:fix(Hol(T))}
\end{table}}

\begin{rem}
	Let us briefly comment on the notation we adopt in the third column of Table \ref{tab:fix(Hol(T))}, where we record an element $x\in \Aut(T)$ with $|C_T(x)| = h(T)$.
	\begin{enumerate}\addtolength{\itemsep}{0.2\baselineskip}
		\item[{\rm (i)}] We adopt the notation in \cite{W_WebAt} for labelling conjugacy classes when $T$ is a sporadic group. If $T$ is Lie type, then we write $u_\a$ for a long root element.
		\item[{\rm (ii)}] When $T = \LL_n(q)$, we write $\phi$ for a field automorphism of order $f = \log_pq$, where $p$ is the characteristic of the field $\mathbb{F}_q$.
		\item[{\rm (iii)}] If $T = \LL_2(q)$, then let $H$ be the normaliser in $\PGL_2(q)$ of a non-split maximal torus of $T$, so $H\cong D_{2(q+1)}$. We then define $s\in H$ to be the central involution if $q$ is odd, and an arbitrary element of odd prime order if $q$ is even.
		\item[{\rm (iv)}] We adopt the notation in \cite[Chapter 3]{BG_classical} for elements of classical groups. For example, if $n$ is even, $q$ is odd and $T = \POmega_n^\e(q)$, then a pre-image in $\mathrm{O}_n^\e(q)$ of an element of type $\gamma_1$ is an involution of the form $[-I_1,I_{n-1}]$ (see \cite[Section 3.5.2.14]{BG_classical}).
	\end{enumerate}
\end{rem}

\begin{proof}[Proof of Theorem \ref{t:fixity_Hol}.]
	First observe that we only need to consider prime order elements in $\Aut(T)$, since $C_T(x)\leqs C_T(x^m)$ for any integer $m$ and $x\in\Aut(T)$.
	
	Assume $T = A_n$ is an alternating group. If $n = 5$ or $6$, then the result can be checked using {\sc Magma}. Now assume $n\geqs 7$, so $\Aut(T) = S_n$. It is easy to see that $|C_T(x)|$ is maximal when $x$ is a transposition, in which case $C_{S_n}(x) \cong S_2\times S_{n-2}$ and thus $|C_T(x)| = (n-2)!$. Hence, $h(T) = (n-2)!$. If $T$ is a sporadic group, then $|C_T(x)|$ can be read off from the character table of $T$, which can be accessed computationally via the \textsf{GAP} Character Table Library \cite{B_GAPCTL}.
	
	For the remainder, we may assume $T$ is a simple group of Lie type over $\mathbb{F}_q$, where $q = p^f$ with $p$ a prime. Assume $x\in\Aut(T)$ is of prime order $r$. If $x\in\mathrm{Inndiag}(T)$, then $x$ is semisimple if $p\ne r$, otherwise $x$ is unipotent. And if $x\notin\mathrm{Inndiag}(T)$, then $x$ is a field, graph or graph-field automorphism. Here if $x$ is a graph or graph-field automorphism, then $r\in\{2,3\}$.
	
	Assume $T$ is an exceptional group. Here we assume $T\ne {^2}G_2(3)'\cong \LL_2(8)$ and $T\ne G_2(2)'\cong \UU_3(3)$ as noted in \eqref{e:iso_simple}. By \cite[Proposition 2.11]{BT_extremely}, $|C_T(x)|$ is maximal when $x\in T$ is a long root element. Now assume $x\in T$ is a long root element. If $T$ is not ${^3}D_4(q)$ or ${^2}B_2(q)$, then $|C_T(x)|$ can be read off from the tables in \cite[Chapter 22]{LS_uni}, noting that $x^{\mathrm{Inndiag}(T)} = x^T$ by \cite[Corollary 17.10]{LS_uni}. If $T = {^3}D_4(q)$ or ${^2}B_2(q)$ then we can find $|C_T(x)|$ in \cite[p. 677]{S_3D4} and \cite{S_Suzuki}, respectively.
	
	For the remainder of the proof, we assume $T$ is a classical group defined over $\mathbb{F}_q$. Let $V$ be the natural module of $T$ and write $\overline{V} = V\otimes K$, where $K$ is the algebraic closure of $\mathbb{F}_q$. For $x\in\PGL(V)$, let $\widehat{x}$ be a pre-image of $x$ in $\GL(V)$. Following \cite[Definition 3.16]{B_fpr2}, we define
	\begin{equation*}
	\nu(x) = \min\{\dim [\overline{V},\lambda \widehat{x}]:\lambda\in K^*\},
	\end{equation*}
	where $[\overline{V},\lambda\widehat{x}] = \{v-\lambda\widehat{x}v:v\in \overline{V}\}$. That is, $\nu(x)$ is the codimension of the largest eigenspace of $\widehat{x}$ on $\overline{V}$, noting that $\nu(x)$ is independent of the choice of the pre-image $\widehat{x}$. Upper and lower bounds on $|x^T|$ in terms of $n$, $q$ and $\nu(x)$ are given in \cite[Section 3]{B_fpr2}. Similarly, if $x$ is a field, graph or graph-field automorphism, then lower bounds for $|x^T|$ can be read off from \cite[Table 3.11]{B_fpr2}. In addition, $|C_{\mathrm{Inndiag}(T)}(x)|$, and a description of the splitting of $x^{\mathrm{Inndiag}(T)}$ into distinct $T$-classes, can be found in \cite[Chapter 3]{BG_classical}. In particular, note that if $x\in\mathrm{Inndiag}(T)$ is a semisimple element of prime order, then $x^{\mathrm{Inndiag}(T)} = x^T$ (see \cite[Theorem 4.2.2(j)]{GLS_CFSG3}, also recorded as \cite[Theorem 3.1.12]{BG_classical}).
	
	We start with the case where $T = \LL_2(q)$. Let $H$ be the normaliser in $\PGL_2(q)$ of a non-split maximal torus of $T$, so $H\cong D_{2(q+1)}$. If $q$ is odd, then we let $x$ be the central involution in $H$, and if $q$ is even, let $x\in H$ be an element of odd prime order. Then $|C_T(x)| = q+1$, so $h(T)\geqs q+1$. Let $y\in\Aut(T)$ be an element of prime order. Note that if $y$ is unipotent then $|C_T(y)| = q$, and $|C_T(y)|$ divides $q+1$ or $q-1$ if $y$ is semisimple. Thus, we only need to consider field automorphisms, noting that $|C_{\PGL_2(q)}(y)| = |\PGL_2(q^{1/r})|$ if $y$ is a field automorphism of prime order $r$. It follows that $|C_{\PGL_2(q)}(y)| > q+1$ only if $r = 2$ (so $f$ is even). Indeed,
	\begin{equation*}
	|C_T(y)| = |C_{\PGL_2(q)}(y)| = |\PGL_2(q^{1/2})| >q+1
	\end{equation*}
	if $y$ is an involutory field automorphism, and so we conclude that $h(T) = |\PGL_2(q^{1/2})|$ if $f$ is even, and $h(T) = q+1$ if $f$ is odd.
	
	To complete the proof for linear and unitary groups, we assume $T = \LL_n^\e(q)$ with $n\geqs 3$. Let $x\in T$ be a unipotent element with Jordan form $[J_2,J_1^{n-2}]$ on the natural module, noting that $x$ is a long root element. Then $|C_{\PGL_n^\e(q)}(x)|$ can be read off from \cite[Tables B.3 and B.4]{BG_classical}, and we have $x^{\PGL_n^\e(q)} = x^T$ by \cite[Propositions 3.2.7 and 3.3.10]{BG_classical}. More specifically,
	\begin{equation*}
	|C_T(x)| = (n,q-\e)^{-1} q^{2n-3}|\GL_{n-2}^\e(q)|
	\end{equation*}
	and
	\begin{equation*}
	|x^T| = |x^{\PGL_n^\e(q)}|= \frac{|\PGL_n^\e(q)|}{q^{2n-3}|\GL_{n-2}^\e(q)|} < \frac{2q^{2n-1}}{q-1}.
	\end{equation*}
	The cases where $n\in\{3,4\}$ require special attention, which will be treated separately.
	
	Assume $T = \LL_3^\e(q)$, so $|C_T(x)| = (3,q-\e)^{-1}q^3(q-\e)$, and let $y$ be an element in $\Aut(T)$ of prime order that is not of Jordan form $[J_2,J_1]$. If $y$ is unipotent or semisimple and $\nu(y) = 2$, then either $y$ has Jordan form $[J_3]$ or $|y|$ is odd, so by \cite[Propositions 3.22 and 3.36]{B_fpr2},
	\begin{equation*}
	|y^T| > \frac{1}{2(3,q-\e)}\left(\frac{q}{q+1}\right)q^6 > (q^2-1)(q^2+\e q+1) = |x^T|.
	\end{equation*}
	If $\nu(y) = 1$ and $y$ is semisimple, then a pre-image $\widehat{y}$ of $y$ in $\GL(V)$ is $[\omega I_1, I_{2}]$, so
	\begin{equation*}
	|C_T(y)| = (3,q-\e)^{-1}|\GL_{2}^\e(q)|
	\end{equation*}
	It is easy to see that $|C_T(y)| < |C_T(x)|$. If $y$ is a graph automorphism, then $|C_{\PGL_3^\e(q)}(y)| = |\SL_2(q)|$, so $|C_T(y)|<|C_T(x)|$ evidently. If $y$ is a field automorphism of odd prime order $r$, then by \cite[Propositions 3.2.9 and 3.3.12]{BG_classical},
	\begin{equation*}
	|C_{\PGL_3^\e(q)}(y)| =|\PGL_3^\e(q^{1/r})| \leqs  q(q^{2/3}-1)(q-\e),
	\end{equation*}
	so $|C_T(y)|\leqs |C_{\PGL_3^\e(q)}(y)| < |C_T(x)|$. Thus, we only need to consider involutory field or graph-field automorphisms, so we can assume $\e = +$ and $f$ is even. Let $y_1$ be an involutory field automorphism. Then by \cite[Proposition 3.2.9]{BG_classical},
	\begin{equation*}
	|C_T(y_1)| = \frac{(3,q^{1/2}+1)}{(3,q-1)}|\PGL_3(q^{1/2})|.
	\end{equation*}
	Similarly, if $y_2$ is a graph-field automorphism, then
	\begin{equation*}
	|C_T(y_2)| = \frac{(3,q^{1/2}-1)}{(3,q-1)}|\PGU_3(q^{1/2})|
	\end{equation*}
	by \cite[Proposition 3.2.15]{BG_classical}. Note that
	\begin{equation*}
	|\PGL_3(q^{1/2})| < q^3(q-1) < |\PGU_3(q^{1/2})| < 3|\PGL_3(q^{1/2})|.
	\end{equation*}
	Therefore, $h(T) = |C_T(x)|$ if $f$ is odd or $\e = -$, $h(T) = |C_T(y_1)|$ if $\e = +$, $f$ is even and $3\mid q^{1/2}+1$, otherwise $h(T) = |C_T(y_2)|$.
	
	Next, assume $T = \LL_4^\e(q)$ and let $z$ be a graph automorphism of type $\gamma_1$ (see \cite[Sections 3.2.5 and 3.3.5]{BG_classical}), so by \cite[Propositions 3.2.14 and 3.3.17]{BG_classical}, we have
	\begin{equation*}
	|C_T(z)| = \frac{(2,q-\e)}{(4,q-\e)}|\mathrm{PGSp}_4(q)| > \frac{1}{(4,q-\e)}q^6(q^2-1)(q-\e) = |C_T(x)|
	\end{equation*}
	and we claim that $h(T) = |C_T(z)|$. Note that
	\begin{equation*}
	|z^T| = \frac{q^2(q^3-\e)}{(2,q-\e)}.
	\end{equation*}
	By \cite[Propositions 3.22, 3.36, 3.37 and 3.48]{B_fpr2}, we have
	\begin{equation*}
	|y^T| > \frac{1}{2}\left(\frac{q}{q+1}\right)q^6
	\end{equation*}
	for any unipotent, semisimple, field or graph-field element $y\in\Aut(T)$ of prime order. Hence, $|y^T| > |z^T|$ if $q\geqs 4$, and for $q\in\{2,3\}$ we can check that $|y^T| > |z^T|$ using {\sc Magma}. Similarly, if $y$ is a graph automorphism, then $|y^T|\geqs |z^T|$ by inspecting \cite[Tables B.3 and B.4]{BG_classical}.
	
	Finally, assume $T = \LL_n^\e(q)$ and $n\geqs 5$. Then by applying the bounds in \cite[Table 3.11]{B_fpr2} we see that
	\begin{equation*}
	|y^T| > \frac{1}{2}\left(\frac{q}{q+1}\right)^{\frac{1}{2}(1-\e)}q^{\frac{1}{2}(n^2-n-4)} > \frac{2q^{2n-1}}{q-1} > |x^T|
	\end{equation*}
	if $y$ is a field, graph or graph-field automorphism, unless $(n,q) = (5,2)$ or $(6,2)$, in which cases one can check that $|y^T|>|x^T|$ using {\sc Magma}. If $y$ is a unipotent or semisimple element with $\nu(y)\geqs 2$, then
	\begin{equation*}
	|y^T| > \frac{1}{2}\left(\frac{q}{q+1}\right)q^{4n-8} > \frac{2q^{2n-1}}{q-1} > |x^T|
	\end{equation*}
	by \cite[Proposition 3.36]{B_fpr2}.
	Thus, we only need to consider the cases where $\nu(y) = 1$ and $y$ is not $\Aut(T)$-conjugate to $x$. In this setting, $y$ is semisimple, and a pre-image $\widehat{y}$ of $y$ in $\GL(V)$ is $[\omega I_1,I_{n-1}]$, where $\omega$ is a non-trivial $r$-th root of unity in $\mathbb{F}_q$ if $\e = +$, or $\mathbb{F}_{q^2}$ if $\e = -$, for some prime $r$. It follows that
	\begin{equation*}
	|C_T(y)| = (n,q-\e)^{-1}|\GL_{n-1}^\e(q)|.
	\end{equation*}
	Note that $|C_T(y)| > |C_T(x)|$ if and only if $\e = -$ and $n$ is even. This implies that 
	\begin{equation*}
	h(T) = (n,q-\e)^{-1}|\GL_{n-1}^\e(q)|
	\end{equation*}
	if $\e = -$ and $n$ is even, otherwise $h(T) = |C_T(x)|$.

	This concludes the proof of Theorem \ref{t:fixity_Hol} for linear and unitary groups. We can use a very similar approach to handle the symplectic and orthogonal groups and we omit the details. But let us remark that if $T = \PSp_n(q)$ is a symplectic group, then $|C_T(x)|$ is maximal when $x$ is a long root element, unless $n = 4$ and $q$ is odd, where an involution of type $t_1$ gives the maximal centraliser. If $T = \POmega_n^\e(q)$, where $n$ is odd or $q$ is even, then $|C_T(x)|$ is maximal when $x$ is an involution of type $t_1'$ or $b_1$, respectively. Finally, if $T = \POmega_n^\e(q)$ with $n$ even and $q$ odd, then a graph automorphism of type $\gamma_1$ has the maximal centraliser. All the relevant information about these elements can be found in \cite[Chapter 3]{BG_classical}.
\end{proof}

An immediate corollary is the following, which will be useful in Section \ref{s:prob}.

\begin{cor}
	\label{c:h(T)_le_|T|/10}
	We have $h(T)\leqs |T|/10$ for any non-abelian finite simple group $T$.
\end{cor}

\subsection{Holomorph of simple groups}

Recall that $\Hol(T) = T{:}\Aut(T)$ is the \textit{holomorph} of $T$, which acts faithfully and primitively on $T$ (in fact, $\Hol(T) = T^2.\Out(T)$ is a diagonal type primitive group). Note that every element in $\Hol(T)$ can be uniquely written as $g\a$, where $g\in T$ acts on $T$ by left translation and $\a\in\Aut(T)$ acts naturally on $T$. That is,
\begin{equation*}
t^{g\a} = (g^{-1}t)^\a
\end{equation*}
for every $t\in T$. Let $\Hol(T,S)$ be the setwise stabiliser of $S\subseteq T$ in $\Hol(T)$. Throughout this section, we assume $P = S_k$, so $W = T^k.(\Out(T)\times S_k)$. The following result is a key observation.

\begin{lem}
	\label{l:Sk_eq_Hol}
	The following statements are equivalent.
	\begin{enumerate}\addtolength{\itemsep}{0.2\baselineskip}
		\item[\rm (i)] $\{D,D(\varphi_{t_1},\dots,\varphi_{t_k})\}$ is a base for $W$;
		\item[\rm (ii)] $t_1,\dots,t_k$ are distinct and $\Hol(T,\{t_1,\dots,t_k\}) = 1$.
	\end{enumerate}
\end{lem}

\begin{proof}
	First assume (i) holds. If $t_i = t_j$ for some $i\ne j$, then $(i,j)\in W$ stabilises $D$ and $D(\varphi_{t_1},\dots,\varphi_{t_{k}})$, which is incompatible with (i). Thus, $t_1,\dots,t_k$ are distinct. Suppose $g\alpha\in \Hol(T,\{t_1,\dots,t_k\})$. Then for any $i$ we have
	\begin{equation}\label{e:Sk_eq_Hol}
	t_j = t_i^{g\alpha} = (g^{-1}t_i)^\alpha = (g^{-1})^\alpha t_i^\alpha
	\end{equation}
	for some $j$. That is, $g\alpha$ induces a permutation $\pi\in S_k$ by $(g^{-1})^\alpha t_i^\alpha = t_{i^\pi}$. Now it is easy to see that $(\alpha,\dots,\alpha)\pi$ fixes $D(\varphi_{t_1},\dots,\varphi_{t_k})$. Hence, $\a = 1$ and $\pi = 1$, which implies that $g = 1$ by \eqref{e:Sk_eq_Hol}, noting that $i = j$ since $\pi = 1$.
	
	Conversely, suppose (ii) holds and $(\alpha,\dots,\alpha)\pi$ fixes $D$ and $D(\varphi_{t_1},\dots,\varphi_{t_k})$. Then there exists $g\in T$ such that $t_{i^\pi}= g^{-1}t_i^\alpha$ for all $i$. It follows that $g^{\a^{-1}}\alpha\in\Hol(T,\{t_1,\dots,t_k\})$, which implies that $g=1$ and $\alpha = 1$. As $t_1,\dots,t_k$ are distinct, this gives $\pi = 1$ and so (i) holds.
\end{proof}

Let $\mathscr{P}_k(T)$ (or just $\mathscr{P}_k$ if $T$ is clear from the context) be the set of $k$-subsets of $T$. Recall that $r(G)$ is the number of regular suborbits of $G$.

\begin{lem}
	\label{l:Sk_eq_Hol_r}
	The number of regular orbits of $\Hol(T)$ on $\mathscr{P}_k$ or $\mathscr{P}_{|T|-k}$ is $r(W)$. In particular, $b(W) = 2$ if and only if $\Hol(T)$ has a regular orbit on $\mathscr{P}_k$ or $\mathscr{P}_{|T|-k}$.
\end{lem}

\begin{proof}
	This follows directly from Lemma \ref{l:Sk_eq_Hol}, noting that $\Hol(T,S) = \Hol(T,T\setminus S)$.
\end{proof}

Given a subset $S\subseteq T$, it is difficult to determine $\Hol(T,S)$. In particular, it is difficult to construct a subset $S\subseteq T$ such that $\Hol(T,S) = 1$. By the transitivity of $\Hol(T)$ on $T$, we may assume $1\in S$.

\begin{lem}
	\label{l:translation}
	Let $S_1$ and $S_2$ be subsets of $T$ such that $1\in S_1\cap S_2$ and $S_1^{g\a} = S_2$. Then $g\in S_1$.
\end{lem}

\begin{proof}
	We have $g^{-1}S_1 = S_2^{\a^{-1}}$, so $1\in g^{-1}S_1$ and thus $g\in S_1$.
\end{proof}

Now we give some sufficient conditions that allow us to deduce that $\Hol(T,S) = 1$ for a subset $S\subseteq T$ containing $1$. Here we write $\Aut(T,R)$ for the setwise stabiliser of $R\subseteq T^\#$ in $\Aut(T)$.

\begin{lem}
	\label{l:Hol(T,S)=1_cond}
	Let $S = \{t_1,\dots,t_k\}\in\mathscr{P}_k$ with $t_1 = 1$. Then $\Hol(T,S) = 1$ if the following conditions are satisfied:
	\begin{enumerate}\addtolength{\itemsep}{0.2\baselineskip}
		\item[\rm (i)] $\Aut(T,\{t_2,\dots,t_k\}) = 1$; and
		\item[\rm (ii)] for all $2\leqs i\leqs k$, $\{|t_i^{-1}t_1|,\dots,|t_i^{-1}t_k|\}\ne\{1,|t_2|,\dots,|t_k|\}$.
	\end{enumerate}
\end{lem}

\begin{proof}
	Suppose $g\alpha\in\Hol(T,S)$, where $g\in T$ and $\alpha\in\Aut(T)$. By Lemma \ref{l:translation}, we have $g\in S$. If $g = t_1 = 1$ then $\alpha\in\Aut(T,\{t_2,\dots,t_k\})$ and the condition (i) forces $\alpha = 1$. If $g = t_i$ for some $2\leqs i\leqs k$ then $t_i^{-1}S = S^{\alpha^{-1}}$, which implies that $\{|t_i^{-1}t_1|,\dots,|t_i^{-1}t_k|\}=\{1,|t_2|,\dots,|t_k|\}$, which is incompatible with the condition (ii).
\end{proof}

\begin{cor}
	\label{c:Hol(T,S)=1_cond}
	Let $S = \{t_1,\dots,t_k\}\in\mathscr{P}_k$ with $t_1 = 1$. If $\Out(T) = 1$ then $\Hol(T,S) = 1$ if all the following conditions are satisfied:
	\begin{enumerate}\addtolength{\itemsep}{0.2\baselineskip}
		\item[\rm (i)] $t_2,\dots,t_k$ have distinct orders;
		\item[\rm (ii)] $M = \la t_2,\dots,t_k\ra$ is a maximal subgroup of $T$ such that $Z(M) = 1$;
		\item[\rm (iii)] for all $2\leqs i\leqs k$, $\{|t_i^{-1}t_1|,\dots,|t_i^{-1}t_k|\}\ne\{1,|t_2|,\dots,|t_k|\}$.
	\end{enumerate}
\end{cor}

\begin{proof}
	In view of Lemma \ref{l:Hol(T,S)=1_cond}, it suffices to show that the conditions (i) and (ii) imply that $\Aut(T,\{t_2,\dots,t_k\}) = 1$. Suppose $\alpha\in\Aut(T,\{t_2,\dots,t_k\})$. Then $\alpha\in C_{\Aut(T)}(t_i)$ for each $i$, as $t_2,\dots,t_k$ have distinct orders. It follows that $\alpha$ centralises $\la t_2,\dots,t_k\ra = M$ and so $\alpha\in C_{\Aut(T)}(M)$. Since $\Out(T) = 1$, this implies that $\alpha\in C_T(M)\leqs N_T(M) = M$ since $M$ is maximal, so $\alpha\in Z(M) = 1$. This completes the proof.
\end{proof}

\begin{lem}
	\label{l:r_ge_2_cond}
	Let $S_1 = \{t_1,\dots,t_k\}$ and $S_2 = \{s_1,\dots,s_k\}$ be elements in $\mathscr{P}_k$ such that $1\in S_1\cap S_2$ and $\Hol(T,S_j) = 1$ for each $j\in\{1,2\}$. Then $S_1$ and $S_2$ are in distinct $\Hol(T)$-orbits if
	\begin{equation*}
	\{|t_i^{-1}t_1|,\dots,|t_i^{-1}t_k|\}\ne \{|s_1|,\dots,|s_k|\}
	\end{equation*}
	for any $i\in[k]$.
\end{lem}

\begin{proof}
	This follows immediately from Lemma \ref{l:translation}.
\end{proof}

\begin{rem}
	\label{r:magma}
	Let us briefly discuss the main computational techniques we will use to prove $r(W)\geqs 2$ for some suitable $T$ and $k$.
	\begin{enumerate}\addtolength{\itemsep}{0.2\baselineskip}
		\item[\rm (i)] Let $S_1$ and $S_2$ be $k$-element subsets of $T$ containing $1$, and let $O_j = \{|t|:t\in S_j\}$. Assume that $|O_j| = k$, $\la S_j\ra = T$ and
		\begin{equation*}
		O_j\ne \{|x^{-1}t|:t\in S_j\}
		\end{equation*}
		for any $x\in S_j\setminus\{1\}$. Then $\Hol(T,S_j) = 1$ by Lemma \ref{l:Hol(T,S)=1_cond}, noting that the first two conditions imply that $\Aut(T,S_j\setminus\{1\}) = 1$. Combining Lemmas \ref{l:Sk_eq_Hol_r} and \ref{l:r_ge_2_cond}, we have $r(W)\geqs 2$ if
		\begin{equation*}
		O_2\ne \{|x^{-1}t|:t\in S_1\}
		\end{equation*}
		for any $x\in S_1$. For suitable $T$ and $k$, we can construct $T$ with an appropriate permutation representation in {\sc Magma}, and implement this approach to find $k$-subsets $S_1$ and $S_2$ of $T$ with these properties by random search. We will only need to use this method for $k\leqs 11$.
		\item[{\rm (ii)}]
		In some cases where $\Out(T) = 1$, we will work with a centreless maximal subgroup $M$ of $T$, rather than $T$ itself. More precisely, if $S_1$ and $S_2$ are $k$-element subsets of $M$ containing $1$ and $O_j = \{|t|:t\in S_j\}$, then by Corollary \ref{c:Hol(T,S)=1_cond}, we have $\Hol(T,S_j) = 1$ if $|O_j| = k$, $\la S_j\ra  = M$ and
		\begin{equation*}
		O_j \ne \{|x^{-1}t|:t\in S_j\}
		\end{equation*}
		for any $S_j\setminus\{1\}$. Again, by Lemmas \ref{l:Sk_eq_Hol_r} and \ref{l:r_ge_2_cond}, we have $r(W)\geqs 2$ if
		\begin{equation*}
		O_2\ne \{|x^{-1}t|:t\in S_1\}
		\end{equation*}
		for any $x\in S_1$.	For example, if $T = \mathbb{M}$ is the Monster sporadic group and $3\leqs k \leqs 5$, then we will work with a maximal subgroup $M$ of $T$ isomorphic to $\LL_2(71)$ (this case arises in the proofs of Lemma \ref{l:sporadic_k=3,4} and Proposition \ref{p:sporadic}).
	\end{enumerate}
\end{rem}

\section{Probabilistic methods}

\label{s:prob}

In this section, we assume $G = T^k.(\Out(T)\times S_k)$ with $2<k<|T|$. By Lemma \ref{l:Sk_eq_Hol_r}, we have $r(G) \geqs  2$ for $k=m$ if and only if $r(G) \geqs 2$ for $k = |T|-m$, so we will assume $5\leqs k\leqs |T|/2$ throughout this section (we will treat the cases where $k\in\{3,4\}$ separately in Section \ref{s:proof_b(G)=2}).

In Section \ref{ss:fix_subsets}, we will estimate the probability $\mathrm{Pr}_k(T)$ that a random $k$-subset of $T$ has non-trivial setwise stabiliser in $\Hol(T)$, noting that
\begin{equation}\label{e:prob_k}
\mathrm{Pr}_k(T) = \frac{|\{S\in\mathscr{P}_k:\Hol(T,S) \ne 1\}|}{{|T|\choose k}}.
\end{equation}
As noted above, we have $r(G)\geqs 2$ if and only if
\begin{equation}\label{e:prob_k_ineq}
\mathrm{Pr}_k(T) < 1-\frac{|\Hol(T)|}{{|T|\choose k}}.
\end{equation}
To establish this inequality, we will give upper bounds on $\mathrm{Pr}_k(T)$ in Section \ref{ss:fix_subsets}. In particular, we will show that $r(G)\geqs 2$ if $4\log|T| < k\leqs |T|/2$ (see Proposition \ref{p:log}).

Finally, to handle certain cases where $k$ is small, we will consider the probability that a random pair of elements in $\Omega$ is not a base for $G$ in Section \ref{ss:fpr}, which is a widely used method in the study of base sizes.

\subsection{Holomorph and subsets}

\label{ss:fix_subsets}

We first consider $\mathrm{Pr}_k(T)$ defined as in \eqref{e:prob_k}. Let $\mathcal{F} = \{S\in\mathscr{P}_k:\Hol(T,S)\ne 1\}$ and suppose $S\in\mathcal{F}$. Then there exists $\sigma\in\Hol(T,S)$ of prime order. In other words, $S\in\fix(\sigma,\mathscr{P}_k)$, where
\begin{equation*}
\fix(\sigma,\mathscr{P}_k) = \{S\in\mathscr{P}_k:\sigma\in\Hol(T,S)\}
\end{equation*}
is the set of fixed points of $\sigma$ on $\mathscr{P}_k$. It follows that
\begin{equation*}
|\mathcal{F}| = \left|\bigcup_{\sigma\in\mathcal{R}}\fix(\sigma,\mathscr{P}_k)\right|\leqs\sum_{\sigma\in\mathcal{R}}|\fix(\sigma,\mathscr{P}_k)|,
\end{equation*}
where $\mathcal{R}$ is the set of elements of prime order in $\Hol(T)$. As discussed above, we have $r(G)\geqs 2$ if and only if \eqref{e:prob_k_ineq} holds. Thus, $r(G)\geqs 2$ if
\begin{equation*}
\sum_{\sigma\in\mathcal{R}}|\fix(\sigma,\mathscr{P}_k)| < {|T|\choose k} - |\Hol(T)|.
\end{equation*}
Moreover, since $5\leqs k\leqs |T|/2$, we note that $|\Hol(T)| < \frac{1}{2}{|T|\choose k}$ by Lemma \ref{l:|Out(T)|}. This observation yields the following result.

\begin{lem}
	\label{l:prob_ori}
	We have $r(G)\geqs 2$, and hence $b(G) =2$, if
	\begin{equation}
	\label{e:prob_ori}
	{|T|\choose k} > 2\sum_{\sigma\in\mathcal{R}}|\fix(\sigma,\mathscr{P}_k)|.
	\end{equation}
\end{lem}

In order to apply Lemma \ref{l:prob_ori}, we need to derive a suitable upper bound for the summation appearing on the right-hand side of \eqref{e:prob_ori}.

\begin{lem}
	\label{l:fix}
	Let $\sigma\in\Hol(T)$ be of prime order $r$ with cycle shape $[r^m,1^{|T|-mr}]$. Then
	\begin{equation*}
	|\fix(\sigma,\mathscr{P}_k)| = \sum_{u=0}^{\lfloor k/r\rfloor}{m\choose u}{|T|-mr\choose k-ru}.
	\end{equation*}
\end{lem}

\begin{proof}
	This follows by noting that any subset fixed by $\sigma$ is a union of some cycles comprising $\sigma$.
\end{proof}

If $\sigma\in\Hol(T)$ is an element as described in Lemma \ref{l:fix}, then $|T|-mr$ is the number of elements in $T$ fixed under $\sigma$. It follows that $|T|-mr\leqs \fix(\Hol(T))$, where $\fix(\Hol(T))$ is the fixity of $\Hol(T)$ (the \textit{fixity} of a permutation group is the maximum number of elements fixed by a non-identity permutation). Recall that 
\begin{equation*}
h(T) = \max\{|C_T(x)|:1\ne x\in\Aut(T)\},
\end{equation*}
which has been determined in Theorem \ref{t:fixity_Hol}.

\begin{lem}
	\label{l:mu}
	We have $\fix(\Hol(T)) = h(T)$.
\end{lem}

\begin{proof}
	Let $\sigma\in\Hol(T)$ be such that it fixes at least one element in $T$. We may assume $\sigma$ fixes $1\in T$ by the transitivity of $\Hol(T)$. Thus, $\sigma\in\Aut(T)$ and hence $C_T(\sigma)$ is the set of fixed points of $\sigma$, which completes the proof.
\end{proof}

\begin{cor}
	\label{c:fix}
	Let $\sigma\in\Hol(T)$ be of prime order $r$. Then
	\begin{equation*}
	|\fix(\sigma,\mathscr{P}_k)| \leqs  \sum_{u=0}^{\lfloor k/r\rfloor}{|T|/r\choose u}{h(T)\choose k-ru}.
	\end{equation*}
\end{cor}

The following bounds on binomial coefficients come from \cite[Theorem 2.6]{S_binomial}, where $e$ is the exponential constant.

\begin{lem}
	\label{l:bound_binomial_better}
	Let $\ell, m, n$ be positive integers with $n>m$. Then
	$$e^{-\frac{1}{8\ell}}a(\ell,m,n)<{n\ell\choose m\ell}<a(\ell,m,n),$$where
	\begin{equation*}
	a(\ell,m,n) = \frac{1}{\sqrt{2\pi}}\ell^{-\frac{1}{2}}\left(\frac{n}{(n-m)m}\right)^{\frac{1}{2}}\left(\frac{n^n}{(n-m)^{n-m}m^m}\right)^\ell.
	\end{equation*}
\end{lem}

\begin{cor}
	\label{c:binomial_bound}
	Suppose $n = tm$ for some integer $t\geqs 2$. Then
	\begin{equation}\label{e:bin_bound}
	e^{-\frac{1}{8}}\left(\frac{t^2}{(t-1)n}\right)^{\frac{1}{2}}\left(\frac{t^t}{(t-1)^{t-1}}\right)^{\frac{n}{t}}<\sqrt{2\pi}{n\choose m}<\left(\frac{t^2}{(t-1)n}\right)^{\frac{1}{2}}\left(\frac{t^t}{(t-1)^{t-1}}\right)^{\frac{n}{t}}.
	\end{equation}
\end{cor}

\begin{proof}
	Put $\ell = 1$ and $m = n/t$ in Lemma \ref{l:bound_binomial_better}.
\end{proof}

\begin{prop}
	\label{p:log}
	If $4\log |T| < k\leqs |T|/2$, then $r(G)\geqs 2$. In particular, $b(G) = 2$.
\end{prop}

\begin{proof}
	First, if $T = A_5$, then we construct the permutation group $\Hol(T)$ on $T$ using the function \texttt{Holomorph} in {\sc Magma}. Then we find two random $k$-subsets of $T$ lying in distinct regular $\Hol(T)$-orbits by random search.
	
	Hence, we may assume $|T|\geqs 168$ and thus $4\log|T| < |T|/4$. First assume $|T|/4\leqs k\leqs |T|/2$. By Corollary \ref{c:fix}, we have
	\begin{equation*}
	|\fix(\sigma,\mathscr{P}_k)| \leqs \sum_{u=0}^{\lfloor k/r\rfloor}{|T|/r\choose u}{h(T)\choose \lfloor h(T)/2\rfloor}\leqs  2^{|T|/r}{h(T)\choose \lfloor h(T)/2\rfloor}\leqs  2^{|T|/2}{h(T)\choose \lfloor h(T)/2\rfloor}
	\end{equation*}
	for every element $\sigma\in\Hol(T)$ of prime order. Hence, \eqref{e:prob_ori} holds if
	\begin{equation}\label{e:|T|/4}
	{|T|\choose k}>|\Hol(T)|2^{|T|/2+1}{h(T)\choose \lfloor h(T)/2\rfloor},
	\end{equation}
	and it suffices to consider $k = |T|/4$. Now we apply \eqref{e:bin_bound}, which gives
	\begin{equation*}
	{|T|\choose |T|/4}>\frac{1}{\sqrt{2\pi}}e^{-\frac{1}{8}}\frac{4}{\sqrt{3|T|}}\left(\frac{4}{3^{3/4}}\right)^{|T|}
	\end{equation*}
	and
	\begin{equation*}
	{h(T)\choose \lfloor h(T)/2\rfloor} < \frac{1}{\sqrt{2\pi}}\cdot \sqrt{\frac{4}{ h(T)}}\cdot 2^{ h(T)} \leqs \frac{1}{\sqrt{2\pi}}\cdot \sqrt{\frac{40}{|T|}}\cdot 2^{|T|/10}
	\end{equation*}
	as $h(T)\leqs |T|/10$ by Corollary \ref{c:h(T)_le_|T|/10}.	Combining the inequalities above, we see that \eqref{e:|T|/4} holds for $k = |T|/4$ if
	\begin{equation*}
	\frac{1}{\sqrt{2\pi}}e^{-\frac{1}{8}}\frac{4}{\sqrt{3|T|}}\left(\frac{4}{3^{3/4}}\right)^{|T|}>|\Hol(T)|\cdot 2^{|T|/2+1}\cdot \frac{1}{\sqrt{2\pi}}\cdot \sqrt{\frac{40}{|T|}}\cdot 2^{|T|/10}.
	\end{equation*}
	Finally, since $|\Out(T)| < |T|^{1/3}$ by Lemma \ref{l:|Out(T)|}, it suffices to show that
	\begin{equation}\label{e:t_0}
	t_0^{|T|}>\sqrt{30}e^{\frac{1}{8}}|T|^{\frac{7}{3}},
	\end{equation}
	where 
	\begin{equation*}
	t_0 = 4\cdot 3^{-\frac{3}{4}}\cdot 2^{-\frac{1}{2}-\frac{1}{10}} = 1.1577....
	\end{equation*}
	and it is easy to check that the inequality in \eqref{e:t_0} holds for all $|T|\geqs 168$.
	
	Now assume $4\log |T| < k < |T|/4$ and let $\sigma\in\Hol(T)$ be of prime order $r$. Observe that we have $ru\leqs k< |T|/4$ for all $u\in\{0,\dots,\lfloor k/r\rfloor\}$, so
	\begin{equation*}
	\begin{aligned}
	\sum_{u=0}^{\lfloor k/r\rfloor}{|T|/r\choose u}{ h(T)\choose k-ru}&<\sum_{u=0}^{\lfloor k/r\rfloor}{|T|/2\choose u}{ h(T)\choose k-ru}\\&<\sum_{u=0}^{\lfloor k/r\rfloor}{|T|/2\choose ru}{ h(T)\choose k-ru} \\&< {|T|/2+ h(T)\choose k},
	\end{aligned}
	\end{equation*}
	noting that the third inequality follows from the Vandermonde's identity. Thus, \eqref{e:prob_ori} holds if
	\begin{equation}\label{e:log}
	{|T|\choose k}>2|\Hol(T)|{|T|/2+ h(T)\choose k}.
	\end{equation}
	It is easy to see that \eqref{e:log} is equivalent to
	\begin{equation*}
	\frac{|T|!}{(|T|-k)!}>2|\Hol(T)|\frac{(|T|/2+ h(T))!}{(|T|/2+ h(T)-k)!}.
	\end{equation*}
	Now
	\begin{equation*}
	\frac{|T|-m}{|T|/2+ h(T)-m}\geqs \frac{|T|}{|T|/2+ h(T)} =: t
	\end{equation*}
	for every $m\in\{0,\dots,k-1\}$ and thus \eqref{e:log} holds if $t^k>2|\Hol(T)|$. By Corollary \ref{c:h(T)_le_|T|/10}, we have $|T|/ h(T)\geqs 10$, and hence $t\geqs 5/3$. Therefore, \eqref{e:log} holds if $(5/3)^k > |T|^{8/3}$ (by applying Lemma \ref{l:|Out(T)|}), which implies the desired result.
\end{proof}

Now we turn to the cases where $5\leqs k\leqs 4\log|T|$. We will give some sufficient conditions for $r(G)\geqs 2$.

\begin{lem}
	\label{l:prob}
	Suppose $5\leqs k\leqs 4\log|T|$. Then $r(G)\geqs 2$, and hence $b(G) = 2$, if
	\begin{equation}\label{e:prob}
	{|T|\choose k}>2|\Hol(T)|\sum_{u=0}^{\lfloor k/2\rfloor}{|T|/2\choose u}{h(T)\choose k-2u}.
	\end{equation}
\end{lem}

\begin{proof}
	If $8\log |T| < h(T)$, then $k < h(T)/2$ and \eqref{e:prob_ori} follows via \eqref{e:prob} and Corollary \ref{c:fix}. By inspecting Table \ref{tab:fix(Hol(T))}, we see that $8\log |T|\geqs h(T)$ only if $T$ is isomorphic to one of the following groups:
	\begin{equation}\label{e:k/r<k/2_rem}
	\mathrm{M}_{11},\ \mathrm{J}_1,\ {^2}B_2(8),\ \LL_3(3),\ \LL_2(q)\ (q\leqs 167).
	\end{equation}
	
	Assume $T$ is one of the groups in \eqref{e:k/r<k/2_rem} and suppose $\sigma\in\Hol(T)$ has prime order $r$. We claim that
	\begin{equation}
	\label{e:k/r<k/2}
	|\fix(\sigma,\mathscr{P}_k)| < \sum_{u=0}^{\lfloor k/2\rfloor}{|T|/2\choose u}{h(T)\choose k-2u}.
	\end{equation}
	To see this, first assume $\sigma$ is fixed-point-free on $T$. Here $|\fix(\sigma,\mathscr{P}_k)| = 0$ if $r\nmid k$, and
	\begin{equation*}
	|\fix(\sigma,\mathscr{P}_k)| = {|T|/r\choose k/r}
	\end{equation*}
	otherwise. In particular, the inequality in \eqref{e:k/r<k/2} holds. Now assume $\sigma$ has a fixed point on $T$. Since $\sigma$ is conjugate to an element fixing the identity element in $T$, we may assume $\sigma\in\Aut(T)$. Then with the aid of {\sc Magma} and Corollary \ref{c:fix}, it is easy to check that \eqref{e:k/r<k/2} holds when $T$ is one of the groups in \eqref{e:k/r<k/2_rem}.
	
	We conclude that the proof is complete by combining \eqref{e:prob} and \eqref{e:k/r<k/2} with Lemma \ref{l:prob_ori}.
\end{proof}

\begin{lem}
	\label{l:prob_u_weak}
	The inequality \eqref{e:prob} holds if
	\begin{equation}
	\label{e:prob_u_weak}
	2^uu^u|T|^{k-u} > 2|\Hol(T)|\lfloor k/2\rfloor k^{2u}e^{k+u} h(T)^{k-2u}
	\end{equation}
	for every $u\in\{0,\dots,\lfloor k/2\rfloor\}$, where we define $u^u = 1$ if $u=0$.
\end{lem}

\begin{proof}
	First observe that \eqref{e:prob} holds if
	\begin{equation}\label{e:prob_u}
	{|T|\choose k}>2|\Hol(T)|\lfloor k/2\rfloor{|T|/2\choose u}{ h(T)\choose k-2u}
	\end{equation}
	for every $u\in\{0,\dots,\lfloor k/2\rfloor\}$. Now
	\begin{equation*}
	\left(\frac{k}{k-2u}\right)^{k-2u} < e^{2u}
	\end{equation*}
	for all such $u$. Therefore, \eqref{e:prob_u} follows by combining \eqref{e:prob_u_weak} and the well-known bounds on binomial coefficients
	\begin{equation*}
	\frac{n^m}{m^m}<{n\choose m}<\frac{(en)^m}{m^m}
	\end{equation*}
	for any integers $n\geqs m\geqs 0$, where we define $m^m = 1$ if $m = 0$.
\end{proof}

We conclude this section by establishing two more technical lemmas, which will play a key role in Section \ref{s:proof_b(G)=2}.

\begin{lem}
	\label{l:u=k/2}
	Suppose $|T| > 4080$ and $5\leqs k\leqs 4\log|T|$. Then \eqref{e:prob} holds if there exists an integer $k_0$ such that $5\leqs k_0\leqs k$,
	\begin{equation}
	\label{e:prob_u=k/2_weak}
	|T|^{k_0} > |\Hol(T)|^2 k_0^{2+k_0}e^{3k_0}
	\end{equation}
	and
	\begin{equation}
	\label{e:prob_u=k/2_weak_cond}
	h(T)^2 < k_0|T|.
	\end{equation}
\end{lem}

\begin{proof}
	We first prove that \eqref{e:prob} holds if $k = k_0$. In view of Lemma \ref{l:prob_u_weak}, it suffices to verify the inequality in \eqref{e:prob_u_weak} for all $u\in\{0,\dots,\lfloor k/2\rfloor\}$ and we will do this by induction. First assume $u = \lfloor k/2\rfloor$ and note that \eqref{e:prob_u=k/2_weak} is equivalent to \eqref{e:prob_u_weak} if $k$ is even. For $k$ odd we have $u = (k-1)/2$ and the inequality in \eqref{e:prob_u_weak} is
	\begin{equation}\label{e:prob_u=k/2_weak_odd}
	\left(\frac{|T|(k-1)}{k^2e^3}\right)^{k}|T| > \frac{k-1}{k^2e}\cdot 4|\Hol(T)|^2\left(\frac{k-1}{2}\right)^2h(T)^2.
	\end{equation}
	In view of \eqref{e:prob_u=k/2_weak_cond}, we see that \eqref{e:prob_u=k/2_weak_odd} holds if
	\begin{equation*}
	\left(\frac{|T|}{ke^3}\right)^{k}\left(\frac{k-1}{k}\right)^{k-1}e > k^2|\Hol(T)|^2,
	\end{equation*}
	which is implied by \eqref{e:prob_u=k/2_weak} since $(\frac{k-1}{k})^{k-1}>e^{-1}$. Therefore, \eqref{e:prob_u_weak} holds for $u = \lfloor k/2\rfloor$ and we have established the base case for the induction. Now suppose \eqref{e:prob_u_weak} holds for $u=u_0$, where $1\leqs u_0\leqs \lfloor k/2\rfloor$. It suffices to show that \eqref{e:prob_u_weak} holds for $u = u_0-1$. Here the desired inequality holds if
	\begin{equation*}
	2^{-1}|T|\cdot\frac{(u_0-1)^{u_0-1}}{u_0^{u_0}} > k^{-2}e^{-1}\cdot h(T)^2,
	\end{equation*}
	but this is implied by \eqref{e:prob_u=k/2_weak_cond}, noting that $(\frac{u_0-1}{u_0})^{u_0-1} > e^{-1}$ and $2u_0\leqs k$. In conclusion, if $k = k_0$ then \eqref{e:prob_u_weak} holds for all $u\in\{0,\dots,\lfloor k/2\rfloor\}$ and thus \eqref{e:prob} holds by Lemma \ref{l:prob_u_weak}.
	
	Finally, we need to show that \eqref{e:prob} holds when $k_0<k$. By \eqref{e:prob_u=k/2_weak_cond} we have $h(T)^2 < k_0|T| < k|T|$, and by arguing as above, it suffices to show that
	\begin{equation}\label{e:prob_u=k/2_weak_ind}
	|T|^k > |\Hol(T)|^2k^{2+k}e^{3k}.
	\end{equation}
	Since $|T| > 4080$ and $5\leqs k\leqs 4\log|T|$, we get
	\begin{equation*}
	|T| > 2e^4(4\log|T| + 1) \geqs 2e^4(k+1) > \left(\frac{k+1}{k}\right)^{k+2}e^3(k+1).
	\end{equation*}
	Therefore, \eqref{e:prob_u=k/2_weak_ind} holds for all $k_0 \leqs k\leqs 4\log|T|$ by induction on $k$, and the proof is complete.
\end{proof}

\begin{lem}
	\label{l:u=0}
	Suppose $5\leqs  k\leqs 4\log|T|$. Then \eqref{e:prob} holds if there exists an integer $k_0$ such that $5\leqs k_0\leqs k$,
	\begin{equation}
	\label{e:prob_u=0_weak}
	|T|^{k_0}>2|\Hol(T)|\lfloor k_0/2\rfloor e^{k_0} h(T)^{k_0}
	\end{equation}
	and
	\begin{equation}
	\label{e:prob_u=0_weak_cond}
	2h(T)^2 > (4\log|T|)^2e|T|.
	\end{equation}
\end{lem}

\begin{proof}
	This is similar to the proof of Lemma \ref{l:u=k/2}, working with Lemma \ref{l:prob_u_weak} to establish the inequality in \eqref{e:prob}. First assume $k = k_0$ and note that \eqref{e:prob_u=0_weak} is equivalent to \eqref{e:prob_u_weak} with $u = 0$. We now use induction to show that \eqref{e:prob_u_weak} holds for all $u\in\{0,\dots,\lfloor k/2\rfloor\}$. To do this, suppose \eqref{e:prob_u_weak} holds for $u = u_0$, where $0\leqs u_0\leqs \lfloor k/2\rfloor-1$. Then \eqref{e:prob_u=0_weak_cond} implies that 
	\begin{equation*}
	2|T|^{-1}\cdot  \frac{(u_0+1)^{u_0+1}}{u_0^{u_0}}>k^2e\cdot h(T)^{-2},
	\end{equation*}
	and thus \eqref{e:prob_u_weak} holds for $u = u_0+1$ and the result follows.
	
	Finally, let us assume $k_0<k$. It suffices to show that
	\begin{equation*}
	|T|^{k} > 2|\Hol(T)|\lfloor k/2\rfloor e^{k}h(T)^k
	\end{equation*}
	for all $k_0\leqs k\leqs 4\log|T|$. This is clear by induction on $k$, since we have
	\begin{equation*}
	|T| > 2eh(T)
	\end{equation*}
	for every $T$ by Corollary \ref{c:h(T)_le_|T|/10}.
\end{proof}

\subsection{Fixed point ratios}

\label{ss:fpr}

Now we turn to another powerful probabilistic approach to study $b(G)$, where $G = T^k.(\Out(T)\times S_k)$, which was initially introduced by Liebeck and Shalev \cite{LS_prob}. Here we will estimate the probability $\mathbb{P}_k(T)$ that a random element in $\Omega$ is in a regular orbit of $G_D = D$, noting that $b(G) = 2$ if and only if $\mathbb{P}_k(T)>0$. Equivalently,
\begin{equation*}
\mathbb{P}_k(T) = \frac{r(G)|G|}{|T|^{2k-2}}
\end{equation*}
is the probability that a random pair of elements in $\Omega$ is a base for $G$.

Clearly, $\{\omega_1,\omega_2\}\subseteq\Omega$ is not a base for $G$ if and only if there exists an element $x\in G_{\omega_1}\cap G_{\omega_2}$ of prime order. Now the probability that $x\in G$ fixes a random element in $\Omega$ is given by the \textit{fixed point ratio}
\begin{equation*}
\fpr(x) = \frac{|\fix(x,\Omega)|}{|\Omega|} = \frac{|x^G\cap D|}{|x^G|},
\end{equation*}
where $\fix(x,\Omega)$ is the set of fixed points of $x$ on $\Omega$. Hence, we have
\begin{equation*}
1-\mathbb{P}_k(T)\leqs \sum_{x\in R(G)}|x^G|\cdot\fpr(x)^2 = \sum_{x\in R(G)}\frac{|x^G\cap D|^2|C_G(x)|}{|G|},
\end{equation*}
where $R(G)$ is the set of representatives for the $G$-conjugacy classes of elements in the stabiliser $D$ in $G$ which have prime order. We adopt the notation from \cite{F_diag} and define
\begin{equation*}
\begin{aligned}
R_1(G) &:= \{(\a,\dots,\a)\pi\in R(G):\pi\mbox{ is fixed-point-free on }[k]\},\\
R_2(G) &:=\{(\a,\dots,\a)\pi\in R(G):\pi = 1\},\\
R_3(G) &:=\{(\a,\dots,\a)\pi\in R(G):\pi\ne 1\mbox { and }\pi \mbox{ has a fixed point on }[k]\},
\end{aligned}
\end{equation*}
and 
\begin{equation*}
r_i(G) :=\sum_{x\in R_i(G)}\frac{| x^G\cap D|^2|C_G(x)|}{|G|}.
\end{equation*}
It follows that
\begin{equation}\label{e:F_prob_r}
1-\frac{r(G)|G|}{|T|^{2k-2}} = 1-\mathbb{P}_k(T) \leqs r_1(G)+r_2(G) + r_3(G),
\end{equation}
which gives a lower bound on $r(G)$. In particular, $b(G) = 2$ if $r_1(G)+r_2(G)+r_3(G) < 1$. Thus, we need to bound each $r_i(G)$ above.

\begin{lem}
	\label{l:r1}
	We have $r_1(G) < (k!)^2|T|^{8/3-\lceil k/2\rceil}$.
\end{lem}

\begin{proof}
	This is established in the proof of Theorem 1.5 in \cite{F_diag}.
\end{proof}

\begin{lem}
	\label{l:r2}
	We have $r_2(G) < (|T|/h(T))^{4-k}$.
\end{lem}

\begin{proof}
	Let $f_p(\Aut(T))$ be the number of conjugacy classes of elements of prime order in $\Aut(T)$. It follows from the proof of \cite[Lemma 4.2]{F_diag} that
	\begin{equation*}
	r_2(G)\leqs |\Out(T)|f_p(\Aut(T))\left(\frac{ h(T)}{|T|}\right)^{k-2}.
	\end{equation*}
	Thus, it suffices to show that
	\begin{equation}\label{e:r2_f}
	|\Out(T)|f_p(\Aut(T))<\left(\frac{|T|}{ h(T)}\right)^2.
	\end{equation}
	
	First assume $T = A_n$ is an alternating group. Then as discussed in the proof of \cite[Lemma 4.2]{F_diag}, we have $f_p(\Aut(T)) < \frac{n^2}{2}$. This implies \eqref{e:r2_f} since $h(T) = (n-2)!$ by Theorem \ref{t:fixity_Hol}.
	
	Next, assume $T$ is a sporadic group. Then $f_p(\Aut(T))$ can be read off from the character table of $\Aut(T)$ and it is easy to check that \eqref{e:r2_f} holds in every case.
	
	Finally, assume $T$ is a simple group of Lie type over $\mathbb{F}_q$. Let $f(T)$ be the number of conjugacy classes in $T$. As noted in \cite{G_conju}, we have $f_p(\Aut(T))\leqs |\Out(T)|f(T)$. Thus, it suffices to show that
	\begin{equation}
	\label{e:f(T)}
	|\Out(T)|^2f(T) < \left(\frac{|T|}{ h(T)}\right)^2.
	\end{equation}
	We divide the proof into several cases. 
	
	\vs
	
	\noindent \emph{Case 1. $T\ne \LL_n^\e(q)$.}
	
	\vs
	
	In this setting, \cite[Theorem 1.2]{GG_conju} implies that $f(T)<|T|/h(T)$, so \eqref{e:f(T)} holds (and thus the lemma follows), if we can show that
	\begin{equation}
	\label{e:r2_out^2}
	h(T)|\Out(T)|^2<|T|.
	\end{equation}
	
	First, we assume $T\ne \POmega_{8}^+(q)$. Here $|\Out(T)|\leqs 8\log q$ and by inspecting Table \ref{tab:fix(Hol(T))}, one can see that $|T|/h(T)\geqs q^3/2$. It is straightforward to check that if $q\geqs 13$, then $128(\log q)^2 < q^3$, which implies that \eqref{e:r2_out^2} holds for $q\geqs 13$. Then there are only finitely many exceptional groups of Lie type to consider, and in each case we can use the precise value of $h(T)$ in Table \ref{tab:fix(Hol(T))} to verify \eqref{e:r2_out^2}. Hence, we may assume $q\leqs 11$ and $T$ is a classical group. By our assumption, $T = \PSp_n(q)$, $\Omega_n(q)$, $\POmega_n^-(q)$, or $\POmega_n^+(q)$ with $n\geqs 10$ in the latter case. In each case, we have $|T|/h(T) > q^{n-2}$ by inspecting Table \ref{tab:fix(Hol(T))}, so if $n\geqs 8$ we have
	\begin{equation*}
	|\Out(T)|^2\leqs 64(\log q)^2 \leqs q^6 \leqs q^{n-2} < |T|/h(T)
	\end{equation*}
	and thus \eqref{e:r2_out^2} holds. There are finitely many groups remaining and we can check that the inequality in \eqref{e:r2_out^2} holds in each case.
	
	Now assume $T = \POmega_{8}^+(q)$. Here $|T|/h(T)>  q^6$ and $|\Out(T)|\leqs 24f\leqs 24\log q$. This shows that \eqref{e:r2_out^2} holds for $q\geqs 4$ since we have $24^2(\log q)^2<q^6$. If $q = 2$, then $|\Out(T)|^2 =  36 < 120 = |T|/h(T)$, while if $q = 3$, then $|\Out(T)|^2 = 576 < 1080 = |T|/h(T)$.
	
	\vs
	
	\noindent \emph{Case 2. $T = \UU_n(q)$, $n\geqs 3$.}
	
	\vs
	
	In this case, \cite[Theorem 1.2]{GG_conju} implies that $f(T)<\frac{1}{2}|T|/h(T)$, except when $(n,q) = (3,3)$ or $(4,3)$. In the latter two cases, it is easy to check \eqref{e:f(T)}. In other cases, we have $|T|/h(T)>q^n$ by inspecting Table \ref{tab:fix(Hol(T))}, so \eqref{e:f(T)} holds if
	\begin{equation}
	\label{e:r2_out^2_u}
	|\Out(T)|^2<2q^n.
	\end{equation}
	Notice that $|\Out(T)| \leqs 2(q+1)\log q < q^2$ for $q\geqs 7$, and for $q \in\{3,5\}$ we still have the inequalities $|\Out(T)|\leqs 2(q+1)<q^2$. This implies that if $q\notin\{2,4\}$ and $n\geqs 4$, then we have
	\begin{equation*}
	|\Out(T)|^2<q^4\leqs q^n<2q^n
	\end{equation*}
	and so \eqref{e:r2_out^2_u} is satisfied. If $q = 2$ then $|\Out(T)| \leqs 6$, so \eqref{e:r2_out^2_u} holds if $n\geqs 5$; and if $q = 4$ then $|\Out(T)|\leqs 20$, and thus \eqref{e:r2_out^2_u} holds for $n\geqs 4$. It is straightforward to check \eqref{e:f(T)} when $T = \UU_4(2)$, where we have $f(T) = 20$.
	
	Finally, assume $n = 3$, so $|\Out(T)|\leqs 6\log q$. Here \eqref{e:r2_out^2_u} is satisfied for all $q>4$ since $(6\log q)^2<2q^3$. By our assumption, the only remaining cases are $T = \UU_3(3)$ with $f(T) = 14$ and $T=\UU_3(4)$ with $f(T) = 22$, so the inequality in \eqref{e:f(T)} holds.
	
	\vs
	
	\noindent \emph{Case 3. $T = \LL_n(q)$.}
	
	\vs
	
	Here we assume $(n,q)\ne (2,4), (2,5),(2,9),(3,2),(4,2)$ as noted in \eqref{e:iso_simple}. If $n = 2$ and $q\in \{7,11\}$, then an easy computation using {\sc Magma} shows that \eqref{e:r2_f} holds, and the result follows.
	
	In each of the remaining cases, we have $|T|/h(T) > q^{n-1}$ by inspecting Table \ref{tab:fix(Hol(T))}. Moreover, \cite[Corollary 1.2]{FG_conju} implies that $f_p(\Aut(T))<100|T|/h(T)$, so \eqref{e:r2_f} holds if
	\begin{equation}
	\label{e:r2_out^2_l}
	100|\Out(T)|<q^{n-1}.
	\end{equation}
	Since $|\Out(T)|\leqs 2(q-1)\log q<q^2$ for all $q$, \eqref{e:r2_out^2_l} holds if $n\geqs 10$. Moreover, if $n\geqs 4$ then \eqref{e:r2_out^2_l} holds if $q> 100$, while for $q<100$ it is easy to check that \eqref{e:r2_out^2_l} still holds in each case, unless $q=2$ and $n\leqs 8$, or $n\in\{5,6\}$ and $q\leqs 4$, or $n = 4$ and $q\leqs 9$. But in each of these cases, it is straightforward to check that \eqref{e:r2_f} is satisfied, so to complete the proof we may assume $n\in\{2,3\}$
	
	Suppose $n = 3$, so $|\Out(T)|\leqs 6\log q$, and \eqref{e:r2_out^2_l} holds if $600\log q <q^2$. The latter holds if $q>59$. In fact, by working with the precise value of $|\Out(T)|$ we see that \eqref{e:r2_out^2_l} holds if $q>25$. Finally, if $q\leqs 25$, then we can check \eqref{e:r2_f} using {\sc Magma}.
	
	To complete the proof, we may assume $T = \LL_2(q)$, so $|\Out(T)| \leqs 2\log q$ and $|T|/h(T)\geqs (q+1)q^{1/2}/2$. Thus, \eqref{e:r2_f} holds if
	\begin{equation*}
	800\log q <(q+1)^2
	\end{equation*}
	since we have $f_p(\Aut(T))<100q$ by \cite[Corollary 1.2]{FG_conju}. In this way, we deduce that \eqref{e:r2_f} holds if $q\geqs 71$. And for $q < 71$, we can check that \eqref{e:r2_f} holds with the aid of {\sc Magma}.
\end{proof}

\begin{lem}
	\label{l:r3}
	We have
	\begin{equation*}
	r_3(G) < {k\choose 2}\left(\frac{1}{|T|}+\frac{|\Out(T)| h(T)^{k-3}}{|T|^{k-3}}\right)+\frac{k!}{|T|^{\frac{4}{3}}}+|T|^{-\frac{1}{3}}\left(2{k\choose 3}+\frac{1}{2}{k\choose 2}{k-2\choose 2}\right).
	\end{equation*}
\end{lem}

\begin{proof}
	First, let
	\begin{equation*}
	\begin{aligned}
	R_4(G) &= \{(\a,\dots,\a)\pi\in R_3(G):\pi = (1,2)\},\\
	R_4(T) &= \{\a\in\Aut(T):(\a,\dots,\a)\pi\in R_4(G)\}
	\end{aligned}
	\end{equation*}
	as in the proof of \cite[Theorem 1.5]{F_diag}. Set $P = S_k$ and
	\begin{equation*}
	r_4(G):=|(1,2)^P|\sum_{\a\in R_4(T)}\frac{|\a^{\Aut(T)}|}{|T|}\left(\frac{|C_{\Inn(T)}(\a)|}{|T|}\right)^{k-3}.
	\end{equation*}
	Then we have
	\begin{equation}\label{e:r4}
	\begin{aligned}
	r_4(G) 
	&= {k\choose 2}\left(\frac{1}{|T|}+\sum_{\a\in R_4(T)\setminus\{1\}}\frac{|\a^{\Aut(T)}|}{|T|}\left(\frac{|C_{\Inn(T)}(\a)|}{|T|}\right)^{k-3}\right)\\
	&\leqs {k\choose 2}\left(\frac{1}{|T|}+|\Out(T)|\left(\frac{ h(T)}{|T|}\right)^{k-3}\right).
	\end{aligned}
	\end{equation}
	
	As noted in the proof of \cite[Theorem 1.5]{F_diag}, we have
	\begin{equation}
	\label{e:r3_F}
	r_3(G) \leqs  r_4(G)+\sum_{\pi\in R\setminus\{(1,2)\}}\frac{|\pi^P|}{|T|^{k-r_\pi-\frac{5}{3}}},
	\end{equation}
	where $R$ is a set of representatives for the conjugacy classes of elements of prime order in $P$ and $r_\pi$ is the number of $\la\pi\ra$-orbits in $[k]$. Without loss of generality, we may assume $(1,2)\in R$.
	
	Let $x,y\in R$ be the representatives the $P$-classes $(1,2,3)^P$ and $(1,2)(3,4)^P$, respectively. Note that $r_x = r_y = k-2$ and $r_z\leqs k-3$ for all $z\in R\setminus\{(1,2),x,y\}$. Then
	\begin{equation*}
	\begin{aligned}
	\sum_{\pi\in R\setminus\{(1,2)\}}\frac{|\pi^P|}{|T|^{k-r_\pi-\frac{5}{3}}} &= \sum_{\pi\in R\setminus\{(1,2),x,y\}}\frac{|\pi^P|}{|T|^{k-r_\pi-\frac{5}{3}}} + |T|^{-\frac{1}{3}}\left(2{k\choose 3}+\frac{1}{2}{k\choose 2}{k-2\choose 2}\right)\\
	&<\frac{k!}{|T|^{\frac{4}{3}}}+|T|^{-\frac{1}{3}}\left(2{k\choose 3}+\frac{1}{2}{k\choose 2}{k-2\choose 2}\right)
	\end{aligned}
	\end{equation*}
	and so the lemma follows by combining \eqref{e:r4} and \eqref{e:r3_F}.
\end{proof}

Now we define
\begin{equation}\label{e:Q1}
Q_1(G):= (k!)^2|T|^{\frac{8}{3}-\frac{k}{2}-\frac{1}{2}\delta_{5,k}} + \frac{k!}{|T|^{\frac{4}{3}}} + \frac{k^4}{2|T|^{\frac{1}{3}}},
\end{equation}
where $\delta_{5,k} = 1$ if $k = 5$ and $\delta_{5,k} = 0$ otherwise, and
\begin{equation}\label{e:Q2}
Q_2(G):=\left(\frac{|T|}{h(T)}\right)^{4-k} + {k\choose 2}|\Out(T)|\left(\frac{|T|}{h(T)}\right)^{3-k}.
\end{equation}
By Lemmas \ref{l:r1}, \ref{l:r2} and \ref{l:r3}, we have
\begin{equation}\label{e:r<Q}
r_1(G)+r_2(G)+r_3(G) < Q_1(G) + Q_2(G).
\end{equation}

\begin{lem}
	\label{l:r_prob_F}
	If $Q_1(G)+Q_2(G) < 1/2$ and $5\leqs k\leqs 4\log |T|$, then $r(G)\geqs 2$. In particular, $b(G) = 2$.
\end{lem}

\begin{proof}
	By \eqref{e:F_prob_r} and \eqref{e:r<Q}, we have
	\begin{equation*}
	\frac{1}{2} > Q_1(G) + Q_2(G) > 1-\frac{r(G)|G|}{|T|^{2k-2}}= 1-\frac{r(G)|\Out(T)|\cdot k!}{|T|^{k-2}}.
	\end{equation*}
	It suffices to prove that
	\begin{equation*}
	2|\Out(T)|\cdot k!\leqs |T|^{k-2},
	\end{equation*}
	which is clear since $k\leqs 4\log |T|$.
\end{proof}

\section{Proofs of Theorems \ref{thm:b(G)=2}, \ref{thm:r=1} and \ref{thm:Hol}}

\label{s:proof_b(G)=2}

In this section, we will establish Theorems \ref{thm:b(G)=2}, \ref{thm:r=1} and \ref{thm:Hol}. We will consider the following cases in turn:

\begin{enumerate}\addtolength{\itemsep}{0.2\baselineskip}
	\item[{\rm (a)}] $P\in\{A_k,S_k\}$ and $k\in\{3,4,|T|-4,|T|-3\}$;
	\item[{\rm (b)}] $P\in\{A_k,S_k\}$ and $k\in\{|T|-2,|T|-1\}$;
	\item[{\rm (c)}] $P = S_k$, $5\leqs k\leqs |T|/2$ and $G = W$.
\end{enumerate}

More specifically, we will prove that $r(G)\geqs 2$ for every group in cases (a) and (c), with the exception of the two special cases arising in the statement of Theorem \ref{thm:r=1} (in both cases, $b(G) = 2$ and $r(G) = 1$). Then Lemma \ref{l:Sk_eq_Hol_r} shows that $b(G) = 2$ if $P\in\{A_k,S_k\}$ and $3\leqs k\leqs |T|-3$, as in part (ii) of Theorem \ref{thm:b(G)=2}, which also establishes Theorem \ref{thm:Hol}. In particular, we deduce that $r(G)\geqs 2$ if $P\notin\{A_k,S_k\}$ and $k\leqs 32$, as noted in Remark \ref{r:32}.

As explained in Section \ref{s:pre}, we will exclude the simple groups listed in \eqref{e:iso_simple}, due to the existence of isomorphisms. 

\subsection{The groups with $k\in\{3,4,|T|-4,|T|-3\}$}

\label{ss:34}

We start with case (a).

\begin{lem}
	\label{l:sporadic_k=3,4}
	Suppose $k\in\{3,4\}$, $P = S_k$ and $T$ is a sporadic simple group. Then $r(G)\geqs 2$.
\end{lem}

\begin{proof}
	If $T\notin\{\mathrm{Ly},\mathrm{Th},\mathrm{J}_4,\mathbb{B},\mathbb{M}\}$ then we can construct $T$ as a permutation group in {\sc Magma} using the function \texttt{AutomorphismGroupSimpleGroup}. Then the result follows by random search (see Remark \ref{r:magma}(i)). If $T\in \{\mathrm{Ly},\mathrm{Th},\mathrm{J}_4,\mathbb{B},\mathbb{M}\}$, then $|\Out(T)| = 1$. Let $M$ be a maximal subgroup of $T$ with
	\begin{equation}\label{e:sporadic_K}
	(T,M) \in \{ (\mathrm{Ly},G_2(5)),\ (\mathrm{Th},\mathrm{AGL}_2(5)),\ (\mathrm{J}_4,\mathrm{M}_{22}.2),\ (\mathbb{B},\mathrm{Fi}_{23}),\ (\mathbb{M},\LL_2(71))\}.
	\end{equation}
	In view of Corollary \ref{c:Hol(T,S)=1_cond}, the result follows by random search as in Remark \ref{r:magma}(ii).
\end{proof}

We define the following set of finite simple groups of Lie type:
\begin{equation*}
\begin{aligned}
\mathcal{C}:=\{&{^2}B_2(8),{^2}B_2(32),G_2(3),G_2(4),{^2}F_4(2)',{^3}D_4(2),F_4(2),\LL_2(7),\LL_2(8),\\&\LL_2(11),\LL_2(13),\LL_2(16),\LL_2(27),\LL_2(32),\LL_3^\e(3),\LL_3^\e(4),\UU_3(5),\UU_3(8),\LL_4^\e(3),\\&\PSp_4(3),\Sp_4(4),\LL_5^\e(2),\UU_6(2),\Sp_6(2),\PSp_4(3),\Sp_8(2),\Omega_8^\e(2),\POmega_8^+(3)\}.
\end{aligned}
\end{equation*}
Recall that an element $x$ of a simple group of Lie type $T$ defined over a field of characteristic $p$ is regular semisimple if and only if $|C_T(x)|$ is indivisible by $p$.

\begin{lem}
	\label{l:reg_semi_8}
	Suppose $T\notin\mathcal{C}$ is a finite simple group of Lie type. Then $T$ has at least $8$ regular semisimple $\Aut(T)$-classes.
\end{lem}

\begin{proof}
	Suppose $T$ is a Lie type group defined over $\mathbb{F}_q$, where $q = p^f$ for some prime $p$. We will work with a quasisimple group $Q$ with $Q/Z(Q) = T$. Let $m$ be the number of regular semisimple conjugacy classes in $Q$. Then $T$ has at least $m|T|/|Q|$ regular semisimple $T$-classes, and thus $T$ has at least $8$ regular semisimple $\Aut(T)$-classes if
	\begin{equation}\label{e:reg_m}
	m|T|\geqs 8|\Out(T)||Q|.
	\end{equation}
	
	First assume $Q$ is a simply connected quasisimple exceptional group. Then $m$ has been computed by L\"{u}beck \cite{L_web}, and one can see that \eqref{e:reg_m} holds for every $T\notin\mathcal{C}$ by inspecting \cite{L_web}.
	
	Next, assume  $Q\in\{\SL_n^\e(q),\Sp_n(q)\}$, so $m$ is given in \cite{FG_reg_semi}. The result now follows by inspecting \cite{FG_reg_semi}. For example, if $Q = \SL_2(q)$ then $|Q|/|T| = (2,q-1)$, $|\Out(T)| = (2,q-1)f$ and
	\begin{equation*}
	m = q-3+(2,q)
	\end{equation*}
	by \cite[Theorem 2.4]{FG_reg_semi}. Thus, \eqref{e:reg_m} is valid if
	\begin{equation*}
	q-3+(2,q)\geqs 8(2,q-1)^2f,
	\end{equation*}
	which holds for all $q > 81$. For the cases where $q\leqs 81$ and $T\notin \mathcal{C}$, one can check using {\sc Magma} that there are at least $8$ regular semisimple $\Aut(T)$-classes. We use an entirely similar argument to treat all the other cases and we omit the details.
	
	To complete the proof, we assume $Q = \Omega_n^\e(q)$, so $Q$ has index $2$ in $\mathrm{SO}_n^\e(q)$. First assume $q$ is even. Here $Q = T$ and every semisimple element in $\mathrm{SO}_n^\e(q)$ has odd order, and so lies in $Q$. This implies that $m$ is at least the number of regular semisimple $\mathrm{SO}_n^\e(q)$-classes in $\mathrm{SO}_n^\e(q)$, which is computed in \cite[Theorem 5.12]{FG_reg_semi}, and the result follows by arguing as above.
	
	Finally, assume $Q = \Omega_n^\e(q)$ and $q$ is odd. Write $d=\lceil n/2\rceil-1$. Let $A\in\GL_d(q)$ be of order $q^d-1$ and let
	\begin{equation*}
	x = 
	\begin{pmatrix}
	A&&\\
	&(A^{-1})^T&\\
	&&I_{n-2d}
	\end{pmatrix}
	\end{equation*}
	with respect to a standard basis (see \cite[Proposition 2.5.3]{KL_classical}). Then $x\in \mathrm{SO}_n^\e(q)$, so $y:=x^2\in\Omega_n^\e(q)$, noting that 
	\begin{equation*}
	y = 
	\begin{pmatrix}
	B&&\\
	&(B^{-1})^T&\\
	&&I_{n-2d}
	\end{pmatrix},
	\end{equation*}
	where $B = A^2$. Let $\mu$ be an eigenvalue of $B$ of order $(q^d-1)/2$ in the algebraic closure $K$ of $\mathbb{F}_q$. Then it is easy to show that $\mu\ne \mu^{\pm q^t}$ for any $1\leqs t\leqs d-1$, and the set of eigenvalues of $y$ is
	\begin{equation*}
	\{\mu,\mu^q,\dots,\mu^{q^{d-1}},\mu^{-1},\mu^{-q},\dots,\mu^{-q^{d-1}},1\},
	\end{equation*}
	where $1$ has multiplicity $n-2d\in\{1,2\}$ and any other eigenvalue has multiplicity $1$. It follows that $y^i$ is regular semisimple if $(i,(q^d-1)/2) = 1$. This gives at least
	\begin{equation*}
	\frac{\phi\left((q^d-1)/2\right)}{2d}
	\end{equation*}
	regular semisimple $\mathrm{GO}_n^\e(q)$-classes in $Q$, where $\phi$ is the Euler’s totient function (note that two elements are not conjugate in $\GL_n(q)$ if they have distinct sets of eigenvalues). By arguing as above, $T$ has at least $8$ regular semisimple $\Aut(T)$-classes if
	\begin{equation}
	\label{e:reg_d_o}
	\phi\left((q^d-1)/2\right)\geqs 32d\cdot |\Aut(T):\mathrm{PGO}_n^\e(q)|,
	\end{equation}
	noting that $|\Aut(T):\mathrm{PGO}_n^\e(q)|\leqs f\leqs \log q$ if $d\ne 3$, while $|\Aut(T):\mathrm{PGO}_n^\e(q)|\leqs 3f\leqs 3\log q$ if $d = 3$. It is easy to check that \eqref{e:reg_d_o} holds unless
	\begin{equation*}
	(d,q)\in\{(6,3),(5,3),(4,3),(4,5),(4,7),(3,3),(3,5),(3,7)\}.
	\end{equation*}
	For these remaining cases, one can use {\sc Magma} to obtain $m$ and so \eqref{e:reg_m} holds unless
	\begin{equation*}
	Q\in\{\Omega_{10}^-(3),\Omega_8^+(5),\Omega_8^\e(3),\Omega_7(3)\},
	\end{equation*}
	where we can directly check that there are at least $8$ regular semisimple $\Aut(T)$-classes in $T$ with the aid of {\sc Magma}.
\end{proof}

We remark that $\POmega_8^+(3)$ has exactly $8$ regular semisimple $\Aut(T)$-classes in $T$. If $T\ne\POmega_8^+(3)$ and $T\in\mathcal{C}$, then the number of $\Aut(T)$-classes of regular semisimple elements in $T$ is strictly less than $8$, which can be checked using {\sc Magma}. We include $\POmega_8^+(3)$ in $\mathcal{C}$ in view of Theorem \ref{t:inv_gen_reg}, so if $T\notin\mathcal{C}$ then $T$ is invariably generated by a pair of regular semisimple elements of distinct orders.

%Recall that by Lemma \ref{l:Sk_eq_Hol_r}, we have $r(G)\geqs 2$ if there exist subsets $S_1,S_2\subseteq T$ lying in distinct $\Hol(T)$-orbits such that $\Hol(T,S_1) = \Hol(T,S_2) = 1$.

\begin{lem}
	\label{l:k=3_r(G)}
	Suppose $k = 3$, $P = S_k$ and $T\notin\mathcal{C}$ is a simple group of Lie type. Then $r(G)\geqs 2$.
\end{lem}

\begin{proof}
	Let $x$ and $y$ be as described in Theorem \ref{t:inv_gen_reg}. Let $z_1$ and $z_2$ be semisimple elements in $T$ lying in distinct $\Aut(T)$-classes and
	\begin{equation*}
	z_1,z_2\notin x^{\Aut(T)}\cup (x^{-1})^{\Aut(T)}\cup y^{\Aut(T)}\cup (y^{-1})^{\Aut(T)}.
	\end{equation*}
	Note that the existence of $z_1$ and $z_2$ follows from Lemma \ref{l:reg_semi_8}. Then by applying \cite[Theorem 2]{G_reg_semi}, which asserts that the product of any two regular semisimple $T$-classes contains all semisimple elements in $T$, there exist $g_i$ and $h_i$ in $T$ such that $z_i = x^{g_i}y^{h_i}$, and without loss of generality we may assume $g_i = 1$, so $z_i = xy^{h_i}$. It is easy to see that $\Hol(T,\{1,x^{-1},y^{h_i}\}) = 1$, and so $b(G) = 2$. By Lemma \ref{l:Sk_eq_Hol_r}, it suffices to show that $S_1 = \{1,x^{-1},y^{h_1}\}$ and $S_2 = \{1,x^{-1},y^{h_2}\}$ are in distinct $\Hol(T)$-orbits. Suppose $S_1^{g\a} = S_2$ for some $g\a\in\Hol(T)$, and note that $g\in S_1$ by Lemma \ref{l:translation}. If $g = 1$ then $(x^{-1})^\a = x^{-1}$ and $(y^{h_1})^\a = y^{h_2}$. However, this implies that
	\begin{equation*}
	z_1^\a = (xy^{h_1})^\a = xy^{h_2} = z_2,
	\end{equation*}
	which is incompatible with our assumption $z_1^{\Aut(T)}\ne z_2^{\Aut(T)}$. If $g = x^{-1}$ then $(y^{h_1})^g = xy^{h_1} = z_1$, which is not $\Aut(T)$-conjugate to any element in $S_2$, a contradiction. Finally, if $g = y^{h_1}$ then we have $(x^{-1})^g = y^{-h_1}x^{-1} = z_1^{-1}$. With the same reason, this is impossible. Therefore, there is no $g\a\in\Hol(T)$ such that $S_1^{g\a} = S_2$, which completes the proof.
\end{proof}

\begin{lem}
	\label{l:k=4_r(G)>1}
	Suppose $k=4$, $P = S_k$ and  $T\notin\mathcal{C}$ is a simple group of Lie type. Then $r(G)\geqs 2$.
\end{lem}

\begin{proof}
	Let $x$ and $y$ be as in Theorem \ref{t:inv_gen_reg}. By \cite[Theorem 2]{G_reg_semi}, every semisimple element in $T$ lies in $x^Ty^T$, so we may assume that
	\begin{equation}\label{e:k=4_construction_x-1y}
	x^{-1}y\notin x^{\Aut(T)}\cup (x^{-1})^{\Aut(T)}\cup y^{\Aut(T)}\cup (y^{-1})^{\Aut(T)}.
	\end{equation}
	Additionally, using Lemma \ref{l:reg_semi_8}, let $z_0$ be a regular semisimple element such that
	\begin{equation}\label{e:k=4_construction_z0}
	z_0\notin x^{\Aut(T)}\cup (x^{-1})^{\Aut(T)}\cup y^{\Aut(T)}\cup (y^{-1})^{\Aut(T)}\cup (x^{-1}y)^{\Aut(T)}\cup (y^{-1}x)^{\Aut(T)}.
	\end{equation}
	Again, \cite[Theorem 2]{G_reg_semi} implies that $x^Tz_0^T$ contains all semisimple elements in $T$. Thus, by Lemma \ref{l:reg_semi_8}, there exists $z\in z_0^T$ such that
	\begin{equation}\label{e:k=4_construction_z}
	z^{-1}x\notin x^{\Aut(T)}\cup (x^{-1})^{\Aut(T)}\cup y^{\Aut(T)}\cup (y^{-1})^{\Aut(T)}\cup (x^{-1}y)^{\Aut(T)}\cup (y^{-1}x)^{\Aut(T)}.
	\end{equation}
	Set $S_1 = \{1,x,y,z\}$ and suppose $g\a\in\Hol(T,S_1)$. If $g = 1$ then $\a\in\Aut(T,S_1) = 1$ as $\la x,y\ra = T$ and $x,y,z$ are in distinct $\Aut(T)$-classes. If $g = x$ then $x^{-1}y\in x^{-1}S_1 = S_1^{\a^{-1}}$, which is incompatible with either \eqref{e:k=4_construction_x-1y} or \eqref{e:k=4_construction_z0}. The case where $g = y$ can be eliminated using the same argument. If $g = z$, then $z^{-1}S_1 = S_1^{\a^{-1}}$, and by using \eqref{e:k=4_construction_z0} and \eqref{e:k=4_construction_z}, both $z^{-1}$ and $z^{-1}x$ are $\Aut(T)$-conjugate to $z$, which yields $z^{-1} = z^\a = z^{-1}x$, a contradiction. Thus, we have $b(G) = 2$.
	
	Similarly, Lemma \ref{l:reg_semi_8} implies that there exists a regular semisimple element $w\in T$ such that $w\ne z$,
	\begin{equation*}
	w\notin x^{\Aut(T)}\cup (x^{-1})^{\Aut(T)}\cup y^{\Aut(T)}\cup (y^{-1})^{\Aut(T)}\cup (x^{-1}y)^{\Aut(T)}\cup (y^{-1}x)^{\Aut(T)}
	\end{equation*}
	and 
	\begin{equation*}
	w^{-1}x\notin x^{\Aut(T)}\cup (x^{-1})^{\Aut(T)}\cup y^{\Aut(T)}\cup (y^{-1})^{\Aut(T)}\cup (x^{-1}y)^{\Aut(T)}\cup (y^{-1}x)^{\Aut(T)}.
	\end{equation*}
	Set $S_2 = \{1,x,y,w\}$. By arguing as above, we have $\Hol(T,S_2) = 1$ and it suffices to show that $S_1$ and $S_2$ are in distinct $\Hol(T)$-orbits. Suppose $S_1^{g\a} = S_2$ and note that $g\in S_1$ by Lemma \ref{l:translation}. If $g=1$ then $x^\a = x$ and $y^\a = y$, which implies that $\a = 1$. However, this is incompatible with $z\ne w$. If $g = x$ then
	\begin{equation*}
	1^g = x^{-1},\ y^g = x^{-1}y \ \mbox{and}\ z^g = x^{-1}z.
	\end{equation*}
	Thus, one of the above is $\Aut(T)$-conjugate to $w$, which has to be $z^g = x^{-1}z$ by our assumption. However, this gives a contradiction since $y^g = x^{-1}y$ is not $\Aut(T)$-conjugate to $x$ or $y$ by \eqref{e:k=4_construction_x-1y}. The case where $g = y$ can be eliminated similarly. Finally, if $g = z$ then
	\begin{equation*}
	x^g = z^{-1}x,\ y^g = z^{-1}y \ \mbox{and}\ 1^g = z^{-1}.
	\end{equation*}
	Once again, the only possibility is $x^{g\a} = w$ by \eqref{e:k=4_construction_z}. But this leaves $(z^{-1})^\a = 1^{g\a}\in\{x,y\}$, which is incompatible with \eqref{e:k=4_construction_z0}.
\end{proof}

We can now establish Theorems \ref{thm:b(G)=2} and \ref{thm:r=1} for $k\in\{3,4,|T|-4,|T|-3\}$.

\begin{prop}
	\label{p:k=3,4}
	If $k\in\{3,4,|T|-4,|T|-3\}$ then $r(G)\geqs 1$, with equality if and only if $T = A_5$, $k\in\{3,57\}$ and $G = T^k.(\Out(T)\times S_k)$.
\end{prop}

\begin{proof}
	By Proposition \ref{p:r(G)>1_k>32}, we may assume $P\in\{A_k,S_k\}$. First assume $k\in\{3,4\}$ and $P = S_k$. The groups where $T$ is sporadic have been treated in Lemma \ref{l:sporadic_k=3,4}. If $T\notin\mathcal{C}$ is Lie type, then by Lemmas \ref{l:k=3_r(G)} and \ref{l:k=4_r(G)>1}, we have $r(G)\geqs 2$ as desired. The cases where $T\in\mathcal{C}$ can be handled by random search (see Remark \ref{r:magma}(i)).
	
	Thus, to complete the proof for $k\in\{3,4\}$ and $P = S_k$ we may assume $T = A_n$ is an alternating group. First assume $k = 3$ and $T = A_5$. One can check using {\sc Magma} that $\Hol(T)$ has a unique regular orbit on $\mathscr{P}_k$, so $r(G) = 1$ if $G = W = A_5^3.(2\times S_3)$. With the aid of {\sc Magma}, one can show that $r(G)\geqs 2$ if $G<W$. Here we obtain the permutation group $G$ in {\sc Magma} by accessing the primitive group database, noting that $|\Omega| = |A_5|^2 = 3600$.
	
	Next, assume $P = S_3$ and $T = A_n$ with $n\geqs 6$. The cases where $n\leqs 8$ can be easily handled using {\sc Magma} (see Remark \ref{r:magma}(i)). Now assume $n\geqs 9$, so by \cite{M_23_alter}, there exist $x_1,y_1\in T$ such that $|x_1| = 2$, $|y_1| = 3$ and $\la x_1,y_1\ra = T$. Note that if $|x_1y_1| = 2$ or $3$, then $\la x_1,y_1\ra = S_3$ or $A_4$ respectively, so we must have $|x_1y_1|\geqs 4$. Hence, $\Hol(T,\{1,x_1,y_1\}) = 1$ by Lemma \ref{l:Hol(T,S)=1_cond}, and thus $b(G) = 2$. Let $x_2 = (1,2,\dots,n)$ if $n$ is odd, while $x_2 = (1,2)(3,\dots,n)$ if $n$ is even, and let $y_2 = (1,2,3)x_2^{-1}$. Then $\la x_2,y_2\ra = T$ and Lemma \ref{l:Hol(T,S)=1_cond} implies that $\Hol(T,\{1,x_2,y_2\}) = 1$, so we have $r(G)\geqs 2$ by Lemma \ref{l:r_ge_2_cond}.
	
	Now assume $P = S_4$ and $T = A_n$. The cases where $n\leqs 11$ can be handled using {\sc Magma} as noted in Remark \ref{r:magma}(i). Assume $n\geqs 12$ and let $x = (1,2)(3,4)$. Let $C_1$ and $C_2$ be the set of involutions moving $8$ and $12$ points in $[n]$, respectively. Note that there exist $y_1\in C_1$ and $y_2\in C_2$ such that $xy_i\ne y_ix$. Moreover, by \cite[Theorem 1.2]{BGK_spread}, there exist $z_1$ and $z_2$ such that
	\begin{equation*}
	T = \la x,z_1\ra = \la y_1,z_1\ra = \la x,z_2\ra = \la y_2,z_2\ra.
	\end{equation*}
	In particular, $2\notin\{|z_i|,|xz_i|,|y_iz_i|\}$. Set $S_1 = \{1,x,y_1,z_1\}$ and $S_2 = \{1,x,y_2,z_2\}$. We first prove that $\Hol(T,S_i) = 1$. Suppose $g\a\in\Hol(T,S_i)$. If $g = 1$ then $\a\in\Aut(T,S) = 1$ since $\la x,z_i\ra = T$ and $x,y_i,z_i$ are in distinct $\Aut(T)$-classes. If $g = x$ then $2\notin\{|y_i^g|,|z_i^g|\} = \{|xy_i|,|xz_i|\}$, which is impossible. The cases where $g \in\{y_i,z_i\}$ can be eliminated similarly. This implies that $b(G) = 2$. By applying Lemma \ref{l:translation}, one can show that $S_1$ and $S_2$ are in distinct $\Hol(T)$-orbits.
	
	Therefore, we have $r(G)\geqs 1$ if $k\in\{3,4\}$, with equality if and only if $G = A_5^3.(2\times S_3)$. By Lemma \ref{l:Sk_eq_Hol_r}, it suffices to consider the case where $T = A_5$ and $k = |A_5|-3 = 57$. Note that $r(G) = 1$ if $G = W = A_5^{57}.(2\times S_{57})$, and $G$ has at least $|W:G|$ regular suborbits if $G<W$.
\end{proof}

\subsection{The groups with $P\in\{A_k,S_k\}$ and $k\in\{|T|-2,|T|-1\}$}

\begin{lem}
	\label{l:Aut_2,3_r}
	Suppose $m\in\{2,3\}$. Then there exist $S_1,S_2\subseteq T^\#$ such that $|S_i| = m$, $\Aut(T,S_i) = 1$ and $S_1^{\Aut(T)}\ne S_2^{\Aut(T)}$.
\end{lem}

\begin{proof}
	First observe that if $S_1\cup \{1\}$ and $S_2\cup\{1\}$ are in distinct regular $\Hol(T)$-orbits, then all conditions in the statement of the lemma are satisfied. Hence, the result follows from Lemma \ref{l:Sk_eq_Hol_r} and Proposition \ref{p:k=3,4}, except when $T = A_5$ and $m = 2$. In the latter case, we can verify the lemma using {\sc Magma}.
\end{proof}

\begin{prop}
	\label{p:k=|T|-1,|T|-2}
	Assume $k  = |T|-1$ or $|T|-2$.
	\begin{enumerate}\addtolength{\itemsep}{0.2\baselineskip}
		\item[{\rm (i)}] If $G$ contains $S_k$, then $b(G) = 3$.
		\item[{\rm (ii)}] If $G$ does not contain $S_k$, then $r(G)\geqs 2$.
	\end{enumerate}
\end{prop}

\begin{proof}
	Recall that $b(G)\in\{2,3\}$ by Theorem \ref{t:F}(iii). First assume $G$ contains $S_k$. It suffices to show that $b(G) = 3$ if $G = T^k{:}S_k$. Suppose $\{D,D(\varphi_{t_1},\dots,\varphi_{t_k})\}$ is a base for $G$. If $t_i = t_j$ for some $i\ne j$, then $(i,j)\in G$ stabilises $D$ and $D(\varphi_{t_1},\dots,\varphi_{t_k})$ pointwise. Therefore, the elements $t_1,\dots,t_k$ are distinct. Let $S = T\setminus\{t_1,\dots,t_k\}$, so $|S|\in\{1,2\}$. Without loss of generality, we may also assume $1\in S$. Thus, there exists a non-identity element $t\in T$ such that $S^{\varphi_t} = S$, and hence $\varphi_t\in\Hol(T,T\setminus S)$, which is incompatible with Lemma \ref{l:Sk_eq_Hol}.
	
	Now we turn to the case where $G$ does not contain $S_k$. Recall that $T^k{:}A_k\leqs G$ by Corollary \ref{c:Ak<G}. From Lemma \ref{l:Aut_2,3_r}, there are subsets $S_1,S_2\subseteq T^\#$ of size $|T|-k+1$ lying in distinct regular $\Aut(T)$-orbits. Write $T^\#\setminus S_i = \{t_{i,1},\dots,t_{i,k-2}\}$ and consider $\Delta_i = \{D,D(\varphi_{t_{i,1}},\dots,\varphi_{t_{i,k}})\}$, where $t_{i,k-1} = t_{i,k} = 1$. Suppose $x = (\a,\dots,\a)\pi\in G_{(\Delta_i)}$. By Lemma \ref{l:l:3.4_diag}, $t_{i,j}^\a = t_{i,{j}^\pi}$ for all $j$. It follows that $\a\in\Aut(T,S_i)$ and thus $\a = 1$. Hence, $x = \pi \in \la (k-1,k)\ra$, and so $x = 1$ since $G$ does not contain $S_k$. This shows that $b(G) = 2$. Finally, if $\Delta_1$ and $\Delta_2$ are in the same $G_D$-orbit, then
	\begin{equation*}
	D(\varphi_{t_{1,1}},\dots,\varphi_{t_{1,k}})^{(\a,\dots,\a)\pi} = D(\varphi_{t_{2,1}},\dots,\varphi_{t_{2,k}})
	\end{equation*}
	for some $\a\in\Aut(T)$ and $\pi\in S_k$. This implies that $S_1^\a = S_2$, which is incompatible with our assumption. Therefore, $r(G)\geqs 2$ and the proof is complete.
\end{proof}

\subsection{The groups with $P = S_k$, $5\leqs k\leqs |T|/2$ and $G = W$}

Finally, let us turn to case (c) mentioned in the beginning of this section. Note that if $r(G)\geqs 2$ in every case, then the proofs of Theorems \ref{thm:b(G)=2} and \ref{thm:r=1} are complete by combining Corollary \ref{c:b(G)=2} with Propositions \ref{p:r(G)>1_k>32}, \ref{p:k=3,4} and \ref{p:k=|T|-1,|T|-2}. By Proposition \ref{p:log}, it suffices to consider the cases where $5\leqs k\leqs 4\log |T|$. Recall that $r(G)\geqs 2$ if \eqref{e:prob} holds or $Q_1(G)+Q_2(G) < 1/2$ (see Lemmas \ref{l:prob} and \ref{l:r_prob_F}).

\begin{prop}
	\label{p:sporadic}
	The conclusions to Theorems \ref{thm:b(G)=2} and \ref{thm:r=1} hold when $T$ is a sporadic simple group.
\end{prop}

\begin{proof}
	As noted above, we may assume $5\leqs k\leqs 4\log|T|$. With the aid of {\sc Magma}, it is easy to check that \eqref{e:prob} holds for all $k$ in this range unless $T$ is one of the following groups:
	\begin{equation*}
	\mathrm{Suz},\mathrm{Co}_1,\mathrm{Co}_2,\mathrm{Fi}_{22},\mathrm{Fi}_{23},\mathrm{Fi}_{24}',\mathbb{B},\mathbb{M}.
	\end{equation*}

	Assume $T\in\{\mathrm{Suz},\mathrm{Co}_1,\mathrm{Co}_2,\mathrm{Fi}_{22},\mathrm{Fi}_{23},\mathrm{Fi}_{24}'\}$. Here we can construct $T$ as a permutation group in {\sc Magma} using the function \texttt{AutomorphismGroupSimpleGroup}, and we can then check that \eqref{e:prob} holds for $9\leqs k\leqs 4\log|T|$. The cases where $5\leqs k\leqs 8$ can be handled by random search using {\sc Magma} (see Remark \ref{r:magma}(i)).
	
	Finally, if $T \in\{\mathbb{B},\mathbb{M}\}$ then \eqref{e:prob} holds unless $k = 5$ or $(T,k) = (\mathbb{B},6)$. In each case, we can verify that $r(G)\geqs 2$ by random search as described in Remark \ref{r:magma}(ii), with the same centreless maximal subgroup $M$ of $T$ chosen in \eqref{e:sporadic_K}.
\end{proof}

\begin{prop}
	\label{p:alternating}
	The conclusions to Theorems \ref{thm:b(G)=2} and \ref{thm:r=1} hold when $T = A_n$ is an alternating group.
\end{prop}

\begin{proof}
	Once again, we may assume $5\leqs k\leqs 4\log|T|$. The cases where $n\in\{5,6\}$ can be easily handled using {\sc Magma}, so we also assume $n\geqs 7$. First assume $n\leqs k\leqs 4\log|T|$. With the aid of {\sc Magma}, it is easy to check \eqref{e:prob} holds for all $7\leqs n\leqs 29$. Note that $h(T) = (n-2)!$ and thus \eqref{e:prob_u=0_weak_cond} holds. By Lemma \ref{l:u=0}, it suffices to establish the inequality in \eqref{e:prob_u=0_weak} for $k_0 = n$. Thus, we only need to show that
	\begin{equation*}
	\left(\frac{n(n-1)}{2e}\right)^n > \frac{n(n!)^2}{2},
	\end{equation*}
	which holds for all $n\geqs 30$.
	
	Finally, let us assume $5\leqs k<n$ and define $Q_1(G)$ and $Q_2(G)$ as in \eqref{e:Q1} and \eqref{e:Q2}, respectively. Then
	\begin{equation*}
	Q_1(G) = (k!)^2|T|^{\frac{8}{3}-\frac{k}{2}-\frac{1}{2}\delta_{5,k}} + \frac{k!}{|T|^{\frac{4}{3}}} + \frac{k^4}{2|T|^{\frac{1}{3}}} < (6!)^2\left(\frac{2}{n!}\right)^{\frac{1}{3}} + \frac{2^{\frac{4}{3}}}{(n!)^{\frac{1}{3}}} + \frac{2^{\frac{1}{3}}n^4}{2(n!)^{\frac{1}{3}}}
	\end{equation*}
	and
	\begin{equation*}
	Q_2(G) = \left(\frac{|T|}{h(T)}\right)^{4-k} + {k\choose 2}|\Out(T)|\left(\frac{|T|}{h(T)}\right)^{3-k} < \frac{2}{n(n-1)} + 20 \left(\frac{2}{n(n-1)}\right)^2.
	\end{equation*}
	Given these bounds, it is easy to check that $Q_1(G)+Q_2(G) < 1/2$ for all $n\geqs 21$. Finally, for the cases where $7\leqs n\leqs 20$ and $5\leqs k < n$, one can use {\sc Magma} to check that either \eqref{e:prob} holds, or $Q_1(G) + Q_2(G) < 1/2$, or $\Hol(T)$ has at least $2$ regular orbits on $\mathscr{P}_k$ (for the latter, we use the random search approach as in Remark \ref{r:magma}(i)).
\end{proof}

To complete the proofs of Theorems \ref{thm:b(G)=2} and \ref{thm:r=1}, we may assume $T$ is a finite simple group of Lie type. First we consider some low rank groups, where $h(T)$ is small and Lemma \ref{l:u=k/2} can be applied.

\begin{lem}
	\label{l:PSL2}
	Suppose $T = \LL_2(q)$ and $5\leqs k \leqs 4\log|T|$. Then $r(G) \geqs  2$.
\end{lem}

\begin{proof}
	If $|T|\leqs 4080$ then $q\leqs 13$ and one can check the result using {\sc Magma}. More precisely, we first check \eqref{e:prob}, and if it fails, then we construct the permutation group $\Hol(T)$ on $T$ using the function \texttt{Holomorph} and use random search to find two $k$-subsets $S_1$ and $S_2$ of $T$ lying in distinct regular $\Hol(T)$-orbits (this is a viable approach since $|T|$ is small).
	
	Thus, we may assume $q\geqs 16$. First assume $k\geqs 6$ and set $k_0 = 6$. For $q\leqs 733$, one can check \eqref{e:prob} using {\sc Magma}. Assume $q > 733$ and note that $h(T)\leqs q^{1/2}(q-1)$ by Theorem \ref{t:fixity_Hol}, so \eqref{e:prob_u=k/2_weak_cond} holds. Moreover, as $|\Out(T)|\leqs 2\log q$, we can check that \eqref{e:prob_u=k/2_weak} holds if
	\begin{equation*}
	q^2(q^2-1)^2 > 16(\log q)^26^8e^{18},
	\end{equation*}
	which holds true for all $q > 733$. Now apply Lemma \ref{l:u=k/2}.
	
	To complete the proof, we assume $k = 5$. By Lemma \ref{l:prob_u_weak}, $r(G)\geqs 2$ if \eqref{e:prob_u_weak} holds for every $u\in\{0,1,2\}$. If $u = 2$, then \eqref{e:prob_u_weak} holds if
	\begin{equation*}
	q^{1/2}(q+1) > 5^4e^7\log q,
	\end{equation*}
	which holds for all $q > 48449$. With the same method, one can check that \eqref{e:prob_u_weak} holds for $u\in\{0,1\}$ if $q > 48449$. With the aid of {\sc Magma}, we see that \eqref{e:prob} holds for all $16\leqs q\leqs 48449$, unless $q\in\{16,25,49,81\}$, and the remaining cases can be handled using {\sc Magma} and random search, utilising the method in Remark \ref{r:magma}(i).
\end{proof}

\begin{lem}
	\label{l:PSL3_2B2_2G2}
	Suppose $T\in\{\LL_3^\e(q),{^2}B_2(q),{^2}G_2(q)\}$ and $5\leqs k\leqs 4\log|T|$. Then $r(G)\geqs 2$.
\end{lem}

\begin{proof}
	Note that $|T| > 4080$ and $h(T)^2 < 5|T|$ by Theorem \ref{t:fixity_Hol}. Thus, in view of Lemma \ref{l:u=k/2}, we only need to prove \eqref{e:prob_u=k/2_weak} for $k_0 = 5$. Assume $T = \LL_3^\e(q)$, so $|T|\geqs q^3(q^2-1)(q^3-1)/3$ and $|\Out(T)|\leqs 6\log q$. Thus, \eqref{e:prob_u=k/2_weak} holds if
	\begin{equation*}
	q^3(q^2-1)(q^3-1) > 3(6\log q)^25^7e^{15},
	\end{equation*}
	which is true for all $q > 73$. By applying the precise values of $h(T)$ and $|\Out(T)|$, we see that \eqref{e:prob} holds unless $\e = -$, $k = 5$ and $q\in\{3,5,8\}$, or $\e = +$ and
	\begin{equation*}
	(q,k)\in \{(3,5),(3,6),(4,5),(13,5)\},
	\end{equation*}
	all of which cases can be handled easily by random search as discussed in Remark \ref{r:magma}(i). We can apply the same method to the cases where $T = {^2}B_2(q)$ or ${^2}G_2(q)$, where \eqref{e:prob_u=k/2_weak} holds if $T\ne {^2}G_2(27)$, ${^2}B_2(8)$, ${^2}B_2(32)$ or ${^2}B_2(128)$ (we are excluding the group ${^2}G_2(3)'$ as noted in \eqref{e:iso_simple}). In the remaining four cases, one can check \eqref{e:prob} directly.
\end{proof}

\begin{prop}
	\label{p:exceptional}
	The conclusions to Theorems \ref{thm:b(G)=2} and \ref{thm:r=1} hold when $T$ is an exceptional group of Lie type.
\end{prop}

\begin{proof}
	Once again, by the previous results, we may assume $5\leqs k\leqs 4\log|T|$. In view of Lemma \ref{l:PSL3_2B2_2G2}, we may also assume $T\ne {^2}B_2(q)$ or ${^2}G_2(q)$. Note that
	\begin{equation*}
	\frac{|T|}{h(T)} > 10|\Out(T)| \geqs 10
	\end{equation*}
	and $|T| > \frac{1}{6}q^d$, where $d$ is as defined in Lemma \ref{l:Inndiag}.
	
	First assume $5\leqs k\leqs 8$. Then
	\begin{equation*}
	Q_2(G) < \frac{h(T)}{|T|} + 10|\Out(T)|\cdot \frac{h(T)^2}{|T|^2} < \frac{1}{10}+\frac{1}{10} =  \frac{1}{5}
	\end{equation*}
	and 
	\begin{equation*}
	Q_1(G) < \frac{(6!)^2}{|T|^{\frac{1}{3}}} + \frac{8!}{|T|^{\frac{4}{3}}} + \frac{8^4}{2|T|^{\frac{1}{3}}} < \frac{6^{\frac{1}{3}}(6!)^2}{q^{\frac{d}{3}}} + \frac{6^{\frac{4}{3}}\cdot 8!}{q^{\frac{4d}{3}}} + \frac{6^{\frac{1}{3}}8^4}{2q^{\frac{d}{3}}} < \frac{3}{10}
	\end{equation*}
	unless $T\in\{{^2}F_4(2)',{^3}D_4(2),{^3}D_4(3),{^3}D_4(4),F_4(2)\}$ or $T = G_2(q)$ for $q\leqs 23$. In this cases, one can check \eqref{e:prob} with the aid of {\sc Magma} unless $T = {^3}D_4(q)$ and $k = 5$, or $T = F_4(2)$ and $k\in\{5,6\}$. In the latter cases, we can do random search using {\sc Magma} as in Remark \ref{r:magma}(i).
	
	To complete the proof, we assume $9\leqs k\leqs 4\log |T|$. The groups with $q = 2$ can be handled by verifying \eqref{e:prob} directly, so we now assume $q\geqs 3$. We first prove \eqref{e:prob_u=0_weak} for $k_0 = 9$. By inspecting Table \ref{tab:fix(Hol(T))}, we have
	\begin{equation}\label{e:t/h_excep}
	2^9\left(\frac{|T|}{h(T)}\right)^9 > |T|^2q^{22}.
	\end{equation}
	For example, if $T = E_8(q)$, then
	\begin{equation*}
	\frac{|T|}{h(T)} = \frac{(q^{30}-1)(q^{24}-1)(q^{20}-1)}{(q^{10}-1)(q^6-1)} > \frac{1}{2}q^{58}
	\end{equation*}
	and $|T| < q^{248}$ by Lemma \ref{l:Inndiag}, which implies \eqref{e:t/h_excep}. Since $|\Out(T)|\leqs 6\log q$, it follows that \eqref{e:prob_u=0_weak} holds for $k_0 = 9$ if
	\begin{equation*}
	q^{22} > 48\log q\cdot (2e)^9
	\end{equation*}
	and one can check that this inequality holds for $q\geqs 3$. By Lemma \ref{l:u=0}, it suffices to prove \eqref{e:prob_u=0_weak_cond}. Here we only give a proof for the case where $T = G_2(q)$, as all the other cases are very similar. First note that $|T| = q^6(q^6-1)(q^2-1) < q^{14}$ and $h(T) = q^6(q^2-1) > \frac{1}{2}q^8$. Then \eqref{e:prob_u=0_weak_cond} holds if
	\begin{equation*}
	q^2 > 56^2(\log q)^2e,
	\end{equation*}
	which holds true for all $q > 907$. One can also check that \eqref{e:prob_u=0_weak_cond} holds for all $601 < q\leqs 907$. If $q\leqs 601$, then we can use the precise values of $|T|$, $h(T)$ and $|\Out(T)|$ to check \eqref{e:prob} for all $9\leqs k\leqs 4\log|T|$. This completes the proof.
\end{proof}

\begin{lem}
	\label{l:PSL4}
	Suppose $T = \LL_4^\e(q)$ and $5\leqs k\leqs 4\log |T|$. Then $r(G)\geqs 2$.
\end{lem}

\begin{proof}
	Recall that $h(T) = (2,q-\e)|\mathrm{PGSp}_4(q)|/(4,q-\e)$ by Theorem \ref{t:fixity_Hol}. First assume that $k\geqs 7$ and set $k_0 = 7$. For $q\leqs 89$, one can check \eqref{e:prob} with the aid of {\sc Magma}. Now assume $q > 89$. It is easy to see that
	\begin{equation*}
	q^5 > \max\{48(4e)^7\log q,4e\cdot 60^2(\log q)^2\},
	\end{equation*}
	which implies \eqref{e:prob_u=0_weak} and \eqref{e:prob_u=0_weak_cond}.
	
	Now assume $k\in\{5,6\}$. Note that $|T|/h(T)> 10|\Out(T)| \geqs 10$, so $Q_2(G) < \frac{1}{5}$. Moreover,
	\begin{equation*}
	Q_1(G) < \frac{(6!)^2}{|T|^{\frac{1}{3}}} + \frac{6!}{|T|^{\frac{4}{3}}} + \frac{6^4}{2|T|^{\frac{1}{3}}},
	\end{equation*}
	so we have $Q_1(G) < \frac{3}{10}$ if $q\geqs 19$ and thus $Q_1(G)+Q_2(G)<1/2$. Finally if $q\leqs 17$ then we can use {\sc Magma} (via random search as in Remark \ref{r:magma}(i)) to check that $r(G)\geqs 2$.
\end{proof}

\begin{lem}
	\label{l:PSp4_odd}
	Suppose $T = \PSp_4(q)$ and $5\leqs k\leqs 4\log|T|$. Then $r(G)\geqs  2$.
\end{lem}

\begin{proof}
	As noted in \eqref{e:iso_simple}, we assume $q\geqs 3$. First assume $k\geqs 6$. It can be checked using {\sc Magma} that \eqref{e:prob} holds for $q\leqs 607$, unless $(k,q) = (6,3)$, in which case we can verify the result using {\sc Magma} and random search as in Remark \ref{r:magma}(i). Now assume $q > 607$. By applying the bounds $|T| < q^{10}$, $h(T) > q^6/2$ and $q^4/2 < |T|/h(T) < 2q^4$, we see that \eqref{e:prob_u=0_weak} holds for $k_0 = 6$ if
	\begin{equation*}
	q^4 > 6(2e)^6\log q,
	\end{equation*}
	while \eqref{e:prob_u=0_weak_cond} holds if
	\begin{equation*}
	q^2 > 40^2(\log q)^2e.
	\end{equation*}
	Note that both inequalities hold for all $q > 607$.
	
	Finally, assume $k = 5$. Once again, we have $|T|/h(T) > 10|\Out(T)|\geqs 10$ and thus $Q_2(G) < \frac{1}{5}$. Additionally,
	\begin{equation*}
	Q_1(G) = \frac{(5!)^2}{|T|^{\frac{1}{3}}}+\frac{5!}{|T|^{\frac{4}{3}}}+\frac{5^4}{2|T|^{\frac{1}{3}}} < \frac{3}{10}
	\end{equation*}
	for all $q\geqs 27$. The remaining groups with $q\leqs 25$ can be handled with the aid of {\sc Magma} via random search (see Remark \ref{r:magma}(i)).
\end{proof}

\begin{prop}
	\label{p:classical}
	The conclusions to Theorems \ref{thm:b(G)=2} and \ref{thm:r=1} hold when $T$ is a classical group.
\end{prop}

\begin{proof}
	Let $T$ be a classical group over $\mathbb{F}_q$ and let $n$ be the dimension of the natural module. Note that $|T| > \frac{1}{8}q^{n(n-1)/2}$ by Lemma \ref{l:Inndiag}. As explained above, we may assume $5\leqs k\leqs 4\log |T|$. In addition, we may also assume $n\geqs 5$ by Lemmas \ref{l:PSL2}, \ref{l:PSL3_2B2_2G2}, \ref{l:PSL4} and \ref{l:PSp4_odd}. Then
	\begin{equation*}
	\frac{|T|}{h(T)} > 10|\Out(T)|\geqs 10
	\end{equation*}
	by inspecting Table \ref{tab:fix(Hol(T))}, and thus
	\begin{equation*}
	Q_2(G) < \frac{h(T)}{|T|} + 10|\Out(T)|\cdot \frac{h(T)^2}{|T|^2} < \frac{1}{10}+\frac{1}{10} = \frac{1}{5}.
	\end{equation*}
	
	First assume $5\leqs k\leqs n+3$. Then
	\begin{equation*}
	\begin{aligned}
	Q_1(G) &< \frac{(6!)^2}{|T|^{\frac{1}{3}}} + \frac{(n+3)!}{|T|^{\frac{4}{3}}} + \frac{(n+3)^4}{2|T|^{\frac{1}{3}}} \\&< \frac{8^{\frac{1}{3}}(6!)^2}{q^{\frac{n(n-1)}{6}}} + \frac{8^{\frac{4}{3}}(n+3)!}{q^{\frac{2n(n-1)}{3}}} + \frac{8^{\frac{1}{3}}(n+3)^4}{2q^{\frac{n(n-1)}{6}}} =: Q(n,q).
	\end{aligned}
	\end{equation*}
	Evidently, $Q(n,q)$ is a decreasing function of $q$. In addition, if $q$ is fixed, then each summand is a decreasing function of $n$. Thus, $Q(n,q)$ is also decreasing of $n$. Note that $Q(n,q) < \frac{3}{10}$ if
	\begin{equation*}
	(n,q)\in\{(12,2),(10,3),(9,4),(8,7),(7,9),(6,23),(5,97)\} =: \mathcal{B}.
	\end{equation*}
	Hence, we only need to consider the cases where $n < n_0$ or $q<q_0$ for some $(n_0,q_0)\in\mathcal{B}$. For these groups, we can show that $r(G)\geqs 2$ either by checking $Q_1(G) + Q_2(G) < 1/2$ or \eqref{e:prob}, or by random search as explained in Remark \ref{r:magma}(i). This shows that $r(G) \geqs 2$ if $5\leqs k\leqs n+3$.
	
	To complete the proof, assume $n+4\leqs k\leqs 4\log|T|$ and let $k_0 = n+4$. We first consider the case where $T = \LL_n^\e(q)$. Note that $|T| < q^{n^2-1}$ and 
	\begin{equation*}
	\frac{|T|}{h(T)} \geqs \frac{|\PGL_n^\e(q)|}{|\GU_{n-1}(q)|} > \frac{1}{2}q^{2n-2}
	\end{equation*}
	by Lemma \ref{l:Inndiag} and Theorem \ref{t:fixity_Hol}. Hence, \eqref{e:prob_u=0_weak} holds if
	\begin{equation*}
	q^{6n-8} > 2(n+4)(2e)^{n+4}
	\end{equation*}
	since $|\Out(T)|\leqs 2(q+1)\log q < 2q^2$. This inequality holds if $q\geqs 3$ or $n\geqs 7$, while we can check \eqref{e:prob_u=0_weak} directly when $(n,q) = (5,2)$ or $(6,2)$. Thus, we have \eqref{e:prob_u=0_weak} for all $n\geqs 5$ and $q\geqs 2$. By Lemma \ref{l:u=0}, it suffices to prove \eqref{e:prob_u=0_weak_cond}. To do this, first note that
	\begin{equation*}
	h(T) \geqs q^{2n-3} |\PGL_{n-2}^\e(q)| > \frac{1}{2}q^{2n-3}q^{(n-2)^2-1} = \frac{1}{2}q^{n^2-2n}
	\end{equation*}
	by Lemma \ref{l:Inndiag} and Theorem \ref{t:fixity_Hol}, so \eqref{e:prob_u=0_weak_cond} holds if
	\begin{equation*}
	q^{n^2-4n-1} > 32e(n^2-1)^2
	\end{equation*}
	since $\log q < q$. One can easily check that the above inequality holds for all $n\geqs 5$ and $q\geqs 2$, unless $n = 5$ and $q\leqs 13$, or $(n,q) = (6,2)$, in which cases we can verify \eqref{e:prob_u=0_weak_cond} directly. This completes the proof for linear and unitary groups.
	
	Next assume $T = \PSp_n(q)$ with $n\geqs 6$. Here $|T| < q^{n(n+1)/2}$ by Lemma \ref{l:Inndiag} and
	\begin{equation*}
	\frac{|T|}{h(T)} = \frac{q^n-1}{(2,q-1)} > q^{n-1}.
	\end{equation*}
	Since $|\Out(T)|\leqs 2\log q$, we see that \eqref{e:prob_u=0_weak} holds if
	\begin{equation*}
	q^{2n-4} > 2\log q\cdot (n+4)e^{n+4}
	\end{equation*}
	and one checks that this inequality is valid unless $q = 2$ and $n\leqs 28$, $n = 6$ and $q\leqs 5$, or $(n,q) \in \{(8,3),(10,3)\}$. In these remaining cases, one can also check \eqref{e:prob_u=0_weak} by applying the precise values of $|T|$, $h(T)$ and $|\Out(T)|$, so as above, it just remains to verify \eqref{e:prob_u=0_weak_cond}. To do this, first note that
	\begin{equation*}
	h(T) = q^{n-1}|\Sp_{n-2}(q)| > \frac{1}{2}q^{n(n-1)/2},
	\end{equation*}
	so it suffices to show that
	\begin{equation*}
	q^{n(n-3)/2} > 8en^2(n+1)^2(\log q)^2.
	\end{equation*}
	The latter holds unless $(n,q) = (6,2)$ or $(6,3)$, in which cases one can directly verify \eqref{e:prob_u=0_weak_cond}. The result now follows from Lemma \ref{l:u=0}.
	
	Finally, assume $T = \POmega_n^\e(q)$ is an orthogonal group, so $n\geqs 7$, and $q$ is odd if $n$ is odd. In this setting, $|T| < q^{n(n-1)/2}$ and 
	\begin{equation*}
	\frac{|T|}{h(T)} > \frac{1}{2}q^{n-1}
	\end{equation*}
	by Lemma \ref{l:Inndiag} and Theorem \ref{t:fixity_Hol}. In addition, \eqref{e:prob_u=0_weak} holds if
	\begin{equation*}
	q^{4n-4} > 24\log q\cdot (n+4)(2e)^{n+4}
	\end{equation*}
	since $|\Out(T)|\leqs 24\log q$, which is valid unless $q = 2$ and $n\leqs 14$. In the remaining cases, \eqref{e:prob_u=0_weak} can be checked directly. Finally, to prove \eqref{e:prob_u=0_weak_cond}, note that
	\begin{equation*}
	h(T) > \frac{1}{4}q^{(n-1)(n-2)/2}
	\end{equation*}
	by Lemma \ref{l:Inndiag} and Theorem \ref{t:fixity_Hol}, so we only need to show that
	\begin{equation*}
	q^{(n-1)(n-4)/2} > 32en^2(n-1)^2(\log q)^2.
	\end{equation*}
	This holds unless $(n,q) = (7,3)$ or $(8,2)$, and in these special cases we can verify \eqref{e:prob_u=0_weak_cond} directly. We now complete the proof by applying Lemma \ref{l:u=0}.
\end{proof}

We conclude that the proofs of Theorems \ref{thm:b(G)=2} and \ref{thm:r=1} are complete by combining Propositions \ref{p:sporadic}, \ref{p:alternating}, \ref{p:exceptional} and \ref{p:classical}. As noted in the beginning of this section, the proof of Theorem \ref{thm:Hol} is also complete.

\section{Proof of Theorem \ref{thm:main}}

\label{s:proof_main}

In this section, we prove Theorem \ref{thm:main}, which is our main result. By Theorems \ref{thm:b(G)=2}, \ref{t:F}, and Proposition \ref{p:k=|T|-1,|T|-2}, we only need to consider the cases where $k =2$, or $k > |T|$ and $P\in\{A_k,S_k\}$.

\subsection{The groups with $k = 2$}

We first consider the case where $k=2$. As recorded in Theorem \ref{t:F}(ii), we have $b(G) = 3$ if $P = 1$, and $b(G)\in\{3,4\}$ if $P = S_2$.

\begin{lem}
	\label{l:k=2_equiv}
	Suppose $W = T^2.(\Out(T)\times S_2)$ and $s,t\in T$. Then $\{D,D(1,\varphi_s),D(1,\varphi_t)\}$ is a base for $W$ if and only if:
	\begin{enumerate}\addtolength{\itemsep}{0.2\baselineskip}
		\item[\rm (i)] $C_{\Aut(T)}(s)\cap C_{\Aut(T)}(t) = 1$; and
		\item[\rm (ii)] there is no $\alpha\in\Aut(T)$ such that $s^\alpha = s^{-1}$ and $t^\alpha = t^{-1}$.
	\end{enumerate}
\end{lem}

\begin{proof}
	This can be deduced from \cite[Lemma 3.5]{LMM_IBIS}.
\end{proof}

The following is \cite[Theorem 1.1]{LL_gen}.

\begin{thm}
	\label{t:LL_t:1.1}
	Suppose $T$ is not $A_7$, $\LL_2(q)$ or $\LL_3^\e(q)$ for some prime power $q$. Then there exists a generating pair $(s,t)$ of $T$ such that $|s| = 2$ and there is no $\a\in\Aut(T)$ with $s^\a = s^{-1}$ and $t^\a = t^{-1}$.
\end{thm}

It has been proved recently that each of the excluded groups $A_7$, $\LL_2(q)$ and $\LL_3^\e(q)$ does not have a generating pair described as in Theorem \ref{t:LL_t:1.1} (see \cite[Theorem 1.3]{J_pair}).

\begin{prop}
	\label{p:k=2}
	The conclusion to Theorem \ref{thm:main} holds for $k = 2$.
\end{prop}

\begin{proof}
	Recall that $b(G) = 3$ if $P = 1$ by Theorem \ref{t:F}(ii). Thus, we may assume $P = S_2$. By Lemma \ref{l:k=2_equiv} and Theorem \ref{t:LL_t:1.1}, we have $b(G) = 3$ if $T\notin\{A_7,\LL_2(q),\LL_3^\e(q)\}$. The case where $T = A_7$ can be easily handled using {\sc Magma} and we deduce that $b(W) = 3$.
	
	Assume $T = \LL_2(q)$, so $\Aut(T) = \PGammaL_2(q)$. If $q\in\{4,5,9\}$ then $T$ is isomorphic to $A_5$ or $A_6$ and we can prove the proposition with the aid of {\sc Magma}, noting that $b(W) = 4$ and $b(G) = 3$ if $G<W$. Now we consider the cases where $q\notin\{4,5,9\}$. Let $s$ be an element in $T$ of order $(q-1)/(2,q-1)$. Then we have $N_{\PGL_2(q)}(\la s\ra)\cong D_{2(q-1)}$ and
	\begin{equation*}
	C_{\PGammaL_2(q)}(s) = C_{\PGL_2(q)}(s) \cong C_{q-1}.
	\end{equation*}
	One can show that $\PGL_2(q)$ is base-two on $[\PGL_2(q):N_{\PGL_2(q)}(\la s\ra)]$ (see for example \cite[Lemma 4.7]{B_sol}), which implies that there exists $g\in \PGL_2(q)$ such that
	\begin{equation*}
	N_{\PGL_2(q)}(\la s\ra)\cap N_{\PGL_2(q)}(\la s^g\ra) = 1.
	\end{equation*}
	We claim that the pair $(s,s^g)$ satisfies the conditions (i) and (ii) in Lemma \ref{l:k=2_equiv}. Indeed, (i) is clear since $C_{\PGammaL_2(q)}(s) = C_{\PGL_2(q)}(s)$ and so it suffices to check (ii). To do this, first note that there exists an element $\beta\in\PGL_2(q)$ such that $s^\beta = s^{-1}$. Therefore, if $\alpha\in\PGammaL_2(q)$ and $s^\alpha = s^{-1}$, then $\alpha$ is contained in the coset $C_{\PGammaL_2(q)}(s)\beta$. In particular, $\alpha\in\PGL_2(q)$ as $C_{\PGammaL_2(q)}(s) \leqs \PGL_2(q)$. It follows that $\alpha\in N_{\PGL_2(q)}(\la s\ra)$. Similarly, if $(s^g)^\alpha = (s^g)^{-1}$ then $\alpha\in N_{\PGL_2(q)}(\la s^g\ra)$, which yields $\alpha = 1$. This leads to a contradiction as $s$ is not an involution. Thus, $b(G) = 3$ by Lemma \ref{l:k=2_equiv}.
	
	Finally, let us turn to the case where $T = \LL_3^\e(q)$. One can easily check the proposition for $q = 3$ using {\sc Magma}, and we will assume $q\ne 2$ as $\LL_3(2)\cong \LL_2(7)$ has been handled above, and $\UU_3(2)$ is not simple. Let $N$ be a subgroup of $\Aut(T)$ of type $\GL_1^\e(q^3)$. Then $N$ is a maximal subgroup of $\Aut(T)$, and $N\cap T\cong \langle s\rangle{:}C_3$, where $|s| = (q^3-\e)/d(q-\e)$ and $d = (3,q-\e)$ (see \cite[Proposition 4.3.6]{KL_classical}). Note that $N=N_{\Aut(T)}(\la s \ra)$. By \cite[Lemma 6.4]{B_sol}, $\Aut(T)$ is base-two on $[\Aut(T):N]$, so there exists $g\in \Aut(T)$ such that $N_{\Aut(T)}(\la s\ra)\cap N_{\Aut(T)}(\la s^g\ra) = 1$. By repeating the above argument, we deduce that the conditions (i) and (ii) in Lemma \ref{l:k=2_equiv} are satisfied if we take $t = s^g$, which completes the proof.
\end{proof}

The following corollary will be useful in Section \ref{ss:|T|^ell-2..|T|^ell}.

\begin{cor}
	\label{c:pair}
	Suppose $T\notin\{A_5,A_6\}$. Then there exist $x,y\in T$ such that $C_{\Aut(T)}(x)\cap C_{\Aut(T)}(y) = 1$ and there is no $\a\in\Aut(T)$ with $(x,y)^\a = (x^{-1},y^{-1})$.
\end{cor}

\begin{proof}
	Proposition \ref{p:k=2} implies that the group $W = T^2.(\Out(T)\times S_2)$ has a base of size $3$. Now apply Lemma \ref{l:k=2_equiv}.
\end{proof}

\subsection{The groups with $|T|^{\ell-1}<k\leqs |T|^{\ell}-3$}

\label{ss:|T|^ell-3}

Next, we assume $P\in\{A_k,S_k\}$ and $|T|^{\ell-1}<k\leqs |T|^{\ell}-3$ for some integer $\ell\geqs 1$. The groups with $\ell = 1$ have been handled in Theorem \ref{thm:b(G)=2} and Proposition \ref{p:k=2}, so we may assume $\ell\geqs 2$. In this setting, Theorem \ref{t:F}(iii) implies that $b(G)\in \{\ell+1,\ell+2\}$, and we will show that $b(G) = \ell+1$ by constructing a base for $G$ of size $\ell+1$. We may assume $G = T^k.(\Out(T)\times S_k)$ throughout.

For any partition $\mathcal{P}$ of $[k]$ into $|T|$ parts, where some parts are allowed to be empty, we may write $\mathcal{P} = \{\mathcal{P}_t:t\in T\}$. Recall that $\Hol(T,S)$ is the setwise stabiliser of $S\subseteq T$ in $\Hol(T)$.

\begin{lem}\label{l:cal_P_exists}
	If $\ell\geqs 2$ and $|T|^{\ell-1}<k\leqs |T|^{\ell}-3$, then there exists a partition $\mathcal{P} = \{\mathcal{P}_t:t\in T\}$ of $[k]$ satisfying the following properties:
	\begin{enumerate}\addtolength{\itemsep}{0.2\baselineskip}
		\item[{\rm (P1)}] $|\mathcal{P}_t|\leqs |T|^{\ell-1}$ for all $t\in T$.
		\item[{\rm (P2)}] $|\mathcal{P}_1|\ne 0$ and $\Hol(T,S) = 1$, where
		\begin{equation*}
		S = \{t\in T :|\mathcal{P}_t| = |\mathcal{P}_1|\}.
		\end{equation*}
		\item[{\rm (P3)}] There exists $x\in T^\#$ such that $|\mathcal{P}_x|\in\{1,|T|^{\ell-1}-1\}$.
	\end{enumerate}
\end{lem}

\begin{proof}
	First assume $|T|^\ell-2|T|^{\ell-1}<k\leqs |T|^\ell-3$. In view of Theorem \ref{thm:Hol}, let $S$ be a subset of $T$ containing $1$ with $|S| = |T|-3$ and $\Hol(T,S) = 1$, and let $\{x_1,x_2,x_3\} = T\setminus S$. Now define $\mathcal{P} = \{\mathcal{P}_t:t\in T\}$, where $|\mathcal{P}_t| = |T|^{\ell-1}$ if $t\in S$, and $|\mathcal{P}_{x_i}| \leqs |T|^{\ell-1} - 1$ with $|\mathcal{P}_{x_1}| = |T|^{\ell-1}-1$ and
	\begin{equation*}
	|\mathcal{P}_{x_2}|+|\mathcal{P}_{x_3}| = k-(|T|-2)|T|^{\ell-1}+1.
	\end{equation*}
	Note that such a partition exists since
	\begin{equation*}
	2\leqs k-(|T|-2)|T|^{\ell-1}+1\leqs 2|T|^{\ell-1}-2.
	\end{equation*}
	It is then easy to check that $\mathcal{P}$ satisfies the conditions (P1)--(P3).
	
	Now assume $3|T|^{\ell-1} < k\leqs |T|^\ell - 2|T|^{\ell-1}$. Then there exists an integer $m$ such that $3\leqs m\leqs |T|-3$ and $m|T|^{\ell-1}< k \leqs  (m+1)|T|^{\ell-1}$. By Theorem \ref{thm:Hol}, there exists a subset $S\subseteq T$ containing $1$ with $|S| = m$ and $\Hol(T,S)=1$. Let $x_1,x_2\in T\setminus S$ and define $\mathcal{P} = \{\mathcal{P}_t:t\in T\}$, where $|\mathcal{P}_t| = |T|^{\ell-1}$ if $t\in S$, $|\mathcal{P}_{x_1}| = 1$ and $|\mathcal{P}_{x_2}| = k-m|T|^{\ell-1}-1$, noting that $0\leqs k-m|T|^{\ell-1}-1< |T|^{\ell-1}$. One can check (P1)--(P3) easily.
	
	To complete the proof, we assume $|T|^{\ell -1}<k\leqs 3|T|^{\ell -1}$ and let $S = \{t_1,t_2,t_3\}\subseteq T$ be such that $t_1 = 1$ and $\Hol(T,S) = 1$. In this setting, let $x_1,x_2,x_3\in T\setminus S$ and define $\mathcal{P} = \{\mathcal{P}_t:t\in T\}$, where $|\mathcal{P}_{t_i}| = 1$, and $|\mathcal{P}_{x_i}|\leqs |T|^{\ell-1}$ with  $|\mathcal{P}_{x_i}|\ne 1$ and $|\mathcal{P}_{x_1}|+|\mathcal{P}_{x_2}|+|\mathcal{P}_{x_3}| = k-3$. We conclude the proof by noting that $\mathcal{P}$ satisfies the conditions (P1)--(P3).
\end{proof}

For the remainder of this subsection, $\mathcal{P} = \{\mathcal{P}_t:t\in T\}$ is a partition of $[k]$ satisfying the conditions in Lemma \ref{l:cal_P_exists}, where $S\subseteq T$ and $x\in T^\#$ are as described in (P2) and (P3), respectively. Define $\mathbf{a}_0 = (\varphi_{t_{0,1}},\dots,\varphi_{t_{0,k}}) \in \Inn(T)^k$ by $t_{0,j} = t$ if $j\in \mathcal{P}_t$.

\begin{lem}
	\label{l:a0}
	Suppose $(\alpha,\dots,\alpha)\pi\in G_{D\mathbf{a}_0}$. Then $\alpha = 1$ and $\pi\in P_{(\mathcal{P})}$.
\end{lem}

\begin{proof}
	First note that there exists a unique $g\in T$ such that $t_{0,j}^\a = gt_{0,j^\pi}$ for all $j\in [k]$, and we have $\pi\in P_{\{\mathcal{P}\}}$ by Lemma \ref{l:l:3.4_diag_ext}(i). This implies that $\pi$ fixes the set $\{\mathcal{P}_t:t\in S\}$, and thus $g^{-1}t^\a\in S$ if $t\in S$, whence $g^{\a^{-1}}\a\in\Hol(T,S) = 1$. It follows that $g = 1$ and $\a = 1$, so $t_{0,j} = t_{0,j^\pi}$ for all $j\in [k]$, which concludes the proof.
\end{proof}

Write $T^{\ell-1} = \{\mathbf{b}_1,\dots,\mathbf{b}_{|T|^{\ell-1}}\}$, where $\mathbf{b}_h = (a_{1,h},\dots,a_{\ell-1,h})$. If $|\mathcal{P}_x| = 1$, then we may assume $\mathbf{b}_1 = (1,\dots,1)$, and if $|\mathcal{P}_x| = |T|^{\ell-1}-1$, we assume $\mathbf{b}_{|T|^{\ell-1}} = (1,\dots,1)$. Let $1\leqs i\leqs \ell-1$ and define $\mathbf{a}_i = (\varphi_{t_{i,1}}, \dots,\varphi_{t_{i,k}})\in\Inn(T)^k$, where $t_{i,j} = a_{i,h}$ if $j$ is the $h$-th smallest number in $\mathcal{P}_t$.
Define $X_{i,t}:=\{j\in\mathcal{P}_x:t_{i,j} = t\}$.

\begin{lem}
	\label{l:g}
	For any $t\in T^\#$ and $i\in\{1,\dots,\ell-1\}$, we have $|X_{i,t}|\ne |X_{i,1}|$.
\end{lem}

\begin{proof}
	If $|\mathcal{P}_x| = 1$, then $\mathbf{b}_1 = (1,\dots,1)$, so $|X_{i,1}| = 1$ and $|X_{i,t}| = 0$ for all $t\in T^\#$. And if $|\mathcal{P}_x| = |T|^{\ell-1}$, then $\mathbf{b}_{|T|^{\ell-1}} = (1,\dots,1)$, which implies that $|X_{i,1}| = |T|^{\ell-1}-1$ and $|X_{i,t}| = |T|^{\ell-1}$ for all $t\in T^\#$.
\end{proof}

\begin{prop}
	\label{p:|T|^ell-3}
	If $\ell\geqs 2$, $P\in\{A_k,S_k\}$ and $|T|^{\ell-1}<k\leqs |T|^\ell-3$, then $b(G) = \ell+1$.
\end{prop}

\begin{proof}
	As noted above, it suffices to show that $\Delta=\{D,D\mathbf{a}_0,D\mathbf{a}_1\dots,D\mathbf{a}_{\ell-1}\}$ is a base for $G$. Suppose $(\a,\dots,\a)\pi\in G_{(\Delta)}$. By Lemma \ref{l:a0}, we have $\a = 1$ and $\pi\in P_{(\mathcal{P})}$. Note that for any $i\in\{1,\dots,\ell-1\}$, there exists a unique $g_i\in T$ such that $t_{i,j} = g_it_{i,j^\pi}$ for any $j\in[k]$. Now $j\in X_{i,1}$ if and only if $j^\pi\in X_{i,g_i^{-1}}$. This implies that $g_i = 1$ by Lemma \ref{l:g}, and hence $t_{i,j} = t_{i,j^\pi}$ for all $i\in\{1,\dots,\ell-1\}$ and $j\in [k]$.
	
	From the definition of $\mathbf{a}_i$, we see that if $j,j'\in\mathcal{P}_t$ and $j\ne j'$, then there exists $i\in\{1,\dots,\ell-1\}$ such that $t_{i,j}\ne t_{i,j'}$. This yields $j^\pi\ne j'$, so $j^\pi = j$ since $\pi\in P_{\{\mathcal{P}_t\}}$. That is, $\pi\in P_{(\mathcal{P}_t)}$ for all $t\in T$, whence we have $\pi = 1$.
\end{proof}

\subsection{The groups with $|T|^{\ell}-2\leqs k\leqs |T|^\ell$}

\label{ss:|T|^ell-2..|T|^ell}

To complete the proof of Theorem \ref{thm:main}, we turn to the cases where $k\in\{|T|^\ell-2,|T|^\ell-1, |T|^\ell\}$ and $P\in\{A_k,S_k\}$ for some integer $\ell\geqs 1$. The groups with $\ell = 1$ have been treated previously, and we record the result as follows.

\begin{prop}
	\label{p:l=1}
	If $k\in\{|T|-2,|T|-1,|T|\}$ and $P\in\{A_k,S_k\}$, then
	\begin{equation*}
	b(G) = 
	\begin{cases}
	2 & \mbox{if $k\in\{|T|-2,|T|-1\}$ and $S_k\not\leqs G$;}\\
	3 & \mbox{otherwise.}
	\end{cases}
	\end{equation*}
\end{prop}

\begin{proof}
	Combine Theorem \ref{t:F}(iii) and Proposition \ref{p:k=|T|-1,|T|-2}.
\end{proof}

From now on, we assume $\ell\geqs 2$. We start with the groups with $S_k\not\leqs G$.

\begin{lem}
	\label{l:Ak_notle_G}
	Suppose $k \in  \{|T|^\ell-2,|T|^\ell-1,|T|^\ell\}$ with $\ell\geqs 2$, $P\in\{A_k,S_k\}$ and $S_k\not\leqs G$. Then $b(G) = \ell+1$.
\end{lem}

\begin{proof}
	In view of Theorem \ref{t:F}(iii), it suffices to construct a base for $G$ of size $\ell+1$. Note that $A_k\leqs G$ by Corollary \ref{c:Ak<G}, so $G$ does not contain any transposition in $S_k$.
	
	By Corollary \ref{cor:Aut}, there exist $x,y\in T^\#$ such that $\Aut(T,\{x,y\}) = 1$. Let $\mathcal{P} = \{\mathcal{P}_t:t\in T\}$ be a partition of $[k]$ with $|\mathcal{P}_1| = |T|^{\ell-1}+1$, $|\mathcal{P}_x| = |T|^{\ell-1}-1$ and $|\mathcal{P}_t| = |T|^{\ell-1}$ if $t\notin\{1,x,y\}$. Thus, $|\mathcal{P}_y| = |T|^{\ell-1}-m$ if $k = |T|^\ell-m$, where $m\in\{0,1,2\}$. Now define $\mathbf{a}_0 = (\varphi_{t_{0,1}},\dots,\varphi_{t_{0,k}})\in\Inn(T)^k$ by setting $t_{0,j} = t$ if $j\in\mathcal{P}_t$. We also write $T^{\ell-1} = \{\mathbf{b}_1,\dots,\mathbf{b}_{|T|^{\ell-1}}\}$, where $\mathbf{b}_h = (a_{1,h},\dots,a_{\ell-1,h})$, and we may assume $\mathbf{b}_{|T|^{\ell-1}} = (y,\dots,y)$. Define $\mathbf{a}_i = (\varphi_{t_{i,1}},\dots,\varphi_{t_{i,k}})\in\Inn(T)^k$ for $i\in\{1,\dots,\ell-1\}$, where
	\begin{equation*}
	t_{i,j} = 
	\begin{cases}
	a_{i,h} & \mbox{if $j$ is the $h$-th smallest number in $\mathcal{P}_t$;}\\
	1 & \mbox{if $j$ is the largest number in $\mathcal{P}_1$.}
	\end{cases}
	\end{equation*}
	We claim that $\Delta=\{D,D\mathbf{a}_0,D\mathbf{a}_1,\dots,D\mathbf{a}_{\ell-1}\}$ is a base for $G$.
	
	Suppose $(\a,\dots,\a)\pi\in G_{(\Delta)}$. By Lemma \ref{l:l:3.4_diag_ext}, we have $\pi\in P_{\{\mathcal{P}\}}$ and $t_{0,j}^\a = t_{0,j^\pi}$ for all $j\in [k]$. We first prove that $\a = 1$. To see this, note that if $k\in\{|T|^\ell-2,|T|^\ell-1\}$, then $\pi\in P_{\{\mathcal{P}_x\cup\mathcal{P}_y\}}$, which implies that $\a\in\Aut(T,\{x,y\})$, and thus $\a = 1$ since $\Aut(T,\{x,y\}) = 1$. Now assume $k = |T|^\ell$. Then $\pi\in P_{\{\mathcal{P}_x\}}$ and thus $\a\in C_{\Aut(T)}(x)$. Note that for each $i\in\{1,\dots,\ell-1\}$, $1$ appears exactly $|T|^{\ell-1}+1$ times in the entries of $\mathbf{a}_i$, while $\varphi_y$ appears exactly $|T|^{\ell-1}-1$ times and every other element appears exactly $|T|^{\ell-1}$ times. By arguing as above, we have $t_{i,j}^\a = t_{i,j^\pi}$ for all $i \in \{1,\dots,\ell-1\}$, which implies that $\a\in C_{\Aut(T)}(y)$, and so $\a = 1$ since $\Aut(T,\{x,y\}) = 1$.
	
	Finally, observe that there exists a unique pair $\{j_1,j_2\}$ of elements in $[k]$ such that $j_1\ne j_2$ and $t_{i,j_1} = t_{i,j_2}$ for all $i\in\{0,\dots,\ell-1\}$, where we have $t_{i,j_1} = t_{i,j_2} = 1$. For each $i$, there exists a unique element $g_i\in T$ such that $t_{i,j} = g_it_{i,j^\pi}$ for all $j\in[k]$, so $t_{i,j_1^\pi} = t_{i,j_2^\pi} = g_i^{-1}$. Since $\pi\in P_{\{\mathcal{P}_1\}}$, it follows that $g_i = 1$ and so $t_{i,j} = t_{i,j^\pi}$ for all $j\in[k]$. It is then easy to see that $\pi\in\la (j_1,j_2)\ra$, and thus $\pi = 1$ as $G$ does not contain any transposition in $S_k$.
\end{proof}

\begin{prop}
	\label{p:|T|^ell-1_|T|^ell}
	If $\ell\geqs 2$, $P\in\{A_k,S_k\}$ and $k\in\{|T|^\ell-1,|T|^\ell\}$, then
	\begin{equation*}
	b(G) = 
	\begin{cases}
	\ell+1 & \mbox{if $S_k\not\leqs G$;}\\
	\ell+2 & \mbox{if $S_k\leqs G$.}
	\end{cases}
	\end{equation*}
\end{prop}

\begin{proof}
	See Theorem \ref{t:F}(iii) for the groups with $S_k\leqs G$ and Lemma \ref{l:Ak_notle_G} for $S_k\not\leqs G$.
\end{proof}

To conclude this section, we turn to the groups with $k = |T|^\ell-2$ and $S_k\leqs G$. The case where $\ell = 2$ requires special attention.

\begin{lem}
	\label{l:|T|2-2,A5A6}
	Suppose $k = |T|^2-2$, $T \in \{A_5,A_6\}$ and $G = T^k.(\Out(T)\times S_k)$. Then $b(G) = 4$.
\end{lem}

\begin{proof}
	First note by Theorem \ref{t:F}(iii) that we have $b(G)\in\{3,4\}$, so it suffices to show that there is no base for $G$ of size $3$.
	
	We argue by contradiction and suppose $\Delta = \{D,D\mathbf{a}_0,D\mathbf{a}_1\}$ is a base for $G$, where $\mathbf{a}_i = (\varphi_{t_{i,1}},\dots,\varphi_{t_{i,k}})\in\Inn(T)^k$. If $\varphi_t$ appears at least $|T|+1$ times in the entries of $\mathbf{a}_0$ for some $t$, then there exist $j,j'\in[k]$ such that $j\ne j'$, $t_{0,j} = t_{0,j'} = t$ and $t_{1,j} = t_{1,j'}$, which implies that $G_{(\Delta)}$ contains the transposition $(j,j')$. Thus, we may assume that each $\varphi_t$ appears at most $|T|$ times in the entries of $\mathbf{a}_0$. The same argument holds for $\mathbf{a}_1$. It follows that the set
	\begin{equation*}
	S_i = \{t\in T:\varphi_t\mbox{ appears exactly $|T|$ times in the entries of $\mathbf{a}_i$}\}
	\end{equation*}
	has size at least $|T|-2$, so $|S_i|\in\{|T|-2,|T|-1\}$.
	
	First, assume either $|S_0|$ or $|S_1|$ is equal to $|T|-1$, say $|S_0| = |T|-1$ and $1\notin S_0$. For the same reason as above, for any $j,j'$ such that $j\ne j'$ and $t_{0,j} = t_{0,j'}$, we have $t_{1,j}\ne t_{1,j'}$, otherwise $(j,j')\in G_{(\Delta)}$. This implies that $|S_1| = |T|-2$, and we may assume $T\setminus S_1 = \{1,x\}$ for some $x\ne 1$. Write $\mathbf{c}_j = (t_{0,j},t_{1,j})$ for $j\in [k]$, noting that
	\begin{equation*}
	\{\mathbf{c}_j:j\in [k]\} = T^2\setminus\{(1,1),(1,x)\}.
	\end{equation*}
	That is, $\{\mathbf{c}_j:j\in [k]\}$ is fixed by $\varphi_x$ setwise, with the componentwise action. This induces a permutation $\pi\in S_k$, where
	\begin{equation*}
	\mbox{$j^\pi = m$ if $\mathbf{c}_j^{\varphi_x} = \mathbf{c}_m$.}
	\end{equation*}
	In particular, $t_{i,j}^{\varphi_x} = t_{i,j^\pi}$ for each $i\in\{0,1\}$. Then
	\begin{equation*}
	D\mathbf{a}_i^{(\varphi_x,\dots,\varphi_x)\pi} = D(\varphi_{t_{i,1^{\pi^{-1}}}^{\varphi_x}},\dots,\varphi_{t_{i,k^{\pi^{-1}}}^{\varphi_x}}) = D(\varphi_{t_{i,1}},\dots,\varphi_{t_{i,k}}) = D\mathbf{a}_i
	\end{equation*}
	for each $i\in\{0,1\}$, and so $(\varphi_x,\dots,\varphi_x)\pi\in G_{(\Delta)}$.
	
	To complete the proof, we may assume $|S_0| = |S_1| = |T|-2$, say $T\setminus S_0 = \{1,x\}$ and $T\setminus S_1 = \{1,y\}$. Write $\mathbf{c}_j = (t_{0,j},t_{1,j})$ for $j\in [k]$ as above, and observe that
	\begin{equation*}
	T^2\setminus\{\mathbf{c}_j:j\in [k]\} = \{(1,1),(x,y)\}\text{ or }\{(1,y),(x,1)\}.
	\end{equation*}
	It is easy to check with the aid of {\sc Magma} that there exists an automorphism $\alpha\in\Aut(T)$ such that $1\ne\alpha\in C_{\Aut(T)}(x)\cap C_{\Aut(T)}(y)$, or $(x,y)^\a = (x^{-1},y^{-1})$.
	
	Assume $\a\ne 1$ and $(x,y)^\alpha = (x,y)$. Then $\{\mathbf{c}_j:j\in [k]\}$ is fixed by $\alpha$ setwise, with the componentwise action. Once again, $\alpha$ induces a permutation $\pi\in S_k$, where
	\begin{equation*}
	j^\pi = m\text{ if }\mathbf{c}_j^\alpha = \mathbf{c}_m.
	\end{equation*}
	Then by arguing as above, we deduce that $(\a,\dots,\a)\pi\in G_{(\Delta)}$.
	
	Finally, assume $(x,y)^\a = (x^{-1},y^{-1})$ and note that
	\begin{equation*}
	\{\mathbf{c}_j:j\in [k]\}^\alpha = \{(x^{-1},y^{-1})\mathbf{c}_j:j\in [k]\}.
	\end{equation*}
	Here $\alpha$ also induces a permutation $\pi\in S_k$, where
	\begin{equation*}
	j^\pi = m\text{ if } \mathbf{c}_j^\alpha = (x^{-1},y^{-1})\mathbf{c}_m,
	\end{equation*}
	and thus $t_{0,j}^\a = x^{-1}t_{0,j^\pi}$ and $t_{1,j}^\a = y^{-1}t_{1,j^\pi}$ for all $j\in[k]$, noting that $\pi\ne 1$ if $\a = 1$. Now we have
	\begin{equation*}
	D\mathbf{a}_0^{(\a,\dots,\a)\pi} = D(\varphi_{t_{i,1^{\pi^{-1}}}^\a},\dots,\varphi_{t_{i,k^{\pi^{-1}}}^\a}) = D(\varphi_{x^{-1}}\varphi_{t_{i,1}},\dots,\varphi_{x^{-1}}\varphi_{t_{i,k}}) = D\mathbf{a}_0
	\end{equation*}
	and similarly, $D\mathbf{a}_1^{(\a,\dots,\a)\pi} = D\mathbf{a}_1$. This completes the proof.
\end{proof}

\begin{prop}
	\label{p:|T|^2-2}
	If $P\in\{A_k,S_k\}$ and $k = |T|^2-2$, then
	\begin{equation*}
	b(G) = 
	\begin{cases}
	4 & \mbox{if $T\in\{A_5,A_6\}$ and $G = T^k.(\Out(T)\times S_k)$;}\\
	3 & \mbox{otherwise.}
	\end{cases}
	\end{equation*}
\end{prop}

\begin{proof}
	By Lemmas \ref{l:Ak_notle_G} and \ref{l:|T|2-2,A5A6}, we may assume that $S_k\leqs G$, and $G$ is not $T^k.(\Out(T)\times S_k)$ if $T\in\{A_5,A_6\}$. That is, $G = T^k.(O\times S_k)$ for some $O\leqs \Out(T)$, with $O\ne \Out(T)$ if $T\in\{A_5,A_6\}$. We will prove that $b(G) = 3$ by constructing a base of size $3$.
	
	Write $K = \Inn(T).O\leqs \Aut(T)$. Note that there exist $x,y\in T$ such that $C_K(x)\cap C_K(y) = 1$ and there is no $\a\in K$ with $(x,y)^\a = (x^{-1},y^{-1})$. This can be obtained by Corollary \ref{c:pair} when $T\notin\{A_5,A_6\}$, and the cases where $T\in\{A_5,A_6\}$ can be checked using {\sc Magma} (note that $K < \Aut(T)$ if $T\in\{A_5,A_6\}$). Now let $\mathcal{P} = \{\mathcal{P}_t:t\in T\}$ be a partition of $[k]$ with $|\mathcal{P}_1| = |\mathcal{P}_x| = |T|-1$, and $|\mathcal{P}_t| = |T|$ if $t\notin \{1,x\}$. And we label the elements in $T$ by $T = \{g_1,\dots,g_{|T|}\}$, where $g_1 = 1$ and $g_{|T|} = y$. Define $\mathbf{a}_0 = (\varphi_{t_{0,1}},\dots,\varphi_{t_{0,k}})\in\Inn(T)^k$, where $t_{0,j} = t$ if $j\in \mathcal{P}_t$, and define $\mathbf{a}_1 = (\varphi_{t_{1,1}},\dots,\varphi_{t_{1,k}})\in\Inn(T)^k$ by setting
	\begin{equation*}
	t_{1,j} = 
	\begin{cases}
	g_h & \mbox{if $t\ne 1$ and $j$ is the $h$-th smallest number in $\mathcal{P}_t$;}\\
	g_{h+1} & \mbox{if $j$ is the $h$-th smallest number in $\mathcal{P}_1$.}
	\end{cases}
	\end{equation*}
	Now we claim that $\Delta=\{D,D\mathbf{a}_0,D\mathbf{a}_1\}$ is a base for $G$.
	
	Suppose $(\a,\dots,\a)\pi\in G_{(\Delta)}$, noting that $\a\in K$. By Lemma \ref{l:l:3.4_diag_ext}(i), we have $\pi\in P_{\{\mathcal{P}\}}$, so either $\pi\in P_{\{\mathcal{P}_1\}}\cap P_{\{\mathcal{P}_x\}}$, or $\mathcal{P}_1^\pi = \mathcal{P}_x$, hence there are two cases to consider.
	
	First assume that $\mathcal{P}_1^\pi = \mathcal{P}_x$. There exists a unique $g\in T$ such that $t_{0,j}^\a = gt_{0,j^\pi}$ for all $j\in[k]$, and by taking $j\in\mathcal{P}_1$ we have $g = x^{-1}$. This implies that $x^\a = x^{-1}$ by taking $j\in\mathcal{P}_x$. Let $\mathcal{Q} = \{\mathcal{Q}_t:t\in T\}$ be the partition of $[k]$ defined by setting $j\in\mathcal{Q}_t$ if $t_{1,j} = t$. Then $|\mathcal{Q}_1|=|\mathcal{Q}_y| = |T|-1$, and $|\mathcal{Q}_t| = |T|$ if $t\notin \{1,y\}$. By arguing as above, either $\pi\in P_{\{\mathcal{Q}_1\}}\cap P_{\{\mathcal{Q}_y\}}$ or $\mathcal{Q}_1^\pi = \mathcal{Q}_y$. If the former holds, then
	\begin{equation*}
	(\mathcal{P}_1\cap \mathcal{Q}_1)^\pi = \mathcal{P}_x\cap \mathcal{Q}_1.
	\end{equation*}
	However, as can be seen from the definitions of $\mathbf{a}_0$ and $\mathbf{a}_1$, we have $|\mathcal{P}_1\cap \mathcal{Q}_1| = 0$, while $|\mathcal{P}_x\cap \mathcal{Q}_1| = 1$. This implies that $\mathcal{Q}_1^\pi = \mathcal{Q}_y$, so $y^\a = y^{-1}$ as above. By our assumptions on $x$ and $y$, there is no $\a\in K$ with $(x,y)^\a = (x^{-1},y^{-1})$, which gives a contradiction.
	
	Finally, suppose that $\pi\in  P_{\{\mathcal{P}_1\}}\cap P_{\{\mathcal{P}_x\}}$. First note that $t_{0,j}^\a = t_{0,j^\pi}$ for all $j\in [k]$, so $x^\a = x$. Similarly, we have $\pi\in P_{\{\mathcal{Q}_1\}}\cap P_{\{\mathcal{Q}_y\}}$ and $y^\a = y$. This implies that $\a\in C_K(x)\cap C_K(y) = 1$, and thus $t_{i,j} = t_{i,j^\pi}$ for all $i\in\{0,1\}$ and $j\in [k]$, which yields $\pi = 1$ and completes the proof.
\end{proof}

\begin{prop}
	\label{p:|T|^ell-2}
	If $\ell\geqs 3$, $k = |T|^\ell-2$ and $P\in\{A_k,S_k\}$, then $b(G) = \ell+1$.
\end{prop}

\begin{proof}
	In view of Theorem \ref{t:F}(iii), it suffices to construct a base for $G$ of size $\ell+1$. First note that there exist $x,y,z\in T$ such that
	\begin{equation*}
	C_{\Aut(T)}(x)\cap C_{\Aut(T)}(y)\cap C_{\Aut(T)}(z) = 1
	\end{equation*}
	and there is no $\a\in\Aut(T)$ with
	\begin{equation*}
	(x,y,z)^\a = (x^{-1},y^{-1},z^{-1}).
	\end{equation*}
	To see this, if $T\notin\{A_5,A_6\}$ then we apply Corollary \ref{c:pair}, and if $T\in\{A_5,A_6\}$ then it can be checked using {\sc Magma}. Let $\mathcal{P} = \{\mathcal{P}_t:t\in T\}$ be a partition of $[k]$ with $|\mathcal{P}_1| = |\mathcal{P}_x| = |T|^{\ell-1}-1$ and $|\mathcal{P}_t| = |T|^{\ell-1}$ if $t\notin\{1,x\}$. Write $T^{\ell-1} = \{\mathbf{b}_1,\dots,\mathbf{b}_{|T|^{\ell-1}}\}$, where $\mathbf{b}_h = (a_{1,h},\dots,a_{\ell-1,h})$, and we may assume $\mathbf{b}_1 = (1,\dots,1)$ and $\mathbf{b}_{|T|^{\ell-1}} = (y,z,\dots,z)$. Now define $\mathbf{a}_i = (\varphi_{t_{i,1}},\dots,\varphi_{t_{i,k}})$ for $i\in\{0,\dots,\ell-1\}$, where $t_{0,j} = t$ if $j\in\mathcal{P}_t$, and if $i\geqs 1$,
	\begin{equation*}
	t_{i,j} = 
	\begin{cases}
	a_{i,h} & \mbox{if $t\ne 1$ and $j$ is the $h$-th smallest number in $\mathcal{P}_t$;}\\
	a_{i,h+1} & \mbox{if $j$ is the $h$-th smallest number in $\mathcal{P}_1$.}
	\end{cases}
	\end{equation*}
	We claim that $\Delta = \{D,D\mathbf{a}_0,D\mathbf{a}_1,\dots,D\mathbf{a}_{\ell-1}\}$ is a base for $G$.
	
	We argue as in the proof of Proposition \ref{p:|T|^2-2}. Suppose $(\a,\dots,\a)\pi\in G_{(\Delta)}$, noting that $\pi\in P_{\{\mathcal{P}\}}$ by Lemma \ref{l:l:3.4_diag_ext}(i). It follows that either $\pi\in P_{\{\mathcal{P}_1\}}\cap P_{\{\mathcal{P}_x\}}$ or $\mathcal{P}_1^\pi = \mathcal{P}_x$.
	
	First assume that $\mathcal{P}_1^\pi = \mathcal{P}_x$. Note that there exists a unique $g\in T$ such that $t_{0,j}^\a = gt_{0,j^\pi}$ for all $j\in[k]$. Now $g = x^{-1}$ by taking $j\in\mathcal{P}_1$, and thus $x^\a = x^{-1}$ by taking $j\in\mathcal{P}_x$. Let $\mathcal{Q} = \{\mathcal{Q}_t:t\in T\}$ be the partition of $[k]$ defined by setting $j\in\mathcal{Q}_t$ if $t_{1,j} = t$. Then $|\mathcal{Q}_1| = |\mathcal{Q}_y| = |T|^{\ell-1}-1$, and $|\mathcal{Q}_t| = |T|^{\ell-1}$ if $t\notin\{1,y\}$. By applying Lemma \ref{l:l:3.4_diag_ext}(i) again, we have either $\pi\in P_{\{\mathcal{Q}_1\}}\cap P_{\{\mathcal{Q}_y\}}$ or $\mathcal{Q}_1^\pi = \mathcal{Q}_y$. If $\pi\in P_{\{\mathcal{Q}_1\}}\cap P_{\{\mathcal{Q}_y\}}$, then
	\begin{equation*}
	(\mathcal{P}_1\cap \mathcal{Q}_1)^\pi = \mathcal{P}_x\cap \mathcal{Q}_1,
	\end{equation*}
	which is impossible since $|\mathcal{P}_1\cap \mathcal{Q}_1| = |T|^{\ell-2}-1$, while $|\mathcal{P}_x\cap \mathcal{Q}_1| = |T|^{\ell-2}$. Hence, we have $\mathcal{Q}_1^\pi = \mathcal{Q}_y$, and thus $y^\a = y^{-1}$ with the same argument as above. Now suppose $i\geqs 2$ and let $\mathcal{R} = \{\mathcal{R}_t:t\in T\}$ be the partition of $[k]$ defined by setting $j\in\mathcal{R}_t$ if $t_{i,j} =  t$. Then $|\mathcal{R}_1| = |\mathcal{R}_z| = |T|^{\ell-1}-1$, and $|\mathcal{R}_t| = |T|^{\ell-1}$ if $t\notin\{1,z\}$. By arguing as above, we have $z^\a = z^{-1}$. However, by our assumptions on $x$, $y$ and $z$, there is no automorphism of $T$ simultaneously inverting all three elements, which gives a contradiction.
	
	It follows that $\pi\in P_{\{\mathcal{P}_1\}}\cap P_{\{\mathcal{P}_x\}}$, and with the same reason, we have $\pi\in P_{\{\mathcal{Q}_1\}}$ and $\pi\in P_{\{\mathcal{R}_1\}}$. Hence, $t_{i,j}^\a = t_{i,j^\pi}$ for all $i\in\{0,\dots,\ell-1\}$ and $j\in[k]$. This implies that
	\begin{equation*}
	\a\in C_{\Aut(T)}(x)\cap C_{\Aut(T)}(y)\cap C_{\Aut(T)}(z),
	\end{equation*}
	so $\a = 1$. Moreover, note that if $j,j'\in\mathcal{P}_t$ for some $t\in T$ and $j\ne j'$, then there exists $i\in\{1,\dots,\ell-1\}$ such that $t_{i,j}\ne t_{i,j'}$. Hence, $\pi = 1$ and so $\Delta$ is a base for $G$.
\end{proof}

We conclude that the proof of Theorem \ref{thm:main} is complete by combining Theorem \ref{thm:b(G)=2} with Propositions \ref{p:k=2}, \ref{p:|T|^ell-3}, \ref{p:l=1}, \ref{p:|T|^ell-1_|T|^ell},  \ref{p:|T|^2-2} and \ref{p:|T|^ell-2}.

\section{Proofs of Theorems \ref{thm:Aut(T,S)_to_1} and \ref{thm:prob_Aut}}

\label{s:GRR}

In this final section, we will prove Theorems \ref{thm:Aut(T,S)_to_1} and \ref{thm:prob_Aut}. As introduced in Section \ref{s:intro}, let $\mathbb{Q}_k(T)$ be the probability that a random $k$-element subset of $T^\#$ has a non-trivial setwise stabiliser in $\Aut(T)$. That is,
\begin{equation*}
\mathbb{Q}_k(T) := \frac{|\{R\in\mathscr{S}_k(T):\Aut(T,R)\ne 1\}|}{|\mathscr{S}_k(T)|},
\end{equation*}
where $\mathscr{S}_k(T)$ is the set of $k$-subsets of $T^\#$ (we will simply write $\mathscr{S}_k$ if $T$ is clear from the context). Consider the diagonal type group $G = T^k.(\Out(T)\times S_k)\leqs\mathrm{Sym}(\Omega)$ and recall that
\begin{equation*}
\mathbb{P}_{k}(T):= \frac{|\{(t_1,\dots,t_{k-1})\in T^{k-1}:\{D,D(\varphi_{t_1},\dots,\varphi_{t_{k-1}},1)\}\mbox{ is a base for $G$}\}|}{|T|^{k-1}},
\end{equation*}
which is the probability that a random element in $\Omega$ is in a regular orbit of $G_D = D$.

The following is \cite[Theorem 1.5]{F_diag}.

\begin{thm}
	\label{t:F_t:1.5}
	Let $k\geqs 5$ and let $(T_n)$ be a sequence of non-abelian finite simple groups such that $|T_n|\to\infty$ as $n\to\infty$. Then $\mathbb{P}_k(T_n)\to 1$ as $n\to \infty$.
\end{thm}

\begin{lem}
	\label{l:PQ}
	For any $k\geqs 4$, we have $\mathbb{Q}_{k}(T)\leqs 1-\mathbb{P}_{k+1}(T)$.
\end{lem}

\begin{proof}
	First, by Lemma \ref{l:Sk_eq_Hol}, we have $\{D,D(\varphi_{t_1},\dots,\varphi_{t_k},1)\}$ is a base for $G$ if and only if $t_1,\dots,t_k\in T^\#$ are distinct and $\Hol(T,\{t_1,\dots,t_k,1\})=1$. The latter condition implies that $\Aut(T,\{t_1,\dots,t_k\}) = 1$, so
	\begin{equation*}
	\mathbb{P}_{k+1}(T)\leqs \frac{|\{(t_1,\dots,t_{k})\in (T^\#)^{k}:t_1,\dots,t_{k}\mbox{ are distinct and }\Aut(T,\{t_1,\dots,t_{k}\}) = 1\}|}{|T|^{k}}.
	\end{equation*}
	Note that the numerator of the expression on the right-hand side is
	\begin{equation*}
	k!\cdot |\{R\in\mathscr{S}_k:\Aut(T,R) = 1\}|.
	\end{equation*}
	Thus, we have
	\begin{equation*}
	\mathbb{P}_{k+1}(T)\leqs \frac{k!\cdot |\{R\in\mathscr{S}_k:\Aut(T,R) = 1\}|}{|T|^k}
	\end{equation*}
	and it suffices to show that
	\begin{equation*}
	|T|^k\geqs k!\cdot |\mathscr{S}_k|.
	\end{equation*}
	This is clear, as $|\mathscr{S}_k| = {|T|-1\choose k}$.
\end{proof}

Theorem \ref{thm:Aut(T,S)_to_1} now follows by combining Theorem \ref{t:F_t:1.5} and Lemma \ref{l:PQ}. Finally, we establish Theorem \ref{thm:prob_Aut}. Recall that $\mathscr{P}_k$ is the set of $k$-subsets of $T$, and
\begin{equation*}
\fix(\sigma,\mathscr{P}_k) = \{S\in\mathscr{P}_k:\sigma\in\Hol(T,S)\}
\end{equation*}
is the set of fixed points of $\sigma\in\Hol(T)$ on $\mathscr{P}_k$.

\begin{prop}
	\label{p:prob_Aut_strong}
	Let $m>0$ be a real number. Then $\mathbb{Q}_k(T) < 1/m$ if
	\begin{equation}
	\label{e:prob_ori_strong}
	{|T|\choose k} > m\sum_{\sigma\in\mathcal{R}}|\fix(\sigma,\mathscr{P}_k)|,
	\end{equation}
	where $\mathcal{R}$ is the set of elements of prime order in $\Hol(T)$.
\end{prop}

\begin{proof}
	As noted in Section \ref{ss:fix_subsets}, we have
	\begin{equation*}
	|\{S\in\mathscr{P}_k:\Hol(T,S) \ne 1\}|\leqs \sum_{\sigma\in\mathcal{R}}|\fix(\sigma,\mathscr{P}_k)|,
	\end{equation*}
	which implies that $\Hol(T)$ has
	\begin{equation*}
	r> \frac{m-1}{m|\Hol(T)|}{|T|\choose k}
	\end{equation*}
	regular orbits on $\mathscr{P}_k$. Then
	\begin{equation*}
	|\{R\in\mathscr{S}_k : \Hol(T,R) = 1\}| = r(|T|-k)|\Aut(T)|> \frac{(m-1)(|T|-k)}{m|T|}{|T|\choose k}
	\end{equation*}
	and thus
	\begin{equation*}
	\mathbb{Q}_k(T) = \frac{|\{R\in\mathscr{S}_k:\Aut(T,R) \ne 1\}|}{|\mathscr{S}_k|}< 1-\frac{(m-1)(|T|-k)}{m|T|}\cdot\frac{{|T|\choose k}}{{|T|-1\choose k}} = \frac{1}{m},
	\end{equation*}
	as desired.
\end{proof}

\begin{proof}[Proof of Theorem \ref{thm:prob_Aut}.]
	Note that if $T = A_5$, then $5\log|T| < k < |T|-5\log|T|$ implies that $k = 30$, in which case we can check the theorem using {\sc Magma}. Now assume $|T|\geqs 168$, so $5\log|T| < |T|/4$. It suffices to show that \eqref{e:prob_ori_strong} holds for $m = |T|$ and $5\log|T| < k \leqs |T|/2$, and we can do this by arguing as in the proof of Proposition \ref{p:log}. More precisely, if $|T|/4\leqs k\leqs |T|/2$ then \eqref{e:prob_ori_strong} holds for $m = |T|$ if
	\begin{equation*}
	2t_0^{|T|} > \sqrt{30}e^{\frac{1}{8}}|T|^{\frac{10}{3}},
	\end{equation*}
	where
	\begin{equation*}
	t_0 = 4\cdot 3^{-\frac{3}{4}}\cdot 2^{-\frac{1}{2}-\frac{1}{10}} = 1.1577....
	\end{equation*}
	This inequality is valid for all $|T|\geqs 168$. And if $k<|T|/4$ then \eqref{e:prob_ori_strong} holds for $m = |T|$ if $(5/3)^k > |T|^{10/3}$, which holds true for all $k> 5\log|T|$.
\end{proof}

\begin{rem}
	\label{r:t:GRR}
	By Proposition \ref{p:prob_Aut_strong}, we have $\mathbb{Q}_k(T) < 1/2$ if \eqref{e:prob_ori} holds. We refer the reader to the proofs in Section \ref{s:proof_b(G)=2} for a wider range of $k$ satisfying \eqref{e:prob_ori} for each class of simple groups. For example, the proof of Proposition \ref{p:alternating} shows that if $T = A_n$ and $n\geqs 7$ then \eqref{e:prob_ori} holds for all $n\leqs k\leqs 4\log|T|$, which implies that $\mathbb{Q}_k(T)<1/2$ for all $n\leqs k\leqs |T|-n$.
\end{rem}

\end{document}